\documentclass[11pt]{article}
\usepackage{amsfonts}
\usepackage{amsthm,amsmath,amssymb}
\usepackage{graphicx}
\usepackage{setspace}
\usepackage{amsfonts,amstext}
\usepackage{authblk}
\usepackage{natbib}
\usepackage{xcolor}
\usepackage{titletoc}
\usepackage{booktabs}
\usepackage{fullpage}
\usepackage{hyperref}

\makeatletter
\g@addto@macro \normalsize {%
	\setlength\abovedisplayskip{5pt plus 2pt minus 2pt}%
	\setlength\belowdisplayskip{5pt plus 2pt minus 2pt}%
}
\makeatother

\hypersetup{
colorlinks,
linkcolor={red!50!black},
citecolor={blue!50!black},
urlcolor={blue!80!black}
}
\newtheorem{lemma}{\textbf{Lemma}}[section]
\theoremstyle{definition}

\newtheorem{assumption}{\textbf{Assumption}}

\newtheorem{theorem}{\textbf{Theorem}}[section]

\newtheorem{remark}{Remark}[section]
\theoremstyle{plain}

\begin{document}

\title{Penalized Sieve GEL for Weighted Average Derivatives of Nonparametric
Quantile IV Regressions\thanks{%
We appreciate discussions with Roger Koenker about quantiles and other
interesting topics. Roger's creativity, curiosity and kindness have inspired
us all these years. We thank guest editors and anonymous referees for their patience and helpful comments. Any errors are the responsibility of the authors. First version: August 2017.}}
\author{Xiaohong Chen\thanks{%
Tel.: +1 203 432 5852; \textit{Email}: xiaohong.chen@yale.edu.} , Demian
Pouzo\thanks{%
Corresponding author. Tel.: +1 510 642 6709. \textit{Email}: dpouzo@econ.berkeley.edu.},
 and James
L. Powell \thanks{%
Tel.: +1 510 643 0709. \textit{Email}: powell@econ.berkeley.edu.}}

\bigskip
\maketitle

\begin{abstract}
\singlespacing This paper considers estimation and inference for a
weighted average derivative (WAD) of a nonparametric quantile instrumental
variables regression (NPQIV). NPQIV is a non-separable and
nonlinear ill-posed inverse problem, which might be why there is no
published work on the asymptotic properties of any estimator of its WAD.
We first characterize the semiparametric
efficiency bound for a WAD of a NPQIV, which, unfortunately, depends on an unknown
conditional derivative operator and hence an unknown degree of ill-posedness, making
it difficult to know if the information bound is singular or not. In either case,
we propose a penalized sieve generalized empirical likelihood (GEL) estimation
and inference procedure, which is based on the unconditional WAD moment
restriction and an increasing number of unconditional moments that are implied by the conditional NPQIV restriction, where the unknown quantile function is approximated by a penalized sieve. Under some regularity conditions, we show that the self-normalized penalized sieve GEL estimator of the WAD of a NPQIV is asymptotically standard normal. We also show that the quasi likelihood ratio statistic based on the
penalized sieve GEL criterion is asymptotically chi-square distributed
regardless of whether or not the information bound is singular.

\medskip

\smallskip

\noindent \emph{JEL Classification:} C14; C22

\medskip

\noindent \emph{Keywords:} Nonparametric quantile instrumental variables;
Weighted average derivatives; Penalized sieve generalized empirical
likelihood; Semiparametric efficiency; Chi-square inference.
\end{abstract}

\newpage 

\section{Introduction}

Since the seminal paper by \cite{koenker1978regression}, quantile
regressions and functionals of quantile regressions have been the subjects
of ever-expanding theoretical research and applications in economics,
statistics, biostatistics, finance, and many other science and social
science disciplines. See \cite{koenker2005quantile} and the forthcoming
Handbook of Quantile Regression (2017) for the latest theoretical advances and empirical applications.

The presence of endogenous regressors is common in many empirical
applications of structural models in economics and other social sciences.
The Nonparametric Quantile Instrumental Variable (NPQIV) regression, $%
E[1\{ Y \leq h_{0}(W) \} - \tau|X]=0$, was, to our knowledge, first proposed in
\cite{chernozhukov2005iv} and \cite{CIN2007instrumental}. This model is a
leading important example of nonlinear and non-separable ill-posed inverse
problems in econometrics, which has been an active research topic following the Nonparametric (mean) Instrumental Variables (NPIV) regression, $E[ Y  - h_{0}(W)|X]=0$, studied by \cite{NP2003instrumental}, \cite{hall2005nonparametric}, \cite{BCK2007}, \cite{CFR2007linear}, \cite{darolles2011nonparametric} and others. See, for example, \cite{horowitz2007nonparametric},
\cite{CP2009efficient,CP2012estimation,CP2015sieve}, \cite{gagliardini2012nonparametric}, \cite{CH2013quantilerev}, \cite{CCLN2014local} and others for recent work on the NPQIV and its
various extensions.

In this paper, we consider
estimation and inference for a Weighted
Average Derivative (WAD) functional of a NPQIV. For models without nonparametric
endogeneity, WAD functionals of nonparametric (conditional) mean regression, $E[ Y  - h_{0}(X)|X]=0$, and of quantile regression, $E[1\{ Y \leq h_{0}(X) \} - \tau|X]=0$, have been extensively studied in both statistics and
econometrics. In particular, under some mild regularity conditions, plug-in estimators for WADs of any nonparametric mean and quantile regressions can be shown to be semiparametrically efficient
and root-$n$ asymptotically normal (where $n$ is the sample size). See, for example, \cite{newey1993efficiency},
\cite{newey1994asymptotic}, \cite{NeweyPowell1999}, \cite%
{ackerberg2014asymptotic} and the references therein. Although unknown functions of endogenous
regressors occur frequently in empirical work, due to the ill-posed nature
of NPIV and NPQIV, there is not much research on WAD functionals of NPIV and
NPQIV yet. In fact, even for the simpler NPIV model that is a linear and
separable ill-posed inverse problem, it is still a difficult question whether a
linear functional of a NPIV could be estimated at the root-$n$ rate; see, e.g.,
\cite{SeveriniTripathi2012} and \cite{davezies2015existence}. Although \cite%
{ai2007estimation} provide low-level sufficient conditions for a root-$n$
consistent and asymptotically normal estimator of the WAD of the NPIV model, and \cite%
{ai2012semiparametric} provide a semiparametric efficient estimator of WAD
for that model, to our knowledge, there is no published work on semiparametric efficient estimation
of the WAD for the NPQIV model yet.

We first characterize the semiparametric efficiency bound for the WAD functional of a
NPQIV model. Unfortunately, the bound depends on an unknown conditional derivative
operator and hence an unknown degree of ill-posedness. Therefore, it is difficult to
know if the semiparametric information bound is singular or not. Further, even if a researcher
assumes that the information bound is non-singular and the WAD is root-$n$
consistently estimable, the results in \cite{ai2012semiparametric} and \cite%
{CS2015overidentification} show that a simple plug-in estimator of a WAD
might not be semiparametrically efficient. This is in contrast to the
results of \cite{newey1993efficiency} and \cite{ackerberg2014asymptotic}
who show that plug-in estimators of a WAD of a nonparametric mean and quantile
regression are semiparametrically efficient.

We then propose penalized sieve Generalized Empirical Likelihood
(GEL) estimation of the WAD for the NPQIV model, which is based on the unconditional
WAD moment restriction and an increasing number of unconditional moments implied by the conditional moment restriction of the NPQIV model, where the unknown quantile function is approximated by a flexible penalized sieve.
Under some regularity conditions, we show that the self-normalized penalized sieve GEL estimator of the WAD of a NPQIV is asymptotically standard normal. We also show that the Quasi Likelihood Ratio (QLR) statistic based on the penalized sieve GEL criterion is asymptotically chi-squared distributed
regardless of whether the information bound is singular or not; this can be
used to construct confidence sets for the WAD of NPQIV without the need to
estimate the variance nor the need to know the precise convergence rates of
the WAD estimator.

Our estimation procedure builds upon \cite{DIN03}, who approximate a
conditional moment restriction $E[\rho (Y,\theta_0 )|X]=0$ by an increasing
sequence of unconditional moment restrictions, and then consider estimation of the Euclidean parameter $\theta_0$ (of fixed and finite dimension) and specification
tests based on GEL (and related) procedures. For the same model $E[\rho (Y,\theta_0 )|X]=0$, \cite{kitamura2004empirical}
directly estimate the conditional moment restriction via kernel and then
apply a kernel-based conditional empirical likelihood (EL) to estimate $%
\theta_0$. However, the model considered in these papers does not contain any
unknown functions  (say $h()$) and the residuals $\rho (.,\theta)$ are assumed
to be twice continuously differentiable with respect to $\theta$ at $%
\theta_0 $. For the semiparametric conditional moment restriction $E[\rho (Y,\theta_0,
h_0 (\cdot) )|X]=0$ when the
unknown function $h(\cdot )$ could depend on an endogenous variable, \cite{otsu2011large} and \cite{tao2013empirical} consider a
sieve conditional EL extension of \cite%
{kitamura2004empirical}, and \cite{sueishi2017} provides a sieve unconditional GEL extension of \cite{DIN03}, where the unknown function $h(.)$ is approximated by a finite dimensional linear sieve (series) as in \cite{ai2003efficient}. However, like
\cite{ai2003efficient}, all these papers assume twice continuously differentiable residuals $%
\rho(.,\theta, h(.))$ with respect to $(\theta_0, h_0(.))$, and hence rule out
the NPQIV model.

\cite{ParenteSmith2011gel} study GEL properties for non-smooth residuals $%
g(.,.)$ in the unconditional moment models $E[g(Y,\theta _{0})]=0$, but
require the dimensions of both $g(.,.)$ and $\theta _{0}$ to be fixed and finite.
Finally, \cite{horowitz2007nonparametric}, \cite{gagliardini2012nonparametric}, \cite%
{CP2009efficient,CP2012estimation,CP2015sieve}, and \cite{CNS2015constrained}  do include the NPQIV model, but none of these papers addresses the issues of 
estimation and inference for the WAD of the NPQIV.

The rest of the paper is organized as follows. Section \ref{sec:model}
introduces notation and the model. Section \ref{sec:eff-bound} characterizes the
semiparametric efficiency bound for the WAD of the NPQIV model. Section \ref%
{sec:PSGEL} introduces a flexible penalized sieve GEL procedure. Section \ref%
{sec:conv-rate} derives the consistency and the convergence rates of the
penalized sieve GEL estimator for the NPQIV model. Section \ref{sec:ADT}
establishes the asymptotic distributions of the WAD estimator and of the QLR statistic based on penalized sieve GEL for the WAD of a NPQIV. Section \ref{sec:conclusion} concludes with a discussion
of extensions.

\section{Preliminaries and Notation}

\label{sec:model}

Let $Z\equiv (Y,W,X)$ be the observable data vector, where $Y$ is the outcome
variable, $W$ is the endogenous variable and $X$ is the instrumental variable (IV); we assume the observable data, $Z$, is distributed according to a probability distribution $\mathbf{P}$. In order to simplify the exposition, we restrict attention to real-valued continuous random variables, i.e., we assume $\mathbf{P}$ has a density $\mathbf{p}$ with support given by $\mathbb{Z}\equiv \mathbb{Y}\times \mathbb{W}\times
\mathbb{X}\subseteq \mathbb{R}^{3}$; extending our results to vector-valued endogenous and instrumental variables would be straightforward but cumbersome in terms of notation.

\textbf{Notation.} For any subset, $\mathbb{Z}$, of an Euclidean space let $\mathcal{P}(\mathbb{Z})$ be the class of Borel probability measures over $\mathbb{Z}$. For any $P\in \mathcal{P}(\mathbb{Z})$, we use $p$ to denote its probability density function (pdf) (with respect to Lebesgue (Leb) measure) and $supp(P)$ to denote its support. We also use $P_{X}$ ($p_{X}$) to denote the marginal probability (pdf) of a random variable $X$; and $P_{Y|X}$ ($p_{Y|X}$) to denote the conditional probability (pdf) of $Y$ given $X$. For expectation, we write $E_{Q}[.]$ to be explicit about the fact that $Q$ is the measure of integration; throughout we sometimes use $E[.]\equiv E_{\mathbf{P}}[.]$ when $\mathbf{P}$ is the true probability of the data. The term ``wpa1" stands for ``with probability approaching one (under $\mathbf{P}$)"; for any two real-valued sequences $(x_{n},y_{n})_{n}$ $x_{n} \precsim y_{n}$ denotes $x_{n} \leq C y_{n}$ for some $C$ finite and universal; $\succsim$ is defined analogously. For any $q \geq 1$, we use $L^{q}(Q)\equiv L^{q}(\mathbb{Z},Q)$ to denote the class of measurable functions $f : \mathbb{Z} \mapsto \mathbb{R}$ such that $||f||_{L^{q}(Q)}=\left( \int_{z\in \mathbb{Z}}
|f(z)|^{q}Q(dz)\right) ^{1/q} < \infty$; as usual $L^{\infty}(Leb)$ denotes the class of essentially bounded real-valued functions. We use $||.||_{e}$ to denote the Euclidean norm, $\mathbb{R}_{+}=[0,\infty )$ and $\mathbb{R}_{++}=(0,\infty )$.

For any subset $S$ of a vector space $(\mathbb{S},||.||_S)$, $lin \{S\}$ denotes the smallest
linear space containing $S$; for any subspace $A \subseteq \mathbb{S}$, $A^{\perp}$ denotes its orthogonal complement in $(\mathbb{S},||.||_S)$. For any linear operator, $M : (\mathbb{S}_1,||.||_1 ) \rightarrow (\mathbb{S}_2,||.||_2 )$, let $Kernel(M) \equiv \{ x\in \mathbb{S}_1 \colon M[x] = 0  \}$ and $Range(M) \equiv \{ y \in \mathbb{S}_2 \colon \exists x \in \mathbb{S}_1,~M[x] = y  \}$; it is bounded if and only if $\sup_{x\in \mathbb{S}_1 : ||x||_1=1} ||M[x]||_2 <\infty$. For any linear bounded operator $M$, $M^{+}$ denotes its generalized inverse; see, e.g., \cite{engl1996regularization}.

\subsection{The WAD of the NPQIV model}

Let $\mathbb{A}\equiv \mathbb{R}\times \mathbb{H}$, where
$\mathbb{H} = \{ h \in L^{2}(Leb) \colon h^{\prime}~exists~and~||h^{\prime}||_{L^{2}(Leb)} < \infty  \}$, i.e., $\mathbb{H}$ is a Sobolev space of order $1$, here $h^{\prime}$ should be viewed as a weak derivative of $h$ (see \cite{brezis2010functional}). We note that $\mathbb{H}$ is a Hilbert space under the norm $||h||_{\mathbb{H}}\equiv ||h||_{L^{2}(Leb)}+||h'||_{L^{2}(Leb)}$, and $\mathbb{A}$ is a Hilbert space under the norm $||(\theta, h)||_{\mathbb{A}}\equiv ||\theta||_e + ||h||_{\mathbb{H}}$. In this paper we measure convergence in $\mathbb{A}$ using another norm $||(\theta, h)||\equiv ||\theta||_e + ||h||$ for $||h||\leq ||h||_{\mathbb{H}}$ (such as $||h||=||h||_{L^{2}(Leb)}$). The parameter set is given by $\mathcal{A} \equiv \Theta \times \mathcal{H} \subseteq \mathbb{A}$, where $\Theta$ is bounded and convex  and $\mathcal{H}$ is a set that contains additional restrictions on $h \in \mathbb{H}$ which will be specified below. We assume that $\mathbf{P}$ is such that there exists a parameter $\alpha_{0} \equiv (\theta_{0},h_{0}) \in \mathcal{A}$ that satisfies
\begin{align}  \label{eqn:model1}
0 = & E_{\mathbf{P}}[1\{ Y \leq h_{0}(W) \} - \tau|X] \\
\theta_{0} = & E_{\mathbf{P}}[\mu(W)h^{\prime }_{0}(W)] \label{eqn:model2}
\end{align}
for $\tau \in (0,1)$, where $\mu$ is a nonnegative, continuously differentiable scalar function in $L^{\infty}(Leb) \cap \mathbb{H}$ and should be viewed as the weighting function of the average derivative, $\theta_{0}$, of $h_{0}$.

The following assumption ensures that the conditions above uniquely identify $\alpha_{0}$; it will be maintained throughout the paper and will not be explicitly referenced in the results below.

\begin{assumption}
\label{ass:ident} There is a unique $\alpha _{0}\in int(\mathcal{A})$ that
satisfies model (\ref{eqn:model1})-(\ref{eqn:model2}).
\end{assumption}

The interior assumption is needed only for the asymptotic distribution results in Section \ref%
{sec:ADT}. In cases where $\mathcal{H}$ has an empty interior, one can use
the concept of relative interior of $\mathcal{H}$.
This assumption is clearly high level. The goal of this paper is to
characterize the asymptotic behavior of a modified GEL estimator of $\alpha $%
, taking as given the identification part; for a discussion of primitive
conditions for Assumption \ref{ass:ident}, we refer the reader to \cite{CCLN2014local} and
references therein.

The following assumption
imposes additional restrictions over the primitives: $\mu $, $\mathbf{P}$ and $%
\alpha _{0}$.

\begin{assumption}
\label{ass:pdf0} (i) $\mathbf{P}$ has a continuously differentiable pdf, $\mathbf{p}$, such that: the marginal density $\mathbf{p}_{W}$ of $W$ is uniformly
bounded, zero at the boundary of the support and $\mathbf{p}^{\prime}_{W} \in L^{2}(Leb)$; the marginal density $\mathbf{p}_{X}$ of $X$ is uniformly bounded away from 0 on its support; $\sup_{y,w,x \in \mathbb{Z}} \mathbf{p}_{Y|WX}(y
\mid w, x) < \infty$, $\sup_{y,w,x \in \mathbb{Z}} \frac{d\mathbf{p}_{Y|WX}(y \mid w, x)}{dy} < \infty$%
;
(ii) $\mathcal{H}$ is convex and such that for all $h \in \mathcal{H}$, $\sup_{w \in \mathbb{W}} |\mu(w) h(w)| < \infty$;
(iii) $Var_{\mathbf{P}}(\mu(W)h^{\prime }(W)) > 0$ for all $h \in \mathcal{H}$ in a $||\cdot ||$-neighborhood of $h_{0}$.
\end{assumption}

Part (i) of this condition imposes differentiability and boundedness
restrictions on different elements of $\mathbf{p}$; part (ii) ensures that $\lim_{w \rightarrow \pm \infty} \mathbf{p}_{W}(w) \mu (w) h(w) = 0$ which allows for an alternative representation for $\theta_{0}$ using integration by parts (see expression \ref{eqn:int-parts}  below); part (iii) is a high level
assumption and essentially implies $Var_{\mathbf{P}}(\mu(W)h^{\prime }_{0}(W)) > 0$ as
well as continuity of $h \mapsto Var_{\mathbf{P}}(\mu(W)h^{\prime }(W))$.

\section{Efficiency Bound for $\theta_{0}$}

\label{sec:eff-bound}

By definition of $\mathbb{H}$, Assumption \ref{ass:pdf0} and integration
by parts, it follows that
\begin{align}\label{eqn:int-parts}
\theta_{0} = E[\mu(W)h_0^{\prime }(W)] = - \int \ell(w) h_0 (w) dw
\end{align}
where $$w \mapsto \ell(w) \equiv \mu^{\prime }(w) \mathbf{p}_{W}(w) + \mu(w)
\mathbf{p}_{W}^{\prime }(w).$$
For the derivations of the
efficiency bound, it is important to recall that $\ell$ depends on $p_{W}$,
so we sometimes use $\ell_{\mathbf{P}}$ to denote $\ell$. Finally, observe that under our assumptions over $\mu$ and $\mathbf{p}_{W}$, $\ell \in L^{2}(Leb)$.

The formal definition of the efficiency bound for the unknown parameter $\theta _{0}$ is given at the beginning of Appendix \ref{app:eff-bound}. Loosely speaking, the efficiency bound is a lower bound for the asymptotic variance of all locally regular and asymptotically linear estimators of $\theta_{0}$; see \cite{bickeletal1998efficient} for details and formal definitions. If it is infinite, then the parameter $\theta_{0}$ cannot be estimated at root-$n$ rate by these estimators. We now derive this bound. For this, we introduce some useful
notation. For any $(y,w,\alpha )\in \mathbb{Y}\times \mathbb{W}\times
\mathbb{A}$, let $$\rho (y,w,\alpha )\equiv \left(\rho_{1}(y,w,\alpha),\rho_{2}(y,w,h)\right)^{T} \equiv \left(\theta -\mu (w)h^{\prime}(w), 1\{y \leq h(w) \} - \tau \right)^{T}.$$
Let $\mathbf{T}:\mathbb{H}\rightarrow L^{2}(\mathbf{P}_{X})$ be given by
\begin{align*}
	\mathbf{T}[g](x)=\int	\mathbf{p}_{Y|WX}(h_{0}(w)|w,x)g(w)\mathbf{p}_{W|X}(w|x)dw
\end{align*}
for all $x\in \mathbb{X}$ and $%
g\in \mathbb{H}$. The fact that $\mathbf{T}$ maps into $L^{2}(\mathbf{P}_{X})$ follows from
Jensen inequality and the fact that $\sup_{w,x} \mathbf{p}_{YW\mid X}(h_{0}(w),w\mid
x)<\infty $ (see Assumption \ref{ass:pdf0}). Its adjoint operator is denoted
as $\mathbf{T}^{\ast }:L^{2}(\mathbf{P}_{X})\rightarrow L^{2}(\mathbf{P}_{W})$. Finally, let
\[
x\mapsto \Gamma (x)\equiv E[\rho _{1}(Y,W,\alpha _{0})\rho _{2}(Y,W,h_{0})|X=x]/(\tau (1-\tau ))
\]
and $z\mapsto \epsilon (z)\equiv \rho
_{1}(y,w,\alpha _{0})-\Gamma (x)\rho _{2}(y,w,h _{0})$. Then $E[\epsilon (Z)\rho _{2}(Y,W,h_{0})|X]=0$ and $E[\epsilon (Z)]=0$.

\begin{theorem}
\label{thm:eff-bound} Suppose Assumptions \ref{ass:ident} and \ref{ass:pdf0}
hold and $\ell \in Kernel(\mathbf{T})^{\perp}$. Then

\begin{enumerate}
\item The efficiency bound of $\theta_{0}$ is finite iff $\ell \in Range
(\mathbf{T}^{\ast})$.

\item If it is finite, its efficient variance $V_0$ is given by
\begin{align*}
V_0= ||\epsilon(\cdot) ||^{2}_{L^{2}(\mathbf{P})} +
\left \Vert \mathbf{T}(\mathbf{T}^{\ast}\mathbf{T})^{+} [ \ell - \mathbf{T}^{\ast}[\Gamma] ] \right
\Vert^{2}_{L^{2}(\mathbf{P})}.
\end{align*}
\end{enumerate}
\end{theorem}

\begin{proof}
	 	See Appendix \ref{app:eff-bound}.
	 \end{proof}

The first result in Theorem \ref{thm:eff-bound} is obtained following the
approach of \cite{bickeletal1998efficient}.  The condition $\ell \in Kernel(\mathbf{T})^{\perp}$ ensures that only the ``identified part" of $h_{0}$ --- that is, the part of $h_{0}$ that is orthogonal to the kernel of $\mathbf{T}$ ---  matters for computing the weighted average derivative; we refer the reader to Appendix \ref{app:eff-bound} and the paper by  \cite{SeveriniTripathi2012} for further discussion.

\cite{SeveriniTripathi2012} provides an analogous result to Theorem \ref{thm:eff-bound}(1) for linear functionals
in a nonparametric \emph{linear} IV regression model. Our condition $\ell
\in Range(\mathbf{T}^{\ast })$, is analogous to theirs, but with a subtle yet
important difference. In \cite{SeveriniTripathi2012}, the object that plays the role of $\ell $ does not depend on $\mathbf{P}$, whereas in our case it does. This
observation changes the nature of our condition vis-a-vis theirs, because,
in our setup, $\ell \in Range(\mathbf{T}^{\ast })$ implies a restriction on $\mathbf{P}$
since both quantities, $\ell $ and $\mathbf{T}$ depend on it.\footnote{%
It is worth pointing out that this restriction was not imposed as one of the
conditions that defined the model used to construct the tangent space; see
Appendix \ref{app:eff-bound} for a definition.} It is also important to note that, if $\mathbf{T}$ is compact, then the range of $%
\mathbf{T}^{\ast }$ is a strict subset of $L^{2}(\mathbf{P}_{W})$ so that $\ell \in
Range(\mathbf{T}^{\ast })$ may not hold. Hence, in this case the weighted average
derivative may not be root-n estimable, and, moreover, the condition that
determines the finiteness of the efficiency bound depends on unknown
quantities. This observation highlights a difference with the no-endogeneity
case, where the efficiency bound is always finite, provided that $\ell \in
L^{2}(\mathbf{P}_{W})$ (see \cite{newey1993efficiency}).

Another discrepancy between the no-endogeneity case and ours is that in the
former case the ``plug in" is always efficient (see \cite%
{newey1993efficiency}, \cite{newey1994asymptotic}) due to the fact that the
tangent space is the whole of $\{ f \in L^{2}(\mathbf{P}) \colon E[f] = 0 \}$. On
the other hand, for NPQIV \cite{CS2015overidentification} show that the
closure of the tangent space is the whole space iff the $Range(\mathbf{T})$ is dense
in $L^{2}(\mathbf{P}_{X})$, which in turn is equivalent to $Kernel(\mathbf{T}^{\ast}) = \{ 0\}$%
. This last condition is comparable to a completeness condition on the
conditional distribution of the exogenous variable given the endogenous
ones, which may or may not hold for a particular $\mathbf{P}$.\footnote{%
In the NPIV setting, $Kernel(\mathbf{T}^{\ast}) = \{ 0\}$ is equivalent to the pdf of
$X$ given $W$ satisfying a completeness condition.}

The second result in Theorem \ref{thm:eff-bound} follows from projecting the
influence function onto the closure of the tangent space (see \cite%
{bickeletal1998efficient} and \cite{VdV2000} and references therein). So as
to shed some light on the expression for the efficiency bound, we point out
that it corresponds to the efficiency bound of the semiparametric
sequential conditional moment model via the ``orthogonalized moments" approach in \cite%
{ai2012semiparametric}.
In their notation, let $\varepsilon_{2}(z,\alpha) \equiv \rho_{2}(y,w,h)$ and
$\varepsilon_{1}(z,\alpha) \equiv \rho_{1}(y,w,\alpha) - \Gamma(x)\rho_{2}(y,w,h)$.
Note that $E[\varepsilon_{1}(Z,\alpha_{0})\varepsilon_{2}(Z,\alpha_{0}) \mid X] = 0$ (and $\varepsilon_{1}(z,\alpha_0 )=\epsilon (z)$). The model (\ref{eqn:model1})-(\ref{eqn:model2}) becomes equivalent to their orthogonalized moment model:
\begin{equation}\label{eqn:model-orthog}
E[\varepsilon_{2}(Z,\alpha_0)|X] = 0~,~~~E[\varepsilon_{1}(Z,\alpha_0)]=0.
\end{equation}
The expression in our Theorem \ref{thm:eff-bound}(2) coincides with their theorem 2.3 semiparametric efficient variance bound for $\theta_0$ of the model (\ref{eqn:model-orthog}). Also see proposition 3.3 in \cite{ai2012semiparametric} for the semiparametric efficient variance bound for the WAD of a NPIV model.

\section{The Penalized-Sieve-GEL Estimator}

\label{sec:PSGEL}

In this section we introduce our estimator for $\alpha_{0}\in \mathcal{A} \equiv \Theta
\times \mathcal{H} \subseteq \mathbb{A}\equiv \Theta \times \mathbb{H}$. In order to do
this, it will be useful to define some quantities. Given the i.i.d. sample $(Z_{i})_{i=1}^{n}$, let $P_{n}$ be the corresponding empirical probability. Let $(q_{k})_{k\in \mathbb{N}}$ be a complete basis in $L^{2}(\mathbb{X},Leb)$. For any $J\in \mathbb{N}$, let $q^{J}(x)=(q_{1}(x),...,q_{J}(x))^{T}$ be $J\times 1$ vector-valued function of $x$, and for any $(z,\alpha )\in \mathbb{Z}\times \mathcal{A}$, let
\[
g_{J}(z,\alpha )\equiv  \left (\rho _{1}(y,w,\alpha ),\rho _{2}(y,w,\alpha)q^{J}(x)^{T} \right)^{T}
=\left (\theta -\mu (w)h^{\prime}(w), [ 1\{y \leq h(w) \} - \tau ]q^{J}(x)^{T} \right)^{T}~.
\]
Let $\mathcal{S}\subseteq \mathbb{R}$ be an open interval that contains $0$. For any $P \in \mathcal{P}(\mathbb{Z})$,  any $\alpha \in \mathcal{A}$ and any $J
\in \mathbb{N}$, denote $\Lambda_{J}(\alpha,P)\equiv \cap_{z \in supp (P)} \{
\lambda \in \mathbb{R}^{J+1} \colon \lambda^{T}g_{J}(z,\alpha) \in \mathcal{S}
\}$, and $\hat{\Lambda}_{J}(\alpha)\equiv \Lambda_{J}(\alpha,P_{n})$.

Let $s:\mathcal{S}\rightarrow \mathbb{R}$ be strictly
concave, twice-continuously differentiable with Lipschitz continuous second
derivative; and $s^{\prime }(0)=s^{\prime \prime }(0)=-1$; see, e.g., \cite{Smith1997} and \cite{DIN03} for examples of such $s(.)$ functions. For any $\lambda \in \Lambda_{J}(\alpha,P)$, let
\begin{align*}
S_{J}(\alpha,\lambda,P) \equiv E_{P}[s(\lambda^{T}g_{J}(Z,\alpha))] - s(0)~,~~~\hat{S}_{J}(\alpha,\lambda)\equiv S_{J}(\alpha,\lambda,P_{n}).
\end{align*}
If $\mathcal{A}$ were a finite-dimensional compact set with $\dim (\mathcal{A}) \leq J+1$, then $\alpha_0$ could be estimated by the GEL procedure: $\arg \min_{\alpha \in \mathcal{A}} \sup_{\lambda \in
\hat{\Lambda}_{J}(\alpha )}\hat{S}_{J}(\alpha ,\lambda )$ (see, e.g., \cite{DIN03}).

Due to the presence of the infinite-dimensional nuisance parameter $h_0 \in \mathcal{H}$ in the NPQIV model (\ref{eqn:model1}), the parameter space $\mathcal{A} \equiv \Theta
\times \mathcal{H}$ is  an infinite-dimensional function space that is typically non-compact subset in $(\mathbb{A}, ||.||)$ and hence the identifiable uniqueness condition needed for consistency in $||.||$-norm might fail; see, e.g., \cite{NP2003instrumental} and \cite{chen2007large}. The above GEL procedure needs to be regularized to regain consistency and/or to speed up rate of convergence in $||.||$-norm. 
To this end, we introduce a \emph{regularizing structure}, which, jointly with $(q_{k})_{k \in \mathbb{N}}$, consists of a sequence of sieve spaces $(\mathcal{A}_k \equiv \Theta
\times \mathcal{H}_{k})_{k\in \mathbb{N}}$ in $(\mathbb{A}, ||.||)$, and a sequence of penalties $(\gamma_{k}\times Pen (\cdot))_{k \in \mathbb{N}}$ with tuning parameters $\gamma_{k} \downarrow 0$ and a penalty function $Pen : \mathbb{A} \rightarrow \mathbb{R}_{+}$.

The \emph{Penalized-Sieve-GEL (PSGEL)} estimator
is defined as
\begin{equation*}
\hat{\alpha}_{L,n} \in \arg \min_{\alpha \in \mathcal{A}_{K}} \left[\sup_{\lambda \in
\hat{\Lambda}_{J}(\alpha )}\hat{S}_{J}(\alpha ,\lambda )+\gamma
_{K}Pen(\alpha )\right],
\end{equation*}%
for any $(L=(J,K),n)\in \mathbb{N}^{3}$. If the \textquotedblleft arg min" in the previous expression is empty, one can replace it by an approximate minimizer.

The following assumption imposes restrictions over the regularizing structure $\{(q_{k}, \mathcal{H}_{k},\gamma_{k}Pen)_{k \in \mathbb{N}}\}$. Let $(\varphi _{k})_{k\in \mathbb{N}}$ be a basis functions in $\mathbb{H}$, and $\nabla \varphi
^{K}=(\varphi _{1}^{\prime },...,\varphi _{K}^{\prime })^{T}$.

\begin{assumption}
\label{ass:reg} (i) $(q_{k})_{k\in \mathbb{N}}$ is a basis in $L^{2}(\mathbf{P}_{X})$, and $E[q^{J}(X)q^{J}(X)^{T}]=I$ for each finite $J$;
\newline (ii) For all $K $, $\mathcal{H}_{K}\subseteq lin\{\varphi _{1},...,\varphi _{K}\}$ is
closed and convex, and $\overline{\cup _{k}\mathcal{H}_{k}}\supseteq
\mathcal{H}$, i.e., for any $\alpha \in \mathcal{A} \equiv \Theta \times \mathcal{H}$ there is an $\Pi_K\alpha \in \mathcal{A}_K \equiv \Theta \times \mathcal{H}_K$ such that $||\Pi_K\alpha - \alpha ||=o(1)$; and for some finite $C\geq 1$,
$C^{-1}I\leq E\left[ \left( \varphi ^{K}(W)\right) \left(
\varphi ^{K}(W)\right) ^{T}+\left( \nabla \varphi ^{K}(W)\right) \left(
\nabla \varphi ^{K}(W)\right) ^{T}\right] \leq CI$;
\newline (iii) (a) $Pen : \mathbb{A} \rightarrow \mathbb{R}_{+}$
is lower semi-compact (in $||.||$), $|Pen (\Pi_K\alpha_0 ) - Pen (\alpha_0)|=O(1)$, $Pen (\alpha_0 )<\infty$, and $\gamma_{k} \downarrow 0$, and (b) there exists an $M<\infty
$ such that for any $m\geq M$, any $K$ and any $\alpha \in \mathcal{A}_{K}$,
if $Pen(\alpha )\leq m$ then $\sup_{w \in \mathbb{W}} |\mu(w) h^{\prime }(w)| \leq m$.
\end{assumption}

Condition (i) is mild (see \cite{DIN03} (DIN) and the discussion therein).
Condition (ii) essentially defines the sieve space. Part (a) of Condition (iii) is standard in ill-posed problems (see \cite{CP2012estimation}); Part (b) is not. If $\mathcal{H}_{K}$ is $||\cdot ||_{L^{\infty}(\mathbb{W},\mu)} $ bounded, then the condition is vacuous.
If this is not the case, then the condition requires $Pen$ to be ``stronger"
than the $||\cdot ||_{L^{\infty}(\mathbb{W},\mu)} $ norm. The need to bound $%
|| h^{\prime } ||_{L^{\infty}(\mathbb{W},\mu)} $ arises from the fact that, in many
instances, in the proofs we need to control $\rho(y,w,\alpha)$ uniformly on $%
(y,w)$ (e.g., see Lemma \ref{lem:Lambda-charac} in the Supplemental Material \ref{supp:conv-rate}). Additionally, in our setup, is useful to link $Pen$ to $||\cdot ||_{L^{\infty}(\mathbb{W},\mu)}$ because the structure of the problem implies a natural bound for $Pen(.)$ --- and thus, through Assumption \ref{ass:reg}(iii), a bound for $||\cdot ||_{L^{\infty}(\mathbb{W},\mu)}$ ---, as shown in the following lemma.

\begin{lemma}
\label{lem:Pen-bound} For any $L=(J,K) \in \mathbb{N}^{2}$ and any $\alpha
\in \mathcal{A}_{K}$,
\begin{align*}
\gamma_{K} Pen(\hat{\alpha}_{L,n}) \leq \sup_{\lambda \in \hat{\Lambda}%
_{J}(\alpha)} \hat{S}_{J}(\alpha,\lambda) + \gamma_{K} Pen(\alpha)~~~wpa1.
\end{align*}
\end{lemma}

\begin{proof}
	See Appendix \ref{app:PSGEL}.
\end{proof}

The bound, however, may depend on $(J,K,n)$ and thus may affect the convergence
rate. Below, we will set $\alpha$ in the right-hand-side (RHS) to a particular value in $\mathcal{A}_{K}$ and use the resulting bound to construct what we call an ``effective sieve space".

\section{Consistency and Convergence Rates of the PSGEL Estimator}
\label{sec:conv-rate}

This section establishes the consistency and the rates of convergence of the PSGEL estimator $\hat{\alpha}_{L,n}$ to the true parameter $\alpha_0$ under a given norm $||.||$ over $\mathbb{A}$. In this and the next section, we note that the implicit constants inside the $O_{\mathbf{P}}$ do not depend on $(J,K,n)$.

\subsection{Effective sieve space}

Throughout the paper we use the following notation. Let $\overline{\theta} \equiv \sup_{ \theta \in \Theta} |\theta| <\infty$; and  $b_{\rho,J} \equiv (E[||q^{J}(X)||^{\rho}_{e}])^{1/\rho}$ for any $\rho > 0$. For any $L=(J,K) \in \mathbb{N}^{2}$, let
\begin{align*}
\Gamma_{L,n} \equiv \left\{ \frac{\bar{g}^{2}_{L,0}}{n} +
||E[g_{J}(Z,\Pi_K \alpha_{0})]||^{2}_{e} + \gamma_{K} Pen(\Pi_K \alpha_{0})
\right\}, ~~\bar{g}^{2}_{L,0} \equiv \overline{\theta} + ||\mu (\Pi_K h_{0})^{\prime
}||^{2}_{L^{2}(\mathbf{P})} + b_{2,J}^{2}.
\end{align*}
Let $(l_{n})_{n}$ be a slowly diverging positive sequence, e.g., $l_{n} =\log \log n$, which is introduced solely to avoid keeping track of constants. Finally we let
\begin{align*}
\bar{\mathcal{A}}_{L,n} \equiv \left\{ \alpha \in \mathcal{A}_{K} \colon
Pen(\alpha) \leq \mho_{L,n} \right\}~,~~~\text{where}~~\mho_{L,n} \equiv l_{n} \gamma^{-1}_{K} \Gamma_{L,n}~.
\end{align*}

The sequence of sets, $(\bar{\mathcal{A}}_{L,n})_{L,n}$, can be viewed as
the sequence of ``effective" sieve spaces, because, as the following lemma
shows, wpa1 the estimator (and, trivially, the sieve approximator $\Pi_K \alpha_{0}\in \mathcal{A}_{K}$) both
belong to it.

\begin{assumption}
\label{ass:rates-mild} (i) $b^{4}_{4,J}/n = o(1)$; (ii) $\delta_{n} = o(1)$, $\delta_{n} \times \mho_{L,n} = o(1)$, $ b_{\varrho,J}^{\varrho} n \delta^{\varrho}_{n}= o(1)$ for some $\varrho > 0$; (iii) $\sqrt{ \frac{\bar{g}^{2}_{L,0}}{n} + ||E[g_{J}(Z,\Pi_K \alpha_{0})]||^{2}_{e} }
= o(\delta_{n})$.
\end{assumption}

\begin{lemma}
	\label{lem:eff-sieve} Let Assumptions \ref{ass:ident}, \ref{ass:pdf0}, \ref{ass:reg} and \ref{ass:rates-mild} hold. Then, for any $L \in \mathbb{N}^{2}$, $\hat{\alpha}%
	_{L,n} \in \bar{\mathcal{A}}_{L,n}$ wpa1.
\end{lemma}

\begin{proof}
	See Appendix \ref{app:conv-rate}.
\end{proof}

The proof of this Lemma follows from Lemma \ref{lem:Pen-bound} with $\alpha
= \Pi_K \alpha_{0}$ and Lemma \ref{lem:SJ-bound} with  $\alpha = \Pi_K \alpha_{0}$ and $P=P_{n}$ in
Appendix \ref{app:conv-rate}. The latter lemma provides a bound for $\sup_{\lambda \in
	\hat{\Lambda}_{J}(\Pi_K \alpha_{0})} \hat{S}_{J}(\Pi_K \alpha_{0},\lambda)$ in terms
of $||n^{-1}\sum_{i=1}^{n} g_{J}(Z_{i},\Pi_K \alpha_{0}) ||^{2}_{e}$ and $%
\gamma_{K} Pen(\Pi_K \alpha_{0})$. With this in mind, the components of $%
\mho_{L,n}$ are intuitive: $\bar{g}^{2}_{L,0}/n$ is related to the
``variance" of $n^{-1}\sum_{i=1}^{n} g_{J}(Z_{i},\Pi_K \alpha_{0})$, where $\bar{%
	g}^{2}_{L,0}$ is a bound for $||g_{J}(.,\Pi_K \alpha_{0})||^{2}_{e}$. The term $%
||E[g_{J}(Z,\Pi_K \alpha_{0})]||_{e}$ is related to the ``bias" and reflects the
fact that $\Pi_K \alpha_{0}\in \mathcal{A}_K$ is a sieve approximate to $\alpha_0$.

\begin{remark}\label{rem:bound}
	As explained above, Lemma \ref{lem:eff-sieve} and Assumption \ref%
	{ass:reg}(iii) are used to ensure that $|| h^{\prime }||_{L^{\infty}(%
		\mathbb{W},\mu)}$ is bounded. If the construction of $\mathcal{H}_{K}$ directly
	implies $|| h^{\prime }||_{L^{\infty}(\mathbb{W},\mu)} \leq \mho$ for some
	fixed constant $\mho < \infty$, then $\mho$ should replace $\mho_{L,n}$ in
	the definition of $\bar{\mathcal{A}}_{L,n}$. This is applicable every time $\mho_{L,n}$ appears below. $\triangle$
\end{remark}

\subsection{Relation to Penalized Sieve GMM}

As expected, the asymptotic properties of the PSGEL estimator are closely related to an approximate minimizer of a GMM criterion associated to the following expression: for any $J \in \mathbb{N}$ and any $P \in \mathcal{P}(\mathbb{Z})$, let%
\begin{align*}
\alpha \mapsto Q_{J}(\alpha,P) \equiv
E_{P}[g_{J}(Z,\alpha)]^{T}H_{J}(\alpha_{0},\mathbf{P})^{-1}E_{P}[g_{J}(Z,\alpha)]
\end{align*}
where $(\alpha,P) \mapsto H_{J}(\alpha,P) \equiv E_{P}[g_{J}(Z,\alpha)
g_{J}(Z,\alpha)^{T}]$. That is, $Q_{J}(.,\mathbf{P})$ is the optimally weighted (population) GMM criterion
function associated with the vector of moments $E[g_{J}(Z,\cdot)]$.

For what follows, it will be useful to define the following intermediate
quantity which can be viewed as a (sequence) of pseudo-true parameters. For
each $L\equiv (J,K)\in \mathbb{N}^{2}$, let
\begin{equation*}
\alpha _{L,0}\equiv \arg \min_{\alpha \in \bar{\mathcal{A}}_{L,n}}Q_{J}(\alpha
,\mathbf{P}).
\end{equation*}%
We note that $\alpha _{0} \in \arg \min_{\alpha \in
\mathcal{A}}Q_{J}(\alpha ,\mathbf{P})$ for any $J$; but as we
restrict to the effective sieve space $\bar{\mathcal{A}}_{L,n}$, it could be
that $\alpha _{L,0}\neq \alpha _{0}$ for any $L\in \mathbb{N}^{2}$. The following lemma guarantees that $\alpha _{L,0}$ is in fact non-empty.

\begin{lemma}
\label{lem:alpha0L-exists} Let Assumptions \ref{ass:pdf0} and \ref{ass:reg} hold. Then, for each $L=(J,K) \in \mathbb{N}^{2}$, $%
\alpha_{L,0}$ is non-empty.
\end{lemma}

\begin{proof}
		See Appendix \ref{app:conv-rate}.
	\end{proof}

While this lemma shows that $\alpha _{L,0}$ is non-empty, it may not be a singleton. Nevertheless, for model (\ref{eqn:model1})-(\ref{eqn:model2}), it is easy to choose some finite-dimensional linear sieve $\mathcal{H}_{K}$ and some strict convex penalty $Pen$ such that $\alpha _{L,0}$ is in fact a singleton. Therefore the next assumption is effectively a way to suggest choices of a regularizing structure:

\begin{assumption}\label{ass:ID-pseudotrue}
	For any $L \in \mathbb{N}^{2}$, $\alpha _{L,0}$ is single-valued.
\end{assumption}

Let $m_{2}(X,\alpha) \equiv E_{\mathbf{P}}[\rho_{2}(Y,W,\alpha) \mid X]$. For each $J \in \mathbb{N}$, the $L^{2}(\mathbf{P})$ projection of $m_{2}(\cdot, \alpha)$ onto the linear span of $q^{J}(X)$ is denoted as $Proj_{J}[m_{2}(\cdot,\alpha)](X)$, where
\begin{align*}
Proj_{J}[m_{2}(\cdot,\alpha)](X) = & E_{\mathbf{P}} \left[ m_{2}(X,\alpha) q^{J}(X)^{T}  \right] (E_{\mathbf{P}}[q^{J}(X)q^{J}(X)^{T}])^{-1} q^{J}(X) \\
= &  E_{\mathbf{P}} \left[ \rho_{2}(Y,W,\alpha) q^{J}(X)^{T} \right] q^{J}(X)
\end{align*}
where $(E_{\mathbf{P}}[q^{J}(X)q^{J}(X)^{T}])^{-1} = I$ by Assumption \ref{ass:reg}.

The next lemma provides sufficient conditions that ensure convergence of $\alpha_{L,0}$ to the true parameter $\alpha_{0}$.

\begin{lemma}
\label{lem:alpha0L-consistent} Let Assumptions \ref{ass:ident}, \ref{ass:pdf0}, \ref{ass:reg} and \ref{ass:ID-pseudotrue} hold. Suppose $\lim_{n \rightarrow \infty} \sup_{ \alpha \in \bar{\mathcal{A}}_{L_{n},n}}||Proj_{J_{n}}[m_{2}(\cdot,\alpha)] - m_{2}(\cdot,\alpha)||_{L^{2}(\mathbf{P})} = 0  $. Then: $||\alpha_{L_n,0}-\alpha_0||=o(1)$.
\end{lemma}

\begin{proof}
		See Appendix \ref{app:conv-rate}.
	\end{proof}

\subsection{Convergence rates}

A crucial part of establishing the convergence rate of $\hat{\alpha}_{L,n}$ is to bound the rate of $||\hat{\alpha}_{L,n} - \alpha_{L,0}||$. For this it is important to quantify how well the population sieve GMM criterion function $Q_{J}$ separates points in $(\bar{\mathcal{A}}_{L,n},||.||)$ around $\alpha_{L,0}$. To do this, we define, for each, $(L,n)\in \mathbb{N}^{3}$, $\varpi_{L,n} : \mathbb{R}%
_{+} \rightarrow \mathbb{R}_{+}$ as
\begin{align}  \label{eqn:IU}
t \mapsto \varpi_{L,n}(t) \equiv \inf_{\{ \alpha \in \bar{\mathcal{A}}_{L,n}
	\colon ||\alpha - \alpha_{L,0} || \geq t \}} Q_{J}(\alpha,\mathbf{P}) -
Q_{J}(\alpha_{L,0},\mathbf{P}).
\end{align}

The function $%
\varpi_{L,n}$ is analogous to the one used in the standard identifiable
uniqueness condition (see \cite{WW1991}, \cite{newey1994large}). Within the
ill-posed inverse literature this function is akin to the notion of sieve
measure of ill-posedness used in \cite{BCK2007} and \cite{CP2012estimation,CP2015sieve}. The following lemma establishes some useful properties.

\begin{lemma}
	\label{lem:sieve-IU} Let Assumptions \ref{ass:pdf0}, \ref{ass:reg} and \ref{ass:ID-pseudotrue} hold. Then: for each $(L,n) \in \mathbb{N}^{3}$, $\varpi_{L,n}(t) = 0$
	iff $t = 0$ and $\varpi_{L,n}$ is continuous and non-decreasing in $t$.
\end{lemma}

\begin{proof}
	See Appendix \ref{app:conv-rate}.
\end{proof}

It is worth noting that even though $\varpi_{L,n}(t) > 0$ for all $t>0$,
it could happen that $\varpi_{L,n}(t) \rightarrow 0$ as $L$ diverges. This
behavior reflects the ill-posed nature of the problem.

We now present some high-level assumptions used to establish the convergence rate of the PSGEL estimator. The first of these assumptions introduces, and imposes restrictions on, a positive real-valued sequence $(\delta _{n})_{n\in \mathbb{N}}$ that is common in the GEL literature (see the Appendix in \cite{DIN03}). It ensures that the ball $\{ \lambda \in \mathbb{R}^{J+1} \colon ||\lambda||_{e}
\leq \delta_{n} \}$ belongs to $\hat{\Lambda}_{J}(\alpha)$ for any $\alpha
\in \bar{\mathcal{A}}_{L,n}$ (see Lemma \ref{lem:Lambda-charac} in the Supplemental Material \ref{supp:conv-rate}).  The assumption also restricts the rates of $(b_{\rho,J})_{\rho \in \mathbb{R},J\in \mathbb{N}},$ and the rate at which $L=(J,K) \in \mathbb{N}^{2}$ diverges relative to $n$:

\begin{assumption}
\label{ass:rates} (i) Assumption \ref{ass:rates-mild} holds; (ii)  $\delta_{n} l_{n} = o(1)$, $b^{3}_{3,J} \delta^{\varrho}_{n}= o(1)$ for some $\varrho > 0$, and $(\mho_{L,n})^{4} b^{4}_{4,J}/n = o(1)$.
\end{assumption}

Recall that the sequence $(l_{n})_{n}$ diverges
arbitrary slowly like $\log \log n$, and the bound $\mho_{L,n}$ is allowed to grow
(slowly) at the rate of $l_{n}$. Assumption \ref{ass:rates} slightly strengthens Assumption \ref{ass:rates-mild}.

The following assumption is a high-level condition that
controls the supremum of the process $f \mapsto \mathbb{G}_{n}[f] \equiv n^{-1/2} \sum_{i=1}^{n} \{f(Z_{i}) -E[f(Z_{i})]\}$ over  the classes $\bar{\mathcal{A}}_{L,n}$ and $\mathcal{G}_{L} \equiv \{ (y,w) \mapsto \rho_{2}(y,w,\alpha) \colon \alpha \in \bar{\mathcal{A}}_{L,n} \} $.

\begin{assumption}
\label{ass:Donsker} There exists a positive real-valued sequence, $%
(\Delta_{L,n})_{L,n \in \mathbb{N}^{3}}$, such that, for any $L \in \mathbb{N%
}^{2}$, $\sup_{(\theta,h) \in \bar{\mathcal{A}}_{L,n}} |\mathbb{G}_{n}[\mu \cdot
h^{\prime }]| =O_{\mathbf{P}}(\Delta_{L,n})$ and for all $1 \leq j \leq J$, $%
\sup_{g \in \mathcal{G}_{L}} |\mathbb{G}_{n}[g \cdot q_{j}]|
=O_{\mathbf{P}}(\Delta_{L,n})$.
\end{assumption}

For instance, if $\{ \mu \cdot h^{\prime }\colon h \in
\mathcal{H} \}$ and $\{ (y,w) \mapsto 1\{ y \leq h(w) \} \colon h \in
\mathcal{H} \}$ are P-Donsker, then $(\Delta_{L,n})_{L,n \in \mathbb{N}^{3}}$
is uniformly bounded.\footnote{Restrictions on the ``complexity" of these classes are implicit restrictions on the ``complexity" of $\mathcal{H}$; see \cite{chen2003estimation} and \cite{VdV2000}.}  But if this is not the case, then $(\Delta_{L,n})_{L,n
\in \mathbb{N}^{3}}$ may diverge as $L$ (or $n$) grows.

The next theorem establishes the convergence rate of the PSGEL estimator; in particular it establishes the rate for the estimator of the infinite dimensional component $h_0 \in \mathcal{H}$.

\begin{theorem}
	\label{thm:conv-rate} Suppose Assumptions \ref{ass:ident}, \ref{ass:pdf0}, %
	\ref{ass:reg}, \ref{ass:ID-pseudotrue} and \ref{ass:Donsker} hold. For any $(\delta_{n},l_{n})_{n}$
	satisfying Assumption \ref{ass:rates}, there exists a finite constant $M>0$ such that
	\begin{align*}
	||\hat{\alpha}_{L,n} - \alpha_{0}|| = O_{\mathbf{P}}\left( \varpi_{L,n}^{-1}\left( M (\delta_{1,L,n} + \delta_{2,L,n}) \right) \right) + ||\alpha_{L,0} -
	\alpha_{0} ||,
	\end{align*}
	where \begin{align*}
	\delta _{1,L,n}\equiv \sqrt{\frac{J}{n}}\times \Delta _{L,n}\times
	\left( \overline{\theta}+\mho _{L,n}+b_{2,J}\right)~,~~
	\delta _{2,L,n}\equiv  \mho _{L,n}^{2}\left\{ \delta _{n}+\delta
	_{n}^{-1}\Gamma _{L,n}\right\}
	\end{align*}
\end{theorem}

\begin{proof}
	See Appendix \ref{app:Thm-conv-rate}.
\end{proof}

The rate of convergence of the PSGEL estimator is composed of two standard
terms reflecting the ``approximation error" $||\alpha_{L,0} - \alpha_{0} ||$
and the ``sampling error" $\varpi_{L,n}^{-1}\left(M(\delta_{1,L,n} +
\delta_{2,L,n})\right)$. The component $\varpi^{-1}_{L,n}(.)$, reflects the ill-posed nature of the estimation
problem. As noted previously, even though, for a fixed $L$, $\varpi_{L,n}(t) >
0$ for $t>0$, this relationship can deteriorate as $L$ diverges, which
implies that $\varpi_{L,n}^{-1}(t)$ may diverge as $L$ diverges.

Below, we present an heuristic description of the proof that sheds light on the role of the sequences $(\delta_{1,L,n},\delta_{2,L,n})_{L,n}$ and of $\varpi_{L,n}$.

\subsection{Heuristics}

By the triangle inequality it suffices to bound the rate of $||\hat{\alpha}_{L,n} - \alpha_{L,0}||$. We do this by linking the PSGEL estimator to the population sieve GMM problem defined by $Q_{J}(\cdot,\mathbf{P})$. The first step to do this is to show that the PSGEL is an \emph{approximate} minimizer of the sample sieve GMM criterion $Q_{J}(.,P_{n})$ with the rate given by $\delta _{2,L,n}$.
\begin{lemma}
	\label{lem:QJ-approx-min} Let Assumptions \ref{ass:pdf0}
	and \ref{ass:reg} hold. For any $(\delta _{n},l_{n})_{n\in \mathbb{N}}$ satisfying Assumption \ref{ass:rates}, we have:
	\begin{equation*}
Q_{J}(\hat{\alpha}_{L,n},P_{n})=O_{\mathbf{P}}(\delta _{2,L,n})~,~~~\text{with}~~\delta _{2,L,n} = \mho _{L,n}^{2}\left\{ \delta _{n}+\delta_{n}^{-1}\Gamma _{L,n}\right\}~.
	\end{equation*}%
\end{lemma}

\begin{proof}
	See Appendix \ref{app:conv-rate}.
\end{proof}

The Lemma illustrates not only the role of $\delta_{2,L,n}$ but its nature. The two terms inside the curly brackets are completely analogous to those appearing in \cite{DIN03}. The
scaling by $\mho^{2}_{L,n}$ is not present in \cite{DIN03} and its
appearance here is due to the fact that the bound of $\rho(.,\alpha_{L,0})$
may depend, in principle, on $n$ and $L=(J,K)$. In \cite{DIN03}, on the
other hand, the upper bound $\mho_{L,n}$ can be taken to be a fixed constant
due to their Assumption 6.

Lemma \ref{lem:QJ-approx-min} implies that, for some finite $M$, the event $Q_{J}(\hat{\alpha}_{L,n},P_{n}) - Q_{J}(\alpha_{L,0},P_{n}) \leq M \delta_{2,L,n}$ occurs wpa1. The next step is to link the \emph{empirical} GMM criterion function, $Q_{J}(\cdot,P_{n})$, to its \emph{population} analog, $Q_{J}(\cdot,\mathbf{P})$ for which we can quantify its behavior (around $\alpha_{L,0}$) using $\varpi_{L,n}$. The next lemma provides such a link by showing that $Q_{J}(\cdot,P_{n})$ converges to its population analog.
\begin{lemma}
	\label{lem:QJ-univ-conv} Let Assumptions \ref{ass:pdf0}, \ref{ass:reg} and \ref{ass:Donsker} hold. Then: for any $L\equiv (J,K)\in
	\mathbb{N}^{2}$,
	\begin{equation*}
	\sup_{\alpha \in \bar{\mathcal{A}}_{L,n}}|Q_{J}(\alpha ,P_{n})-Q_{J}(\alpha,\mathbf{P})|=O_{\mathbf{P}}(\delta _{1,L,n})~,~~~\text{with}~~\delta _{1,L,n}=\sqrt{\frac{J}{n}}\times \Delta _{L,n}\times
	\left( \overline{\theta} +\mho _{L,n}+b_{2,J}\right).
	\end{equation*}%
	\end{lemma}

\begin{proof}
	See Appendix \ref{app:conv-rate}.
\end{proof}

The rate $(\delta_{1,L,n})_{L,n}$ has several components. The component $%
\sqrt{\frac{J}{n}}$ reflects the pointwise convergence rate of $||n^{-1}
\sum_{i=1}^{n} g_{J}(Z_{i},\alpha) - E[g_{J}(Z,\alpha)] ||_{e} $, while the
factor of $\Delta_{L,n}$ reflects the fact that we need \emph{uniform}
convergence of that term. Finally, the term $\left( \overline{\theta} +
\mho_{L,n} + b_{2,J} \right)$ is essentially the (uniform) bound for $\alpha
\mapsto ||n^{-1} \sum_{i=1}^{n} g_{J}(Z_{i},\alpha)||_{e}$ and $\alpha
\mapsto ||E[g_{J}(Z,\alpha)] ||_{e} $ over $\bar{\mathcal{A}}_{L,n}$.

With this result at hand and simple algebra, one can show that for some finite $M$ the set $A \equiv \{Q_{J}(\hat{\alpha}_{L,n},\mathbf{P}) - Q_{J}(\alpha_{L,0},\mathbf{P}) \leq M (\delta_{1,L,n} + \delta_{2,L,n})\}$ occurs wpa1. Therefore, by standard laws of probabilities, it follows that the probability of the set $||\hat{\alpha}_{L,n} - \alpha_{0}|| \geq M^{\prime} \varpi_{L,n}^{-1}\left( M (\delta_{1,L,n} + \delta_{2,L,n}) \right)$ (for any $M^{\prime}$) is --- up to a vanishing term --- less or equal than the probability of the intersection of the same set with $A$. Therefore, it only remains to show that the latter probability is naught for sufficiently large $M^{\prime}$. This follows because this latter probability is in turn bounded above by the probability of $\varpi_{L,n} \left( M^{\prime} \varpi_{L,n}^{-1}\left( M (\delta_{1,L,n} + \delta_{2,L,n}) \right) \right) \leq M (\delta_{1,L,n} + \delta_{2,L,n})$. By the fact that $\varpi_{L,n}$ is non-decreasing (see Lemma \ref{lem:sieve-IU}), this probability is naught by sufficiently large $M^{\prime}$, proving the result of Theorem \ref{thm:conv-rate}.

\subsection{Discussion of the elements in the Convergence Rate}
\label{sec:DiscussionRate}

We now present some observations regarding the main components of the convergence rate in Theorem \ref{thm:conv-rate}, namely, the rates $(\delta_{1,L,n},\delta_{2,L,n})_{L,n}$ and $\varpi_{L,n}$ defined in expression \ref{eqn:IU}. Regarding the latter, we first need to specify the norm $||.||$. We start by taking $(\varphi_{k})_{k \in \mathbb{N}}$ to be an orthogonal basis with respect to the Lebesgue measure over $\mathbb{H}$. Thus, for any $\alpha = (\theta,h) \in \mathbb{A}$, there exists a real-valued sequence, $(\pi_{l})_{l=0}^{\infty}$, such that   $\alpha = (\theta,h) = (\pi_{0} \bar{\varphi}_{0}, \sum_{l=1}^{\infty} \pi_{l} \bar{\varphi}_{l})$ where $\bar{\varphi}_{0} = 1$ and $\pi_{0} = \theta$, and, for any $k \geq 1$, $\bar{\varphi}_{k} = \varphi_{k}$ and $\pi_{k}$ is the ``Fourier" coefficient of $h$ with respect the basis $(\varphi_{k})_{k \in \mathbb{N}}$. This representation gives rise to the following norm over $\mathbb{A}$, $\alpha \mapsto \sqrt{ \sum_{l=0}^{\infty} \pi^{2}_{l}  } = \sqrt{\theta^{2} + ||h||^{2}_{L^{2}(Leb)}}$.  The aforementioned norm presents itself as a ``natural" norm under which we can establish convergence rate and thus we set $||.||$ as this norm; our result can be extended to norms other than this by specifying how the desired norm relates to $\alpha \mapsto \sqrt{|\theta|^{2} + ||h||^{2}_{L^{2}(Leb)}}$.

We now shed light on the behavior of $\varpi_{L,n}$ under our choice of $||.||$. In particular, we will illustrate how this function is linked to the curvature of the criterion function $\alpha \mapsto \bar{Q}_{J}(\alpha,\mathbf{P})$. To do this, it is convenient to use local approximations, so we take, for each $L=(K,L) \in \mathbb{N}^{2}$, $\mathcal{A}_{K}$ to be convex, $\alpha_{L,0}$ to be such that $ARC_{K}(\alpha_{L,0}) \equiv \{ \alpha_{L,0} + t \zeta \colon \zeta \in \mathcal{A}_{K}\setminus\{ \alpha_{L,0} \} ~and~t\in [0,1]  \} \subseteq \mathcal{A}_{K}$,  and require that $Pen$ to be convex and twice continuously differentiable. By the mean value theorem and the fact that $\alpha_{L,0}$ is a minimizer --- and thus satisfies that $\frac{d\bar{Q}_{J}(\alpha_{L,0},\mathbf{P})}{d\alpha}[\cdot]=0$ ---, it follows that for any $\alpha \in \bar{\mathcal{A}}_{L,n}$,
\begin{align*}
\bar{Q}_{J}(\alpha,\mathbf{P}) - \bar{Q}_{J}(\alpha_{L,0},\mathbf{P}) \geq \frac{1}{2}  \inf_{\eta \in ARC_{K}(\alpha_{L,0})} \frac{d^{2} \bar{Q}_{J}(\eta ,\mathbf{P})}{d\alpha^{2}}[\alpha - \alpha_{L,0},\alpha - \alpha_{L,0}].
\end{align*}
By the sieve representation discussed above, the RHS in this expression can be cast as
\begin{align*}
\bar{Q}_{J}(\alpha,\mathbf{P}) - \bar{Q}_{J}(\alpha_{L,0},\mathbf{P}) \geq &  (\pi^{K+1} - \pi^{K+1}_{L,0}) \mathcal{I}_{L} (\pi^{K+1} - \pi^{K+1}_{L,0}) ^{T},
\end{align*}
where $\pi^{K+1}$ denotes the first $K+1$ coefficients of the representation of $\alpha$; $\pi^{K+1}_{L,0}$ is the same but for $\alpha_{L,0}$, and $\mathcal{I}_{L}$ is a $(K+1) \times (K+1)$ matrix where the $(i,j)$-th component is given by
\begin{align*}
	 \mathcal{I}_{L}[i,j] \equiv \frac{1}{2} \inf_{\eta \in ARC_{K}(\alpha_{L,0})} \frac{d^{2} \bar{Q}_{J}(\eta ,\mathbf{P})}{d\alpha^{2}}[\bar{\varphi}_{i},\bar{\varphi}_{j}].
\end{align*}
This result implies that $\varpi_{L,n}(t) \geq t^{2} e_{min}(\mathcal{I}_{L} )$ ($e_{min}(A)$ is the minimal eigenvalue of the matrix $A$). If $e_{min}(\mathcal{I}_{L})>0$, then
\begin{align*}
	||\hat{\alpha}_{L,n} - \alpha_{0} || = O_{\mathbf{P}} \left( (e_{min}(\mathcal{I}_{L} ))^{-1/2} \sqrt{\delta_{1,L,n} + \delta_{2,L,n}}  +  ||\alpha_{L,0} - \alpha_{0}|| \right).
\end{align*}
The scaling factor $(e_{min}(\mathcal{I}_{L} ))^{-1/2}$ summarizes the ill-posed nature of the problem, because, even though we require $e_{min}(\mathcal{I}_{L} )>0$ for \emph{each} $L$, we do not impose this restriction \emph{uniformly} on $L$, i.e., we allow that $e_{min}(\mathcal{I}_{L} ) \rightarrow 0$ as $L \rightarrow \infty$.\footnote{The condition $e_{min}(\mathcal{I}_{L} )>0$ for each $L$ essentially ensures that the ``identifiable uniqueness" condition \ref{eqn:IU} holds for each $L$; this requirement is common in the ill-posed inverse literature (e.g., \cite{chen2007large}).} The speed at which this occurs depends on the local curvature of $\bar{Q}_{J}(\cdot, \mathbf{P})$ (at $\alpha_{L,0}$) and the growth of $\mathcal{A}_{K}$; see \cite{BCK2007} and \cite{CP2012estimation} for a more thorough discussion.

We next discuss the rate components $(\delta_{1,L,n},\delta_{2,L,n})_{L,n}$. As mentioned in Remark \ref{rem:bound} above, if $\sup_{h \in \mathcal{H}} ||h^{\prime}||_{L^{\infty}(\mathbb{W},\mu)}$ is finite, then $\mho_{L,n}$ can be replaced by a fixed constant $\mho$; this fact and some algebra implies that $\delta_{2,L,n} = O(\delta_{n} + \delta^{-1}_{n} (b^{2}_{2,J}/n + ||E[g_{J}(Z,\Pi_{K}\alpha_{0})]||^{2}_{e} + \gamma_{K} Pen(\Pi_{K}\alpha_{0}) ) )$, where $\Pi_{K} \alpha_{0}$ is the projection of $\alpha_{0}$ onto $\mathcal{A}_{K}$ (see Lemma \ref{lem:bound-hLo} in the Supplemental Material). By taking $\delta_{n}$ to balance both terms, it follows $ \delta_{2,L,n} = O\left(\sqrt{b^{2}_{2,J}/n + \{||E[g_{J}(Z,\Pi_{K}\alpha_{0})]||^{2}_{e} + \gamma_{K} Pen(\Pi_{K}\alpha_{0})\}}\right)$. Ignoring the term inside the curly brackets, this is the same rate than the one obtained by DIN in Lemma A.14 (note that in their setup $b^{2}_{2,J} \leq J$); the additional term inside the curly brackets stems from the fact that our estimation problem needs to be regularized and consequently $\alpha_{L,0}$ is not the true parameter that nullifies the moments.

 The sequence $(\delta_{1,L,n})_{L,n}$ is somewhat more standard within the semi-/non-parametric literature, e.g. \cite{chen2007large}, and its components essentially impose restrictions on the ``complexity" of $\mathcal{H}$. For instance, if $\mathcal{A} = \Theta \times \mathcal{H}$ is such that the classes $\{ \rho_{2}(\cdot,\cdot,\alpha) \colon \alpha \in \mathcal{A} \}$ and $\{ \mu h ^{\prime} \colon h \in \mathcal{H}  \}$ are P-Donsker, then $\Delta_{L,n} = O(1)$ and $\delta_{1,L,n} = O \left(  \sqrt{ \frac{J}{n}}\times b_{2,J} \right)$.

To further simplify the expression, suppose $b^{\rho}_{\rho,J} \leq J^{\rho/2}$; cf. Assumption 2 in \cite{DIN03} (see that paper for details and further references). Thus, under these conditions, the result in Theorem \ref{thm:conv-rate} simplifies to
{\small{\begin{align*}
||\hat{\alpha}_{L,n} - \alpha_{0} || = O_{\mathbf{P}} \left( (e_{min}(\mathcal{I}_{L} ))^{-1/2} \left(  \frac{J}{\sqrt{n}}  + ||E[g_{J}(Z,\Pi_{K}\alpha_{0})]||_{e} + \sqrt{\gamma_{K} Pen(\Pi_{K}\alpha_{0})}  \right)^{1/2} +  ||\alpha_{L,0} - \alpha_{0}|| \right);
\end{align*} }}
a rate governed by the degree of ill-posedness, the number $J$ of moment functions, the number $K$ of series terms and the bias arising from $\alpha_{L,0}$.

\section{Asymptotic Distribution Theory}

\label{sec:ADT} We now define the LR-type test statistic for the null
hypothesis $\theta_{0} = \nu$. For any $\nu \in \Theta$ and any $(L,n) \equiv (J,K,n)
\in \mathbb{N}^{3}$, let
\begin{align*}
\hat{\mathcal{L}}_{L,n}(\nu) \equiv 2 \left\{ \inf_{\{\alpha \in \mathcal{A}%
_{K} \colon \theta = \nu\}}\left[ \sup_{\lambda \in \hat{\Lambda}_{J}(\alpha)}
\hat{S}_{J}(\alpha,\lambda) + \gamma_{K} Pen(\alpha)\right] - \inf_{\alpha \in
\mathcal{A}_{K}} \left[\sup_{\lambda \in \hat{\Lambda}_{J}(\alpha)} \hat{S}%
_{J}(\alpha,\lambda) + \gamma_{K} Pen(\alpha)\right] \right\}.
\end{align*}

The goal of this section is to show that this statistic is asymptotically
chi-square distributed with one degree of freedom. The proof of this result
relies on a local quadratic approximation of the criterion function $\hat{S}%
_{J}$ and a representation for the parameter of interest. To derive these
results, we define the following quantities: For any $\alpha \in \mathcal{A}%
_{K}$ and any $(\theta, \zeta ) \in \mathbb{A}$, let
\begin{align*}
G(\alpha)[(\theta,\zeta)] = \frac{dE[g_{J}(Z,\alpha)]}{d\theta}\theta + \frac{dE[g_{J}(Z,\alpha)]}{dh}[\zeta] = \left[
\begin{array}{c}
\theta \\
\mathbf{0}%
\end{array}
\right] + \left[
\begin{array}{c}
E[\ell(W) \zeta(W)] \\
E[\mathbf{p}_{Y|WX}(h(W)|W,X) \zeta(W) q^{J}(X) ]%
\end{array}
\right]
\end{align*}
where $\mathbf{0}$ is a $J \times 1$ vector of zeros. By assumption \ref{ass:pdf0} these quantities are well-defined.

For any $L\in \mathbb{N}^{2}$ and for any $(\theta, \zeta ) \in \mathbb{A}$, we define another norm over $\mathbb{A}$ as,
\begin{equation*}
||(\theta, \zeta )||_{w}^{2}\equiv (G(\alpha _{L,0})[(\theta, \zeta )])^{T}H_{L}^{-1}(G(\alpha
_{L,0})[(\theta, \zeta ) ]),
\end{equation*}%
where $H_{L}\equiv H_{J}(\alpha _{L,0},\mathbf{P})$. This norm acts as the
so-called \textquotedblleft weak norm" in \cite%
{ai2003efficient,ai2007estimation}.

\subsection{Alternative Representation for the Weighted Average Derivative}

Lemma \ref{lem:weak-norm} in Appendix %
\ref{app:ADT} shows that, over $lin\{\mathcal{A}_{K}\}$ for any $L=(J,K)\in
\mathbb{N}^{2}$, $||\alpha ||_{w}=0$ iff $\alpha =0$. This fact implies that
linear functionals are always bounded in the space $(lin\{\mathcal{A}%
_{K}\},||.||_{w})$. Since $\theta $ can be interpreted as a linear
functional of $\alpha $, the following representation for $\theta $ holds:
For all $L=(J,K)\in \mathbb{N}^{2}$, there exists a $v_{L,n}^{\ast }\in
\mathcal{A}_{K}$ such that for any $\alpha =(\theta ,h)\in \mathcal{A}_{K}$,
\begin{equation*}
\theta =\langle v_{L,n}^{\ast },\alpha \rangle _{w},~and~||v_{L,n}^{\ast
}||_{w}=\sup_{a=(\theta ,h)\in lin\{\mathcal{A}_{K}\}, a\neq 0}\frac{|\theta|}{||a||_{w}}.
\end{equation*}%

We note that, even though for each fixed $L \in \mathbb{N}^{2}$, $%
||v^{\ast}_{L,n}||_{w} < \infty$, this quantity may diverge as $L$ diverges
if $\theta$ is not root-n estimable. Hence, we scale $v^{\ast}_{L,n}$ by its
norm, and define $u^{\ast}_{L,n} \equiv
v^{\ast}_{L,n}/||v^{\ast}_{L,n}||_{w} $. Then
\[
\frac{\hat{\theta}_{L,n} - \theta_{L,0}}{||v^{\ast}_{L,n}||_{w}}=\langle u^{\ast}_{L,n} , \hat{\alpha}_{L,n} - \alpha_{L,0} \rangle_{w}~.
\]

\begin{remark}[On the relationship between the Riesz representer and the Efficiency bound]\label{rem:Riesz}
	The weak norm of the Riesz representer, $||v_{L,n}^{\ast
	}||_{w}$, is the efficiency bound of $\theta_{0}$ in a model with $J+1$ unconditional moments functions, $E_{P}[g_{J}(Z,\cdot )]$, and $K+1$ parameters (which define $\alpha \in \mathcal{A}_{K}$).\footnote{\cite{ai2003efficient,ai2012semiparametric} established this claim for a richer model with conditional moments and infinite dimensional parameters.} For a suitably chosen sequence $L\equiv L(n)$ that increases as $n$ does --- since $(q_{j})_{j}$ is dense in $L^{2}(\mathbb{X},Leb)$ --- one expects the sequence of unconditional moment functions to approximate the moments (\ref{eqn:model1})-(\ref{eqn:model2}) defining the model. Thus, by the results in \cite{chamberlain1987} (see also lemma 3.3, lemma 4.1 and appendix A.1 in \cite{CP2015sieve})  one expects $(||v_{L(n),n}^{\ast}||_{w})_{n}$ to converge to the efficiency bound presented in Theorem \ref{thm:eff-bound} \emph{provided it is finite}. If the efficiency bound is infinite, the sequence $(||v_{L(n),n}^{\ast}||_{w})_{n}$ will diverge; this fact reflects the non root-n estimability of the weighted average derivative within the original model (\ref{eqn:model1})-(\ref{eqn:model2}).  $\triangle$
\end{remark}

\subsection{The Asymptotic distributions of $\hat{\theta}_{L,n}$ and LR statistic}

For any positive real-valued sequences $(\eta_{L,n},\eta_{w,L,n})_{L,n \in
\mathbb{N}^{3}}$ (they will be restricted below) and any $(L,n) \in \mathbb{N}%
^{3}$, let
\begin{align*}
\mathcal{N}_{L,n} \equiv \{ \alpha \in \bar{\mathcal{A}}_{L,n} \colon
||\alpha - \alpha_{L,0} || \leq \eta_{L,n}~and~||\alpha - \alpha_{L,0}
||_{w} \leq \eta_{w,L,n} \}.
\end{align*}

In what follows, for any $(L,n) \in \mathbb{N}^{3}$, let $\hat{\alpha}%
^{\nu}_{L,n}$ be the argument that minimizes the restricted criterion
function, i.e., $\hat{\alpha}^{\nu}_{L,n} \in \arg\min_{\{\alpha \in
\mathcal{A}_{K} \colon \theta = \nu\}} \sup_{\lambda \in \hat{\Lambda}%
_{J}(\alpha)} \hat{S}_{J}(\alpha,\lambda) $. We impose the following
assumption that restricts the convergence rate of the unrestricted and
restricted PSGEL estimators.

\begin{assumption}
\label{ass:dev-N}
For any $L \in \mathbb{N}^{2}$ and $\alpha \in \{ \hat{\alpha}_{L,n} , \hat{%
\alpha}^{\nu}_{L,n} \}$, if $\nu=\theta_{0}$: (i) $\alpha \in int(\mathcal{N}%
_{L,n})$; (ii) $\gamma_{K}\sup_{t \colon |t| \leq l_{n} n^{-1/2}} |Pen(\alpha)
- Pen(\alpha + t u^{\ast}_{L,n})| = o_{\mathbf{P}}(n^{-1})$; (iii) There exists a $C< \infty$ such that for any $h \in \mathbb{H}$, $||h||_{L^{2}(Leb)}\leq C ||(0,h)||$.
\end{assumption}

Part (i) of this assumption ensure that both estimators --- the restricted
and unrestricted ones --- converge to $\alpha_{L,0}$ faster than $\eta_{L,n}$
and $\eta_{w,L,n}$ in the respective norms. One can use the results in
Section \ref{sec:conv-rate} to verify this assumption.\footnote{%
The results in Section \ref{sec:conv-rate} apply to the restricted
estimator, under the null, with minimal changes.} Part (ii) ensures that the
penalty term is negligible (see also \cite{CP2015sieve}). Finally part (iii) states a relationship between the norm $h \mapsto ||(0,h)||$ --- used in Section \ref{sec:conv-rate} --- and the $L^{2}(Leb)$ norm over $\mathbb{H}$.

In the following assumption we let $\bar{\mathcal{G}}_{L,n}\equiv \{f(.,\alpha)=\rho _{2}(.,.,\alpha
)-\rho _{2}(.,.,\alpha _{L,0})\colon \alpha \in \mathcal{N}_{L,n}\}$.

\begin{assumption}
\label{ass:Donsker-LQA} There exists positive sequence, $(\Delta
_{2,L,n})_{L,n\in \mathbb{N}^{3}}$, such that, for any $L=(J,K)\in \mathbb{N}%
^{2}$, $\sup_{\alpha=(\theta,h) \in \mathcal{N}_{L,n}}\mathbb{G}_{n}[\mu \cdot (h^{\prime
}-h_{L,0}^{\prime })]=O_{\mathbf{P}}(\Delta _{2,L,n})$ and for all $1\leq j\leq J$%
, $\sup_{f\in \bar{\mathcal{G}}_{L,n}}\mathbb{G}_{n}[f\cdot
q_{j}]=O_{\mathbf{P}}(\Delta _{2,L,n})$.
\end{assumption}

This is a high-level assumption that controls one of the terms in the
remainder of the quadratic approximation in Lemma \ref{lem:LQA} below. As $%
\mathcal{N}_{L,n}$ is shrinking, one would expect $\Delta _{2,L,n}=o(1)$; the
exact rate, however, depends on the complexity of $\bar{\mathcal{A}}_{L,n}$.

\begin{assumption}
\label{ass:HJ-sec} There exists a positive real-valued sequence, $%
(\Xi_{L,n})_{L,n \in \mathbb{N}^{3}}$, such that, for any $L=(J,K) \in
\mathbb{N}^{2}$, $\sup_{\alpha \in \mathcal{N}_{L,n}} \left \Vert
H_{J}(\alpha,P_{n}) - H_{J}(\alpha_{L,0},P_{n}) - \{ H_{J}(\alpha,\mathbf{P}) -
H_{J}(\alpha_{L,0},\mathbf{P}) \} \right \Vert_{e} = O_{\mathbf{P}}(\Xi_{L,n})$.%
\end{assumption}

This high-level assumption implies stochastic equi-continuity of the process
$H_{J}(\cdot ,P_{n})$, and it is used to control one of the terms in the
remainder of the quadratic approximation in Lemma \ref{lem:LQA} below.

The final two assumptions impose additional restrictions on $(\eta
_{L,n},\eta _{w,L,n})_{L,n\in \mathbb{N}^{3}}$, $(b_{\rho ,J})_{\rho \in
\mathbb{R},J\in \mathbb{N}}$, $(\delta _{n})_{n\in \mathbb{N}}$ and the rate
at which $L=(J,K)$ diverges relative to $n$.

\begin{assumption}
\label{ass:undersmooth} (i) $\frac{\sqrt{n}}{||v^{\ast}_{L,n}||_{w}}
||E[g_{J}(Z,\alpha_{L,0})]||_{e} = o(1)$; (ii) $\frac{\sqrt{n}}{%
||v^{\ast}_{L,n}||_{w}} |\theta_{L,0}-\theta_{0}|= o(1)$.
\end{assumption}

This assumption implies that the ``bias" terms arising from working with $%
\alpha_{L,0}$, as opposed to $\alpha_{0}$, are small relative to the rate we
are using to scale the leading term of the asymptotic expansions below $%
\frac{\sqrt{n}}{||v^{\ast}_{L,n}||_{w}}$. Similar assumptions have been
imposed in the literature, e.g. \cite{CP2015sieve} and reference therein.

\begin{assumption}
\label{ass:rates-LQA} (i) $n \delta^{3}_{n} ( \mho_{L,n} + b_{3,J} )^{3} =
o(1)$, $n\delta^{2}_{n} \left( \{ \theta_{L,0}^{2} + || h^{\prime
}_{L,0}||^{2}_{L^{\infty}(\mathbb{W},\mu)} \} \sqrt{\frac{b_{4,J}}{n}} + \mho_{L,n} \eta_{L,n} + \Xi_{L,n} \right) = o(1)$ and $n\delta_{n} \left( \sqrt{\frac{J}{n}}\Delta_{2,L,n} + \eta_{L,n}^{2} b_{2,J} \right) = o(1)$; (ii) $\sqrt{
\bar{g}^{2}_{L,0}/n + ||E[g_{J}(Z,\alpha_{L,0})]||^{2}_{e} } + \eta_{w,L,n}
=o(\delta_{n})$ and $\left( \sqrt{\frac{J}{n}}\Delta_{2,L,n} +
\eta_{L,n}^{2} b_{2,J} \right) = o(\delta_{n})$; (iii) There exists a $%
\varrho>0$ such that $||h^{\prime }_{L,0}||^{2+\varrho}_{L^{\infty}(\mathbb{W},\mu)}/n^{2+\varrho} = o(1)$ and $b^{2+\varrho}_{2+\varrho,J}/n^{2+%
\varrho} = o(1)$; (iv) $A_{L,0} \equiv E[\mathbf{p}_{Y|WX}(h_{L,0}(W)\mid W,X)
q^{J}(X)\varphi^{K}(W)^{T} ]$ has full rank $K$ and $n^{-1/2}
e_{min}(A_{L,0}^{T}A_{L,0})^{-1} = o(\eta_{L,n})$; (v) $||h^{\prime
}_{L,0}||^{2}_{L^{\infty}(\mathbb{W}, \mu)} b^{2}_{4,J}/\sqrt{n} = o(1)$.
\end{assumption}

Part (i) ensures that the remainder term for the asymptotic quadratic
representation of $\hat{S}_{J}$ is negligible (see Lemma \ref{lem:LQA}). The sequence $(\delta _{n})_{n}$ in part (ii) was discussed after Assumption \ref%
{ass:rates}. Part (iii) is used to show asymptotic normality of the leading
term in Lemma \ref{lem:QLR-A-rep} by means of a Lyapounov condition. Finally,
part (iv) ensures that the weak norm is proportional to the strong norm over
$\mathcal{A}_{K}$ (even though the constant of proportionality may vanish as
$L$ diverges) and that deviations of the form $\alpha
+l_{n}n^{-1/2}u_{L,n}^{\ast }$ stay in $\mathcal{N}_{L,n}$ (see Lemma \ref%
{lem:charac-N} in Appendix \ref{app:ADT}). These deviations play a crucial
role in the proof of Lemma \ref{lem:QLR-A-rep}.

\begin{remark}[The rate restrictions of Assumption \ref{ass:rates-LQA}]
	While parts (iii)-(v) are fairly easy to check and interpret, parts (i)-(ii) are not as easy. The goal of this remark is to illustrate the restrictions imposed by these parts on the different rates $(\delta_{n},\eta_{L_{n},n},\eta_{w,L_{n},n},\Delta_{2,L_{n},n},\Xi_{L_{n},n})_{n}$ where $(L_{n})_{n}$ is a diverging sequence in $\mathbb{N}^{2}$. To do this, we take as the point of departure the setting described in Section \ref{sec:DiscussionRate}, which allows us to simplify some expressions. Under this setup, part (i) imposes $\delta_{n} = o\left( n^{-1/3} J^{-1/2}_{n}  \right)$. Given this, the restrictions in parts (i)-(ii) imply that
	 $\eta_{L_{n},n} = O(\min\{ n^{-1/2}\delta^{-1/2}_{n}J^{-1/2}_{n}, n^{-1/6} J^{-3/4}_{n}\})$
 and $\eta_{w,L_{n},n} = o(n^{-1/3} J^{-1/2}_{n})$; we note that by imposing  a polynomial rate of decay, this condition rules out the so-called severely ill-posed case wherein the rate of for $(\eta_{L_{n},n})_{n}$ decays slower than polynomial order (see \cite{CP2012estimation} and references therein). Parts (i)-(ii) also imply that $\Delta_{2,L,n} = o( (\sqrt{nJ_{n}}\delta^{2}_{n})^{-1}  )$ and $\Xi_{L,n} = o(n^{-1} \delta_{n}^{-2})$; for the ``worst case" where $\delta_{n} = \left( n^{-1/3} J^{-1/2}_{n}  \right)/l_{n}$, it follows that $\Delta_{2,L,n} = O( J_{n} n^{-1/6}  )$ and $\Xi_{L,n} = O(n^{-1/3} J_{n} )$, but the restriction can be relaxed if $(\delta_{n})_{n}$ decays faster. Finally, parts (i)-(ii) impose restrictions on the growth of $(L_{n})_{n}$: $J_{n} = O(n^{-1/6})$ and $\sqrt{J_{n}||E[g_{J_{n}}(Z,\alpha_{L_{n},0})]||^{2}_{e}} = o(n^{-1/3})$. $\triangle$
\end{remark}

The following result characterizes the asymptotic distribution of the LR test statistic under the null. This characterization holds regardless of whether the parameter $\theta_{0}$ is root-$n$ estimable or not.

\begin{theorem}
\label{thm:QLR} Let Assumptions \ref{ass:ident}-\ref{ass:ID-pseudotrue}
and \ref{ass:dev-N}-\ref{ass:rates-LQA} hold. Then, under the null $\theta_{0} = \nu$,
\begin{align*}
\hat{\mathcal{L}}_{L,n}(\theta_{0}) \Rightarrow \chi^{2}_{1}.
\end{align*}
\end{theorem}

\begin{proof}
	See Appendix \ref{app:Thm-QLR}.
\end{proof}

This result extends those in \cite%
{ParenteSmith2011gel} to a non-parametric setup where the GEL is constructed
using an increasing number of moment conditions, and wherein the parameter
of interest may not be root-n estimable. Using a related estimator ---an
EL-based on conditional moments a la \cite{kitamura2004empirical} --- \cite%
{tao2013empirical} derived an analogous result but her assumptions rule out non-smooth residuals, relevant for the quantile IV model considered here.

As a by-product of the derivations used to prove Theorem \ref{thm:QLR}, an asymptotic linear representation for the estimator of the WAD is obtained.

\begin{theorem}
	\label{thm:normal-thetahat} Let Assumptions \ref{ass:ident}-\ref{ass:ID-pseudotrue}, \ref{ass:dev-N} (for $\hat{\alpha}_{L,n}$), \ref{ass:Donsker-LQA}, \ref{ass:HJ-sec} and \ref{ass:rates-LQA} hold. Then
	\begin{align*}
	\frac{\hat{\theta}_{L,n} - \theta_{L,0}}{||v^{\ast}_{L,n}||_{w}}=
	n^{-1} \sum_{i=1}^{n} (G(\alpha_{L,0})[u^{\ast}_{L,n}])^{T}H_{L}^{-1}
	g_{J}(Z_{i},\alpha_{L,0}) + o_{\mathbf{P}}(n^{-1/2}).
	\end{align*}
	Further, under Assumption \ref{ass:undersmooth}, we have
	\[
	\frac{\sqrt{n}(\hat{\theta}_{L,n} - \theta_{0})}{||v^{\ast}_{L,n}||_{w}}\Rightarrow N(0,1)~.
	\]
\end{theorem}


The proof is the same as that Lemma \ref{lem:ALR-thetahat} in Appendix \ref{app:ADT} so it is omitted. This result illustrates the role of $||v^{\ast}_{L,n}||_{w}$ as the appropriate scaling of our estimator. If the sequence $(||v^{\ast}_{L,n}||_{w})_{n}$ is uniformly bounded, then this theorem implies that $\hat{\theta}_{L,n}$ is $\sqrt{n}$ asymptotically Gaussian. On the other hand, if the sequence diverges, Gaussianity is still preserve but the rate is slower and given by $\sqrt{n}/||v^{\ast}_{L,n}||_{w}$.

\subsection{Heuristics}

The idea is to show that, asymptotically, $\hat{\mathcal{L}}_{L,n}$ is a quadratic form of Gaussian random variables. The first step is to provide a
quadratic approximation for the criterion function $\hat{S}_{J}(\alpha,\cdot)$ as a function of $\lambda $, as shown in the following lemma.
\begin{lemma}
\label{lem:LQA} Let Assumptions \ref{ass:ident}-\ref{ass:ID-pseudotrue}, \ref{ass:Donsker-LQA}, \ref{ass:HJ-sec} and \ref{ass:rates-LQA}(v) hold. Then uniformly over $(\alpha,\lambda) \in \mathcal{N}_{L,n} \times
\{ \lambda \in \mathbb{R}^{J+1} \colon ||\lambda||_{e} \leq \delta_{n} \}$,
for any $L=(J,K) \in \mathbb{N}^{2}$
\begin{align*}
\hat{S}_{J}(\alpha,\lambda) =& - \lambda^{T} \varDelta(\alpha) - \frac{1}{2}
\lambda^{T} H_{L} \lambda \\
& + O_{\mathbf{P}} \left( \delta^{3}_{n} ( \overline{\theta} + l_{n}\gamma_{K}^{-1}
\Gamma_{L,n} + b_{3,J} )^{3} \right) \\
& + O_{\mathbf{P}} \left( \delta^{2}_{n} \left( ( \overline{\theta} + ||h^{\prime
}_{L,0}||_{L^{\infty}(\mathbb{W},\mu)} )^{2} \sqrt{b_{4,J}/n} + \mho_{L,n}
\eta_{L,n} + \Xi_{L,n} \right) \right) \\
& + O_{\mathbf{P}} \left( \delta_{n} \left( \sqrt{\frac{J}{n}}\Delta_{2,L,n} +
\eta_{L,n}^{2} b_{2,J} \right) \right).
\end{align*}
where $\varDelta(\alpha) \equiv n^{-1} \sum_{i=1}^{n} g_{J}(Z_{i},\alpha_{L,0}) +
G(\alpha_{L,0})[\alpha - \alpha_{L,0}]$.
\end{lemma}

\begin{proof}
	See Appendix \ref{app:ADT}.
\end{proof}

The ``remainder" terms in the RHS (the $O_{\mathbf{P}}(.)$ terms) are fairly intuitive: the order $\delta _{n}^{3}$%
-term requires boundedness of the third derivative of $\hat{S}_{J}(\alpha
,\cdot )$; the $\delta _{n}^{2}$-term arises because the expansion yields a
quadratic term with $H_{J}(\alpha ,P_{n})$ as opposed to $H_{L}$; and the $%
\delta _{n}$-term is the error of approximating $n^{-1}%
\sum_{i=1}^{n}g_{J}(Z_{i},\alpha )$ with $\varDelta(\alpha )$. This last
part handles the non-smooth nature of the residuals $\rho _{2}$ by using $%
E[g_{J}(Z,\cdot )]$, which is a smooth function. Assumption \ref{ass:rates-LQA}(i) ensures that these `remainder" terms are in fact $o_{\mathbf{P}}(n^{-1})$. This fact, and the fact that $\hat{\Lambda}%
_{J}(\alpha)$ contains a $\delta_{n}$-ball (see Lemma \ref{lem:Lambda-charac}
in the Supplemental Material \ref{supp:conv-rate}), imply that the expression in the Lemma provides an asymptotic characterization for
$\sup_{\lambda \in \hat{\Lambda}_{J}(\alpha)} \hat{S}_{J}(\alpha,\lambda)$
in terms of $(\varDelta(\alpha))^{T}H^{-1}_{L}(\varDelta(\alpha))$, which is
a quadratic form in $\alpha$.

With this result at hand and Assumption \ref{ass:dev-N}, one can obtain lower and upper bounds for $\hat{\mathcal{L}}_{L,n}(\theta_{0})$ of the form,
\begin{align*}
\hat{\mathcal{L}}_{L,n}(\theta_{0})  \geq (\varDelta(\hat{\alpha}_{L,n}))^{T}H^{-1}_{L}(\varDelta(\hat{\alpha}_{L,n})) - (\varDelta(\hat{\alpha}_{L,n}) + t u^{\ast}_{L,n})^{T}H^{-1}_{L}(\varDelta(\hat{\alpha}_{L,n} + t u^{\ast}_{L,n})) + o_{\mathbf{P}}(1),
\end{align*}
for appropriately chosen $t \in \mathbb{R}$, and
\begin{align*}
\hat{\mathcal{L}}_{L,n}(\theta_{0})  \leq (\varDelta(\hat{\alpha}^{\theta_{0}}_{L,n}  + t u^{\ast}_{L,n} ))^{T}H^{-1}_{L}(\varDelta(\hat{\alpha}^{\theta_{0}}_{L,n}  + t u^{\ast}_{L,n})) - (\varDelta(\hat{\alpha}^{\theta_{0}}_{L,n})^{T}H^{-1}_{L}(\varDelta(\hat{\alpha}^{\theta_{0}}_{L,n})) + o_{\mathbf{P}}(1),
\end{align*}
for appropriately chosen $t \in \mathbb{R}$. Since $\alpha \mapsto \varDelta(\alpha)$ is an affine function, the RHS in the previous expression is fairly easy to characterize. The following lemma formalizes these steps (its proof presents the explicitly choice for $t$ in the previous two displays).
\begin{lemma}
\label{lem:QLR-A-rep} Let Assumptions \ref{ass:ident}-\ref{ass:ID-pseudotrue}, \ref{ass:dev-N}-\ref{ass:rates-LQA} hold. Then, under the
null $\nu = \theta_{0}$,
\begin{align*}
&\hat{\mathcal{L}}_{L,n}(\theta_{0}) - \left( n^{-1/2} \sum_{i=1}^{n}
(G(\alpha_{L,0})[u^{\ast}_{L,n}])^{T}H_{L}^{-1} g_{J}(Z_{i},\alpha_{L,0})
\right)^{2} \\
& \geq 2\sqrt{n}\frac{(\theta_{0} - \theta_{L,0})}{||v^{\ast}_{L,n}||_{w}}
\left( n^{-1/2} \sum_{i=1}^{n}
(G(\alpha_{L,0})[u^{\ast}_{L,n}])^{T}H_{L}^{-1} g_{J}(Z_{i},\alpha_{L,0})
\right)+ o_{\mathbf{P}}(1).
\end{align*}
and
\begin{align*}
&\hat{\mathcal{L}}_{L,n}(\theta_{0}) - \left( n^{-1/2} \sum_{i=1}^{n}
(G(\alpha_{L,0})[u^{\ast}_{L,n}])^{T}H_{L}^{-1} g_{J}(Z_{i},\alpha_{L,0})
\right)^{2} \\
& \leq 2\sqrt{n}\frac{(\theta_{0} - \theta_{L,0})}{||v^{\ast}_{L,n}||_{w}}
\left( n^{-1/2} \sum_{i=1}^{n}
(G(\alpha_{L,0})[u^{\ast}_{L,n}])^{T}H_{L}^{-1} g_{J}(Z_{i},\alpha_{L,0})
\right) \\
& + \left( \sqrt{n}\frac{(\theta_{0} - \theta_{L,0})}{||v^{\ast}_{L,n}||_{w}}%
\right)^{2} + o_{\mathbf{P}}(1).
\end{align*}
\end{lemma}

\begin{proof}
	See Appendix \ref{app:ADT}.
\end{proof}

This lemma shows the reason for Assumption \ref{ass:undersmooth} in our analysis,
as this assumption ensures that
\begin{align*}
\hat{\mathcal{L}}_{L,n}(\theta_{0}) = \left( n^{-1/2} \sum_{i=1}^{n}
(G(\alpha_{L,0})[u^{\ast}_{L,n}])^{T}H_{L}^{-1} g_{J}(Z_{i},\alpha_{L,0})
\right)^{2} + o_{\mathbf{P}}(1).
\end{align*}

Under mild assumptions and Assumption \ref{ass:undersmooth}, the object
inside the parenthesis is asymptotically Normal with mean 0 and variance 1.
Here we see the importance of the \textquotedblleft optimal weight", $%
H_{L}^{-1}$. If $H_{L}$ differed from $E[g_{J}(Z,\alpha
_{L,0})g_{J}(Z,\alpha _{L,0})^{T}]$, then the variance of the term inside
the parenthesis will not be equal to 1, and the test statistic will only be
proportional to a $\chi _{1}^{2}$ in the limit; see \cite{CP2015sieve} for a
more thorough discussion and results for this case.

\section{Conclusion}

\label{sec:conclusion}

Since the seminal work by Koenker and Bassett about 40 years ago (\cite{koenker1978regression}), quantile regression models have become ubiquitous in econometrics and statistics; see \cite{Koenker2018} for a recent survey. The original linear quantile regression model has been extended in several directions; in particular to the general non-parametric IV framework that allows for ``flexible functional forms" and endogeneity of the regressors. This type of model, while very general, presents technical challenges arising from the non-smooth nature of the criterion function as well as its ill-posedness. One goal of this paper is to shed some light on how the nonlinear ill-posedness of the non-parametric quantile IV (NPQIV) model affects not only the speed of convergence to the conditional quantile function but also the accuracy for estimating even simple linear functionals. For this, we derive the semiparametric efficiency bound for a particular linear functional of the NPQIV --- the weighted average derivative (WAD).

To estimate the parameters of interest --- the NPQIV function and its WAD ---  we propose a general penalized sieve GEL procedure based on the unconditional WAD moment restriction and an increasing number of
unconditional moments that are asymptotically equivalent to the conditional moment defining the NPQIV model (\ref{eqn:model1}). We show that the QLR statistic based on the penalized sieve GEL is asymptotically chi-square distributed regardless of whether or not the information bound of the WAD is singular. This result can be used to construct confidence sets for the WAD without the need to estimate the variance of the estimator of the WAD. We hope these results extend even further the scope of quantile regression models.

The penalized sieve GEL procedure is more generally applicable to any
semi/nonparametric conditional moment restrictions and unconditional moment
restrictions, say of the following form:
\begin{align}
E[\rho _{2}(Y,W;\theta _{02},h_{01}(\cdot ),...,h_{0q}(\cdot ))|X]& =0,\text{
\quad a.s.-}X\text{,}  \label{semi00} \\
E[\rho _{1}(Y,W;\theta _{01},\theta _{02},h_{01}(\cdot ),...,h_{0q}(\cdot
))]& =0\text{.}  \label{semi01}
\end{align}
Here $Y$ denotes dependent (or endogenous) variables, $X$ denotes
conditioning (or instrumental) variables and $W$ could be either endogenous or subset of $X$, $\theta =(\theta _{1}^{\prime
},\theta _{2}^{\prime })^{\prime }$ denotes a vector of finite dimensional
parameters, and $h(\cdot )=\left( h_{1}(\cdot ),...,h_{q}(\cdot )\right) $ a
$q\times 1$ vector of real-valued measurable functions of $Y$, $W$, $X$ and other
unknown parameters. The residual functions $\rho _{j}(y,w;\theta ,h(\cdot ))$%
, $j=1,2$, could be nonlinear, pointwise non-smooth with respect to $(\theta
,h)$. And some of the $\theta$ could have singular information bound. This
is a valuable alternative to classical semiparametric two-step GMM
when the second step finite dimensional parameter $\theta$ might not be root-%
$n$ estimable.

In an old unpublished  draft, \cite{CP2010} study the asymptotic properties of
another estimation procedure, optimally weighted penalized Sieve Minimum
Distance (SMD) based on orthogonalized residuals for model (\ref{semi00})-(%
\ref{semi01}). Under a set of regularity conditions, including the assumption that the WAD of a NPQIV has a positive information bound, \cite{CP2010} establish that their optimally weighted penalized SMD
estimator of the WAD is root-$n$ asymptotically normal and semiparametrically efficient. It would
be interesting to compare this paper's estimator against theirs, and
we leave this to future work.

\bibliographystyle{plainnat}
\bibliography{QAD-biblio}

\appendix

\begin{center}
	\huge{Appendix}
\end{center}

\section{Proof of Theorem \ref{thm:eff-bound}}

\label{app:eff-bound}

To show Theorem \ref{thm:eff-bound} we need some more detailed notation and
definitions. Let $\mathcal{M}$ be the set of Borel probability measures over $\mathbb{Z}$ such
that for each $P \in \mathcal{M}$: (1) there exists a $(\theta(P),h(P)) \in \mathcal{%
A}$ for which equations (\ref{eqn:model1})-(\ref{eqn:model2}) hold for $P$; (2) the conditions of the Theorem are satisfied for $P$.

Given a $Q \in \mathcal{M}$, we use $(\theta(Q),h(Q))$ to denote the
parameters that satisfy equation \ref{eqn:model1} and $\theta(Q) = -
E_{Q}[\ell_{Q}(W)h(Q)(W)]$. For the true $\mathbf{P}$, we simply use $%
(\theta_{0},h_{0}) = (\theta(\mathbf{P}),h(\mathbf{P}))$.

Henceforth, let $L^{2}_{0}(P) \equiv \{ g \in L^{2}(P) \colon E_{P}[g(Z)] =
0\}$. A curve in $\mathcal{M}$ at $P$ is a mapping $[0,1] \ni t \mapsto
P[t] \in \mathcal{M}$ such that there exists a $g \in L^{2}_{0}(P)$ such
that
\begin{align*}
\lim_{t \rightarrow 0} \int \left( \frac{\sqrt{P[t](dz)} - \sqrt{P(dz)}}{t}
- 0.5 g(z) \sqrt{P(dz)} \right)^{2} = 0.
\end{align*}
We call $g$ the tangent of the curve; we typically use $t \mapsto P[t](g)$
to denote a curve with tangent $g$. The set of tangents for all curves in $%
\mathcal{M}$ at $P$ is called the tangent set; the linear span of the set is
called the tangent space of $\mathcal{M}$ at $P\in \mathcal{M}$, and we
denote it as $\mathcal{T}$.

\bigskip

The efficiency bound of $\theta_{0}$ is defined as (e.g., see \cite{bickeletal1998efficient})
\begin{align*}
\mathcal{E}(\mathbf{P}) \equiv \sup_{g \in \mathcal{T}} \frac{|\dot{\theta}(\mathbf{P})[g]|}{||g||_{L^{2}(\mathbf{P})}}
\end{align*}
where $\dot{\theta}(\mathbf{P})$ is the G-derivative of $P \mapsto \theta(P)$ at $\mathbf{P}$, i.e.,
\begin{align*}
g \mapsto \dot{\theta}(\mathbf{P})[g] = \lim_{t \rightarrow 0} \frac{\theta(\mathbf{P}[t](g)) -\theta(\mathbf{P})}{t}.
\end{align*}

Henceforth, we use $\mathbf{T}_{P}$ to denote the operator $\mathbf{T}$ under the probability measure $P$; the notation $\mathbf{T}$ is reserved for $\mathbf{T}_{\mathbf{P}}$.

(1) From the expression for $\mathcal{E}(\mathbf{P})$, it follows that finiteness of the efficiency bound is equivalent to boundedness of the linear functional $\dot{\theta}(\mathbf{P})$. In order to show this, we note that, since $\ell_{P}  \in Kernel(\mathbf{T}_{P})^{\perp}$, it follows that, for any $P \in \mathcal{M}$,
\begin{align*}
	\theta(P) = - \int \ell_{P}(w) h_{id}(P)(w) dw
\end{align*}
where $h_{id}(P)$ is the ``identified part" of $h(P)$ under $\mathbf{T}_{P}$, i.e., $h_{id}(P)$ is such that $h(P) = h_{id}(P) + \nu$ where $h_{id}(P) \in Kernel(\mathbf{T}_{P})^{\perp}$ and $\nu \in Kernel(\mathbf{T}_{P})$. Thus, it is enough to characterize the G-derivative of the RHS, and we do it in the following lemma; for this, let $A_{\mathbf{P}}: \mathcal{T} \rightarrow L^{2}(\mathbf{P}_{X})$ be defined as
\begin{align*}
  g \mapsto A_{\mathbf{P}}[g](\cdot) \equiv \int \rho_{2}(y,w,h(\mathbf{P}))	g(y,w,x) \mathbf{P}_{YW \mid X}(dy,dw\mid \cdot ).
\end{align*}

\begin{lemma}
	\label{lem:diff-alpha} For any $g \in \mathcal{T}$,
	\begin{align*}
	\dot{\theta}(\mathbf{P})[g] = & \theta(g \cdot \mathbf{P}) - E_{\mathbf{P}}\left[ \ell_{\mathbf{P}}(W) \dot{h}_{id}(\mathbf{P})[g](W) \right] \\
	= & \langle \mu h^{\prime }(\mathbf{P}), g \rangle_{L^{2}(\mathbf{P})} - \langle \ell_{\mathbf{P}} ,
	\dot{h}_{id}(\mathbf{P})[g] \rangle_{L^{2}(\mathbf{P})}
	\end{align*}
	and $\dot{h}_{id}(\mathbf{P})[g] = (\mathbf{T}^{\ast}\mathbf{T})^{+}\mathbf{T}^{\ast} A_{\mathbf{P}}[g] $, where  $(\mathbf{T}^{\ast}\mathbf{T})^{+}$ be the generalized inverse of $\mathbf{T}^{\ast}\mathbf{T}$; for a definition see \cite{engl1996regularization} Ch 2.
\end{lemma}

\begin{proof}
	See Section \ref{supp:eff-supp}.
\end{proof}

\begin{remark}
	This lemma illustrates the role that the condition $\ell_{P}  \in Kernel(\mathbf{T}_{P})^{\perp}$ plays in our proof. The previous lemma uses the conditional moment \ref{eqn:model1} to characterize the G-derivative of $P \mapsto h(P)$, and this only allow us to characterize the G-derivative of $h_{id}$, since the part of $h$ in the Kernel of $\mathbf{T}$ vanishes. Under condition, $\ell_{P}  \in Kernel(\mathbf{T}_{P})^{\perp}$, however, this is enough for characterizing the G-derivative of $P \mapsto \theta(P)$. $\triangle$
\end{remark}

Therefore,
\begin{align*}
\mathcal{E}(\mathbf{P}) = \sup_{g \in \mathcal{T}} \frac{| \langle \mu h^{\prime}(\mathbf{P}) , g   \rangle_{L^{2}(\mathbf{P})} - \langle \ell_{\mathbf{P}} ,  \dot{h}(\mathbf{P})[g]  \rangle_{L^{2}(\mathbf{P})}  |}{||g||_{L^{2}(\mathbf{P})}},
\end{align*}
and
\begin{align*}
\sup_{g \in \mathcal{T}} \frac{| \langle \ell_{\mathbf{P}} ,  \dot{h}(\mathbf{P})[g]  \rangle_{L^{2}(\mathbf{P})}  |}{||g||_{L^{2}(\mathbf{P})}} = \sup_{g \in \mathcal{T}} \frac{| \langle \ell_{\mathbf{P}} ,  (\mathbf{T}^{\ast}\mathbf{T})^{+}\mathbf{T}^{\ast} A_{\mathbf{P}} [g]  \rangle_{L^{2}(\mathbf{P})}  |}{||g||_{L^{2}(\mathbf{P})}}.
\end{align*}	

We now show that, if $\ell_{\mathbf{P}} \in Range(\mathbf{T})$, then $\mathcal{E}(\mathbf{P}) <\infty $. By the triangle inequality, it suffices to show that  $\frac{| \langle \mu h^{\prime}(\mathbf{P}) , g   \rangle_{L^{2}(\mathbf{P})}   |}{||g||_{L^{2}(\mathbf{P})}} < \infty$ and $\sup_{g \in \mathcal{T}} \frac{| \langle \ell_{\mathbf{P}} ,  (\mathbf{T}^{\ast}\mathbf{T})^{+}\mathbf{T}^{\ast} A_{\mathbf{P}} [g]  \rangle_{L^{2}(\mathbf{P})}  |}{||g||_{L^{2}(\mathbf{P})}} <\infty$. The former  follows because $\mu$ is uniformly bounded and $h^{\prime}(\mathbf{P}) \in L^{2}(\mathbf{P})$. We now show that the latter holds. As $\ell_{\mathbf{P}} \in Range(\mathbf{T}^{\ast})$, then $(\mathbf{T}^{\ast}\mathbf{T})^{+}[\ell_{\mathbf{P}}]$ is well-defined. And thus
\begin{align*}
\langle \ell_{\mathbf{P}}, (\mathbf{T}^{\ast}\mathbf{T})^{+}\mathbf{T}^{\ast} A_{\mathbf{P}} [g]  \rangle_{L^{2}(\mathbf{P})}  = & \int \mathbf{T} (\mathbf{T}^{\ast}\mathbf{T})^{+}[\ell_{\mathbf{P}}](x) \rho_{2}(y,w,\alpha_{0}) g(y,w,x) \mathbf{P}(dy,dw,dx) \\
= &   \langle  \mathbf{T} (\mathbf{T}^{\ast}\mathbf{T})^{+}[\ell_{\mathbf{P}}] \cdot \rho_{2}  ,g  \rangle_{L^{2}(\mathbf{P})}.
\end{align*}

Also, $|| \mathbf{T} (\mathbf{T}^{\ast}\mathbf{T})^{+}[\ell_{\mathbf{P}}] \cdot \rho_{2}||_{L{2}(\mathbf{P})} \leq 2 || \mathbf{T} (\mathbf{T}^{\ast}\mathbf{T})^{+}[\ell_{\mathbf{P}}] ||_{L^{2}(\mathbf{P})} < \infty$ because $\mathbf{T}(\mathbf{T}^{\ast}\mathbf{T})^{+}$ is bounded and $||\ell_{\mathbf{P}}||_{L^{2}(\mathbf{P})} \precsim ||\ell_{\mathbf{P}}||_{L^{2}(Leb)}  < \infty $ under Assumption \ref{ass:pdf0}. Therefore $\sup_{g \in \mathcal{T}} \frac{| \langle \ell_{\mathbf{P}} ,  (\mathbf{T}^{\ast}\mathbf{T})^{+}\mathbf{T}^{\ast} A_{\mathbf{P}} [g]  \rangle_{L^{2}(\mathbf{P})}  |}{||g||_{L^{2}(\mathbf{P})}} < \infty$ when  $\ell_{\mathbf{P}} \in Range(\mathbf{T}^{\ast})$, as desired.

We now show that if $\ell_{\mathbf{P}} \notin Range(\mathbf{T}^{\ast})$ then $\mathcal{E}(\mathbf{P}) = \infty$.  To show this, observe that by the triangle inequality $\mathcal{E}(\mathbf{P}) \geq \frac{| \langle \mu h^{\prime}(\mathbf{P}) , g   \rangle_{L^{2}(\mathbf{P})} - \langle \ell_{\mathbf{P}} ,  \dot{h}_{id}(\mathbf{P})[g]  \rangle_{L^{2}(\mathbf{P})}  |}{||g||_{L^{2}(\mathbf{P})}} \geq \frac{|\langle \ell_{\mathbf{P}} ,  \dot{h}_{id}(\mathbf{P})[g]  \rangle_{L^{2}(\mathbf{P})}  |}{||g||_{L^{2}(\mathbf{P})}} - C$ for some constant $C < \infty$ and any $g \in \mathcal{T}$. Since $\ell_{\mathbf{P}} \notin Range(\mathbf{T}^{\ast})$, $g \mapsto \langle \ell_{\mathbf{P}} ,  (\mathbf{T}^{\ast}\mathbf{T})^{+}\mathbf{T}^{\ast} A_{\mathbf{P}} [g]  \rangle_{L^{2}(\mathbf{P})} $ is not bounded and thus
\begin{align*}
	\sup_{g \in \mathcal{T}} \frac{| \langle \ell_{\mathbf{P}} ,  (\mathbf{T}^{\ast}\mathbf{T})^{+}\mathbf{T}^{\ast} A_{\mathbf{P}} [g]  \rangle_{L^{2}(\mathbf{P})}  |}{||g||_{L^{2}(\mathbf{P})}} = \infty
\end{align*}
so the result follows by choosing the $g$ that achieves this supremum (or a subsequence which yields a value arbitrarly close to it).

\bigskip

(2)  To prove part (2), we assume that $\ell_{\mathbf{P}} \in Range(\mathbf{T}^{\ast})$. Let $\dot{h}^{\ast}_{id}(\mathbf{P}) : L^{2}(\mathbf{P}) \rightarrow \mathcal{T}^{\ast}$ be the adjoint of $\dot{h}_{id}(\mathbf{P})$ and is given by
\begin{align*}
g \mapsto \dot{h}^{\ast}_{id}(\mathbf{P})[g](y,w,x) = \mathbf{T} (\mathbf{T}^{\ast} \mathbf{T})^{+}[g](x) \rho_{2}(y,w,\alpha(\mathbf{P}))
\end{align*}
for any $(y,w,x) \in \mathbb{Z}$.

It is well known (\cite{VdV2000} p. 363) that the efficiency bound (when it exists) is the variance of the projection of the influence function onto the tangent space, i.e.,
\begin{align*}
\mathcal{E}(\mathbf{P}) = &  ||Proj_{\mathcal{T}} \left[ \mu h^{\prime}(\mathbf{P}) - \dot{h}^{\ast}_{id}(\mathbf{P})[\ell_{\mathbf{P}}]  \right]||_{L^{2}(\mathbf{P})} \\
= & ||Proj_{\mathcal{T}} \left[ \mu h^{\prime}(\mathbf{P}) - \mathbf{T} (\mathbf{T}^{\ast} \mathbf{T})^{+}[\ell_{\mathbf{P}}] \cdot  \rho_{2}  \right]||_{L^{2}(\mathbf{P})}
\end{align*}
where $Proj_{\mathcal{T}} : L^{2}(\mathbf{P})  \rightarrow \bar{\mathcal{T}}$ is the projection operator onto the closure of the tangent space. This operator is characterized in the following lemma:

\begin{lemma}
	\label{lem:proj-T} For any $f \in L^{2}(\mathbf{P})$,
	\begin{align*}
	(y,w,x) \mapsto Proj_{\mathcal{T}}[f](y,w,x) = & f - Proj_{L^{2}(\mathbf{P}_{X})}[f](x) \\
	& - \frac{\rho_{2}(y,w,\alpha(\mathbf{P}))}{\gamma(1-\gamma)} \cdot (I -
	\mathbf{T}(\mathbf{T}^{\ast}\mathbf{T})^{+} \mathbf{T}^{\ast})Proj_{L^{2}(\mathbf{P}_{X})}[\rho_{2} \cdot f](x).
	\end{align*}
\end{lemma}

\begin{proof}
	See Section \ref{supp:eff-supp}.
\end{proof}

Let $\mathbf{M} \equiv (I - \mathbf{T}(\mathbf{T}^{\ast}\mathbf{T})^{+}\mathbf{T}^{\ast})/(\tau(1-\tau))$ and $\Gamma(x) = E[\rho_{1}(Y,W,\alpha(\mathbf{P}))\rho_{2}(Y,W,\alpha(\mathbf{P}))\mid X=x]/(\tau(1-\tau))$, then we can write
\begin{align*}
Proj_{\mathcal{T}} [\mu h^{\prime}(\mathbf{P})](y,w,x) =& (\theta(\mathbf{P}) - \mu(w) h^{\prime}(\mathbf{P})(w)) - E[\theta(\mathbf{P}) - \mu(W)h^{\prime}(\mathbf{P})(W) \mid X=x] \\
& - \rho_{2}(y,w,\alpha(\mathbf{P}))\mathbf{M}[E[\mu(W) h^{\prime}(\mathbf{P})(W) \rho_{2}(Y,W,\alpha(\mathbf{P})) \mid \cdot]](x) \\
= & \rho_{1}(y,w,\alpha(\mathbf{P})) - E[\rho_{1}(Y,W,\alpha(\mathbf{P}))\mid X=x] \\
& -  \rho_{2}(y,w,\alpha(\mathbf{P}))\mathbf{M}[E[\rho_{2}(Y,W,\alpha(\mathbf{P}))\rho_{1}(Y,W,\alpha(\mathbf{P})) \mid \cdot ] ](x) \\
= & \rho_{1}(y,w,\alpha(\mathbf{P})) - E[\rho_{1}(Y,W,\alpha(\mathbf{P}))\mid X=x] -  \rho_{2}(y,w,\alpha(\mathbf{P})) \Gamma(x) \\
& - \rho_{2}(y,w,\alpha(\mathbf{P})) \mathbf{T}(\mathbf{T}^{\ast}\mathbf{T})^{+}\mathbf{T}^{\ast} [\Gamma](x)
\end{align*}
where the second line follows from the fact that $E[\rho_{2}(Y,W,\alpha(\mathbf{P}))\rho_{1}(Y,W,\alpha(\mathbf{P})) \mid X]$ equals\\ $E[\mu(W)h^{\prime}(\mathbf{P})(W)\rho_{1}(Y,W,\alpha(\mathbf{P})) \mid X]$.

In addition,
\begin{align*}
& Proj_{\mathcal{T}} [\mathbf{T}(\mathbf{T}^{\ast}\mathbf{T})^{+}[\ell_{\mathbf{P}}] \cdot \rho_{2}](y,w,x) \\
=& \mathbf{T}(\mathbf{T}^{\ast}\mathbf{T})^{+}[\ell_{\mathbf{P}}](x) \rho_{2}(y,w,\alpha(\mathbf{P})) \\
& - \rho_{2}(y,w,\alpha(\mathbf{P})) \mathbf{M} \left[ E[\mathbf{T}(\mathbf{T}^{\ast}\mathbf{T})^{+}[\ell_{\mathbf{P}}](\cdot)  \rho_{2}^{2} \mid \cdot ]  \right](x)\\
=& \mathbf{T}(\mathbf{T}^{\ast}\mathbf{T})^{+}[\ell_{\mathbf{P}}](x) \rho_{2}(y,w,\alpha(\mathbf{P})) - \rho_{2}(y,w,\alpha(\mathbf{P})) \mathbf{M} \mathbf{T}(\mathbf{T}^{\ast}\mathbf{T})^{+}[\ell_{\mathbf{P}}](x) \tau (1-\tau) \\	
= &  \mathbf{T}(\mathbf{T}^{\ast}\mathbf{T})^{+}[\ell_{\mathbf{P}}](x) \rho_{2}(y,w,\alpha(\mathbf{P})) - \rho_{2}(y,w,\alpha(\mathbf{P})) \mathbf{T}(\mathbf{T}^{\ast}\mathbf{T})^{+}[\ell_{\mathbf{P}}](x)  \\
& + \rho_{2}(y,w,\alpha(\mathbf{P})) \mathbf{T}(\mathbf{T}^{\ast}\mathbf{T})^{+}\mathbf{T}^{\ast} \mathbf{T}(\mathbf{T}^{\ast}\mathbf{T})^{+}[\ell_{\mathbf{P}}](x) \\
= & \mathbf{T}(\mathbf{T}^{\ast}\mathbf{T})^{+}[\ell_{\mathbf{P}}](x) \rho_{2}(y,w,\alpha(\mathbf{P}))
\end{align*}	
where the last line follows from properties of the generalized inverse, namely, $(\mathbf{T}^{\ast}\mathbf{T})^{+}\mathbf{T}^{\ast} \mathbf{T}(\mathbf{T}^{\ast}\mathbf{T})^{+}  =(\mathbf{T}^{\ast}\mathbf{T})^{+}$; see \cite{engl1996regularization}  Proposition 2.3.

Therefore, for all $(y,w,x) \in \mathbb{Z}$,
\begin{align*}
Proj_{\mathcal{T}} \left[ \mu h^{\prime}(\mathbf{P}) - \dot{h}_{id}^{\ast}(\mathbf{P})[\ell_{P}]  \right](y,w,x) = &\epsilon(y,w,x) - E[\epsilon(Y,W,x) \mid X=x] \\
& + \rho_{2}(y,w,\alpha(\mathbf{P}))\mathbf{T}(\mathbf{T}^{\ast}\mathbf{T})^{+}\left[ \ell_{\mathbf{P}} - \mathbf{T}^{\ast} \Gamma   \right](x).
\end{align*}	
where $(y,w,x) \mapsto \epsilon(y,w,x) \equiv  \rho_{1}(y,w,\alpha(\mathbf{P})) - \rho_{2}(y,w,h(\mathbf{P})) \Gamma(x)$. Note that $E[\epsilon(Y,W,x) \mid X=x] = E[\rho_{1}(Y,W,\alpha(\mathbf{P})) \mid X = x]$.

The proof concludes by showing that $\epsilon(Y,W,X) - E[\epsilon(Y,W,X) \mid X]$ is orthogonal to
\begin{align*}
\rho_{2}(Y,W,\alpha(\mathbf{P}))\mathbf{T}(\mathbf{T}^{\ast}\mathbf{T})^{+}\left[ \ell_{\mathbf{P}} - \mathbf{T}^{\ast} \Gamma   \right](X).
\end{align*}
This follows because, conditional on $X$, $\epsilon(Y,W,X) - E[\epsilon(Y,W,X) \mid X]$ is orthogonal to $\rho_{2}(Y,W,\alpha(\mathbf{P}))$ by construction. $\square$

\section{Appendix for Section \ref{sec:PSGEL}}
\label{app:PSGEL}

This Appendix contains the proofs of all the Lemmas presented in Section \ref{sec:PSGEL}.

\begin{proof}[\textbf{Proof of Lemma \ref{lem:Pen-bound}}]
	Note that for any $\alpha \in \mathcal{A}$, $\hat{\Lambda}_{J}(\alpha) \ni 0$ wpa1, hence $ \sup_{\lambda \in \hat{\Lambda}_{J}(\alpha)} \hat{S}_{J}(\alpha,\lambda) \geq 0$ wpa1; in particular this applies to $\alpha = \hat{\alpha}_{L,n}$. Therefore, for any $\alpha
\in \mathcal{A}_{K}$,
	\begin{align*}
	\gamma_{K} Pen(\hat{\alpha}_{L,n}) \leq \sup_{\lambda \in \hat{\Lambda}%
		_{J}(\hat{\alpha}_{L,n})} \hat{S}_{J}(\hat{\alpha}_{L,n},\lambda)  + \gamma_{K} Pen(\hat{\alpha}_{L,n})\leq \sup_{\lambda \in \hat{\Lambda}%
		_{J}(\alpha)} \hat{S}_{J}(\alpha,\lambda) + \gamma_{K} Pen(\alpha),
	\end{align*}
	wpa1, where the second inequality is due to the definition of the minimizer $\hat{\alpha}_{L,n} \in \mathcal{A}_{K}$.
\end{proof}

\section{Proof for Theorem \ref{thm:conv-rate}}
\label{app:Thm-conv-rate}

Consider an $L$ that satisfies the Assumptions of the theorem. By the triangle inequality, it suffices to show that for any $\epsilon>0$, there exists constants $M_{1},M>0$ and $N \in \mathbb{N}$ such that
\begin{align*}
\mathbf{P}\left( ||\hat{\alpha}_{L,n} - \alpha_{L,0} || \geq M \varpi_{L,n}^{-1}(M_{1} \bar{\delta}_{L,n})   \right) \leq \epsilon
\end{align*}
for all $n \geq N$, where $\bar{\delta}_{L,n} \equiv \delta_{1,L,n} + \delta_{2,L,n}$. Henceforth, let $A_{n}(M_{1},M) \equiv \{  ||\hat{\alpha}_{L,n} - \alpha_{L,0} || \geq M \varpi_{L,n}^{-1}(M_{1} \bar{\delta}_{L,n})    \}$.

From the proof of Lemma \ref{lem:eff-sieve}, $\gamma_{K}Pen(\hat{\alpha}_{L,n}) = O_{\mathbf{P}} ( \Gamma_{L,n} ) = o_{\mathbf{P}}( l_{n}(\delta_{1,L,n} + \delta_{2,L,n}))$. This fact,  Lemma \ref{lem:QJ-approx-min} and the fact that $Q_{J}(\alpha_{L,0},P_{n}) \geq 0$, imply that there exists an $M_{0}$ and an $N_{0}$ such that
\begin{align*}
\mathbf{P}\left( Q_{J}(\hat{\alpha}_{L,n},P_{n}) - Q_{J}(\alpha_{L,0},P_{n}) \geq M_{0} \delta_{2,L,n}   \right) \leq \epsilon
\end{align*}
for all $n \geq N_{0}$. This result and Lemma \ref{lem:QJ-univ-conv} in turn imply that
\begin{align*}
\mathbf{P}\left( Q_{J}(\hat{\alpha}_{L,n},\mathbf{P}) - Q_{J}(\alpha_{L,0},\mathbf{P}) \geq 2 M_{0} \{ \delta_{1,L,n} + \delta_{2,L,n}\}   \right) \leq \epsilon
\end{align*}
for all $n \geq N_{0}$. Let $B_{n} \equiv \{  Q_{J}(\hat{\alpha}_{L,n},\mathbf{P}) - Q_{J}(\alpha_{L,0},\mathbf{P}) \leq 2 M_{0} \{ \delta_{1,L,n} + \delta_{2,L,n}\}   \}$.

The previous display implies that for any $n \geq N_{0}$, $\mathbf{P}(A_{n}(M_{1},M)) \leq \mathbf{P}(A_{n}(M_{1},M) \cap B_{n}) + \epsilon$. By definition of $\varpi_{L,n}$, for any history of data $(Z_{i})_{i}$ in $A_{n}(M_{1},M) \cap B_{n}$ it follows that
\begin{align*}
\varpi_{L,n} \left( M \varpi_{L,n}^{-1}(M_{1} \bar{\delta}_{L,n}) \right) \leq 2 M_{0} \{ \delta_{1,L,n} + \delta_{2,L,n}\} \iff  M \varpi_{L,n}^{-1}(M_{1} \bar{\delta}_{L,n}) \leq  \varpi_{L,n}^{-1} \left( 2 M_{0} \{ \delta_{1,L,n} + \delta_{2,L,n}\} \right)
\end{align*}
where the equivalence follows from the fact that $t \mapsto \varpi_{L,n}(t)$ is non-decreasing (see Lemma \ref{lem:sieve-IU}).
By setting $M_{1} = 2M_{0}$ and $M>1$ this display implies that $\mathbf{P}(A_{n}(2M_{0},M) \cap B_{n}) = 0$ thus proving the desired result. $\square$

\section{Appendix for Section \protect\ref{sec:conv-rate}}
\label{app:conv-rate}

This Appendix contains the proofs of all the Lemmas presented in Section \ref{sec:conv-rate}.

The following lemmas are used to prove the Lemmas in Section \ref{sec:conv-rate}; the proofs are relegated to the Supplementary Material \ref{supp:conv-rate}.

\begin{lemma}
\label{lem:SJ-bound} Let Assumption \ref{ass:reg} hold. Suppose $%
(\alpha,P) \in \mathbb{A} \times \mathcal{P}(\mathbb{Z})$ and $\delta>0$ are such that: There exists finite $C>0$ such that (1) $\sup_{\lambda \in B(\delta)} \sup_{z \in supp(P)} s^{\prime \prime}(\lambda^{T} g_{J}(z,\alpha)) \leq - \sqrt{C}$, (2) $e_{\min}(H_{J}(\alpha,P)) \geq \sqrt{C}$, (3) $2C^{-1}
||E_{P}[g_{J}(Z,\alpha)]||_{e} < \delta$, and (4) the hypothesis of Lemma %
\ref{lem:Lambda-charac} are satisfied for $\epsilon>0$.
\newline Then, with probability higher than $1-\epsilon$,

\begin{enumerate}
\item $\arg\max_{\lambda \in \Lambda_{J}(\alpha,P)} S_{J}(\alpha,\lambda,P)
= \arg\max_{\lambda \in B(\delta)} S_{J}(\alpha,\lambda,P) = \{
\lambda_{J}(\alpha,P) \}$; $\frac{d
S_{J}(\alpha,\lambda_{J}(\alpha,P),P)}{d \lambda} = 0$.

\item $\sup_{\lambda \in \Lambda(\alpha,P)}S(\alpha,\lambda,P) \leq 2C^{-1}
||E_{P}[g_{J}(Z,\alpha)]||^{2}_{e}$.

\item $||\lambda_{J}(\alpha,P)||_{e} \leq 2 C^{-1}
||E_{P}[g_{J}(Z,\alpha)]||_{e}$.
\end{enumerate}
\end{lemma}

\begin{lemma}
	\label{lem:avg-gJ} Let Assumption \ref{ass:pdf0} hold. Then, for any
	(non-random) $\alpha \in \mathcal{A}$ and any $J \in \mathbb{N}$,
	\begin{align*}
	||E_{P_{n}}[g_{J}(Z,\alpha)]||_{e} = O_{\mathbf{P}} \left( \sqrt{\frac{%
			\overline{\theta} + ||\mu h^{\prime}||^{2}_{L^{2}(\mathbf{P})} + b^{2}_{2,J} }{n} +
		||E_{\mathbf{P}}[g_{J}(Z,\alpha)]||^{2}_{e}} \right).
	\end{align*}
	(the constant implicit in the $O_{\mathbf{P}}$ does not depend on $J$).
\end{lemma}

\begin{lemma}
	\label{lem:H-bound} Suppose Assumptions \ref{ass:pdf0} and \ref{ass:reg}
	hold. Then for any $\alpha$ in a $||\cdot ||$- neighborhood of $\alpha_{0}$,
	
	\begin{enumerate}
		\item $||H_{J}(\alpha,P_{n}) - H_{J}(\alpha,\mathbf{P})||_{e} = O_{\mathbf{P}}(\{
		\theta^{2} + ||h^{\prime}||^{2}_{L^{\infty}(\mathbb{W},\mu)} \} \sqrt{%
			b^{4}_{4,J}/n})$.\footnote{%
			For matrices, $||.||_{e}$ is the operator norm induced by the Euclidean norm.%
		}
		
		\item If $\{ \theta^{2} + ||h^{\prime}||^{2}_{L^{\infty}(\mathbb{W},\mu)} \}
		\sqrt{b^{4}_{4,J}/n} = o(1)$, there exists a $C < \infty$ such that wpa1, \begin{align*}
			1/C \leq e_{min}(H_{J}(\alpha,P)) \leq e_{max}(H_{J}(\alpha,P)) \leq C
		\end{align*}
		for $%
		P \in \{P_{n}, \mathbf{P} \}$.
	\end{enumerate}
\end{lemma}

\begin{lemma}
\label{lem:norm-g-rate} Suppose Assumption \ref{ass:pdf0} and \ref{ass:reg}
hold. For any $(L=(J,K),n)$, and any positive real-valued sequence $(\delta_{n})_{n}$ satisfying Assumption \ref{ass:rates}(i)(iii)(iv)(v) and $\delta_{n} = o(1)$, it follows that
\begin{align*}
||n^{-1} \sum_{i=1}^{n} g_{J}(Z_{i},\hat{\alpha}_{L,n}) ||_{e} \precsim
C_{L,n} \left\{ \delta_{n} + \delta^{-1}_{n} \left\{
||E_{P_{n}}[g_{J}(Z,\alpha_{L,0})]||^{2}_{e} + \gamma_{K} Pen(\alpha_{L,0})
\right\} \right\}.
\end{align*}
\end{lemma}

\begin{lemma}\label{lem:moment-UBound}
	Suppose Assumption \ref{ass:reg}(iii). Then for $P \in \{ P_{n} , \mathbf{P} \}$,
	\begin{align*}
		\sup_{\alpha \in \bar{\mathcal{A}}_{L,n}} ||E_{P}[g_{J}(Z,\alpha)]||_{e} \precsim \overline{\theta} + l_{n} \gamma^{-1}_{K} \Gamma_{L,n} + E_{P}[||q^{J}(X)||_{e}]
	\end{align*}
	wpa1.
\end{lemma}

\begin{proof}[\textbf{Proof of Lemma \ref{lem:eff-sieve}}]
	By Lemma \ref{lem:Pen-bound} and the fact that $\Pi_K \alpha_{0} \in \mathcal{A}_K$, we have
	\begin{align*}
	Pen(\hat{\alpha}_{L,n}) \leq \gamma^{-1}_{K} \sup_{\lambda \in \Lambda_{J}(\Pi_K \alpha_{0},P_{n})} S(\Pi_K \alpha_{0},\lambda,P_{n}) + Pen(\Pi_K \alpha_{0}).
	\end{align*}
	
	By Lemma \ref{lem:SJ-bound} applied to $(\Pi_K \alpha_{0},P_{n})$ (by Lemma \ref{lem:sJ-bound-alpha0L} in the Supplementary Material \ref{supp:conv-rate}, the conditions of the Lemma \ref{lem:SJ-bound} hold wpa1) it follows that
	\begin{align*}
	Pen(\hat{\alpha}_{L,n}) \precsim \gamma^{-1}_{K} ||E_{P_{n}}[g_{J}(Z,\Pi_K \alpha_{0})]||^{2}_{e} + Pen(\Pi_K \alpha_{0})~~~wpa1.
	\end{align*}
	By Lemma \ref{lem:avg-gJ}, wpa1, $||E_{P_{n}}[g_{J}(Z,\Pi_K \alpha_{0})]||^{2}_{e} \leq l_{n} \left( \bar{g}_{L,0}/n + ||E_{P}[g_{J}(Z,\Pi_K \alpha_{0})]||^{2}_{e}  \right)$, and thus the result follows.
\end{proof}

\begin{proof}[\textbf{Proof of Lemma \ref{lem:alpha0L-exists}}]
	Throughout, fix $L=(J,K)$. We show that the ``argmin" is non-empty by invoking the Weierstrass Theorem (see \cite{Zeidler1985}). For this, note that $\alpha \mapsto Q_{J}(\alpha,\mathbf{P})$ is a continuous transformation of
\[\alpha \mapsto E[g_{J}(Z,\alpha)]^{T} = \left(\theta + E[\ell(W)h(W)],E[E[(F_{Y|WX}(h(W)\mid W,X)-\tau)|X]q^{J}(X)^{T}]\right )
\]
(here $F_{Y|WX}$ is the conditional cdf of $Y$ given $W,X$ associated to $\mathbf{P}$ ). Since $\ell \in L^{2}(Leb)$ and $\mathbf{p}_{W}$ is bounded (see Assumption \ref{ass:pdf0}), then $h \mapsto E[\ell(W)h(W)]$ is continuous with respect to $||.||_{L^{2}(Leb)}$. Also under Assumption \ref{ass:pdf0},
	\begin{align*}
		|E[F_{Y|WX}(h_{1}(W)\mid W,x) - F_{Y|WX}(h_{2}(W)\mid W,x)|x]| \precsim & \int |h_{1}(w) - h_{2}(w)|  \mathbf{p}_{W|X}(w|x)dw \\
		\leq &  \inf_{x} 1/\mathbf{p}_{X}(x) \int |h_{1}(w) - h_{2}(w)|    \mathbf{p}_{W}(w) dw  \\
		\precsim & ||h_{1} - h_{2}||_{L^{2}(Leb)},
	\end{align*}
	so $h \mapsto E[E[(F_{Y|WX}(h(W)\mid W,X)-\tau)|X]q^{J}(X)]$ is also continuous with respect to  $||.||_{L^{2}(Leb)}$. Under assumption \ref{ass:reg}, $\alpha \mapsto Q_{J}(\alpha,\mathbf{P}) + \gamma_{K} Pen(\alpha)$ is lower semi-compact, and $\mathcal{A}_{K}=\Theta \times \mathcal{H}_{K}$ is a finite dimensional and closed set. So all the assumptions of the Weierstrass Theorem hold.
\end{proof}

\begin{proof}[\textbf{Proof of Lemma \ref{lem:alpha0L-consistent}}]
	In the proof we simply use $L$ instead of $L_{n}$. We first present some intermediate results.

	From the definition of $Proj_{J}$ it follows that
	\begin{align*}
	E_{\mathbf{P}}[g_{J}(Z,\alpha)]  = E_{\mathbf{P}}[M_{J}(X) \rho(Y,W,\alpha)] = (E_{\mathbf{P}}[\rho_{1}(Y,W,\alpha)],E_{\mathbf{P}}[\rho_{2}(Y,W,\alpha)q^{J}(X)^{T}])^{T}.
	\end{align*}
	This in turn implies that
	\begin{align*}
	\left( E_{\mathbf{P}}[g_{J}(Z,\alpha)] \right)^{T} M_{J}(X) = & (E_{\mathbf{P}}[\rho_{1}(Y,W,\alpha)],E_{\mathbf{P}}[\rho_{2}(Y,W,\alpha)q^{J}(X)^{T}])M_{J}(X) \\
	= & (E_{\mathbf{P}}[\rho_{1}(Y,W,\alpha)],Proj_{J}[m_{2}(\cdot,\alpha)](X)).
	\end{align*}
	Since $E_{\mathbf{P}}[M_{J}(X) M_{J}(X)^{T}] = I_{J+1}$ under Assumption \ref{ass:reg} the previous result implies that
	\begin{align*}
	(E_{\mathbf{P}}[\rho_{1}(Y,W,\alpha)])^{2} + ||Proj_{J}[m_{2}(\cdot,\alpha)]||^{2}_{L^{2}(\mathbf{P})} = & E_{\mathbf{P}}[ (E_{\mathbf{P}}[g_{J}(Z,\alpha)])^{T} M_{J}(X) M_{J}(X)^{T}  (E_{\mathbf{P}}[g_{J}(Z,\alpha)])  ] \\
	= & ||E_{\mathbf{P}}[g_{J}(Z,\alpha)]||^{2}_{e}.
	\end{align*}
	
	This observation and the proof of Lemma \ref{lem:H-bound} (applied to $\alpha_0$) implies that
	\begin{align*}
	Q_{J}(\alpha,\mathbf{P}) \geq & c^{-1}  E_{\mathbf{P}}[g_{J}(Z,\alpha)^{T}] (E_{\mathbf{P}}[M_{J}(X)M_{J}(X)^{T}])^{-1}  E_{\mathbf{P}}[g_{J}(Z,\alpha)] \\
	= & c^{-1} ||E_{\mathbf{P}}[g_{J}(Z,\alpha)]||^{2}_{e} \\
	= & c^{-1} \left\{ (E_{\mathbf{P}}[\rho_{1}(Y,W,\alpha)])^{2} + ||Proj_{J}[m_{2}(\cdot,\alpha)]||^{2}_{L^{2}(\mathbf{P})}\right\}
	\end{align*}
	and
	\begin{align*}
Q_{J}(\alpha,\mathbf{P}) \leq & c \left\{ (E_{\mathbf{P}}[\rho_{1}(Y,W,\alpha)])^{2} + ||Proj_{J}[m_{2}(\cdot,\alpha)]||^{2}_{L^{2}(\mathbf{P})}\right\}
\end{align*}	
	for some $c>1$.
	
	By the conditions in the lemma, uniformly over $\alpha \in \bar{\mathcal{A}}_{L,n}$,
	\begin{align*}
	\liminf_{J \rightarrow \infty}   Q_{J}(\alpha,\mathbf{P}) \geq c^{-1} \left\{ (E_{\mathbf{P}}[\rho_{1}(Y,W,\alpha)])^{2} + ||m_{2}(\cdot ,\alpha)||^{2}_{L^{2}(\mathbf{P})}\right\}
	\end{align*}
	and
	\begin{align*}
\limsup_{J \rightarrow \infty}   Q_{J}(\alpha,\mathbf{P}) \leq c \left\{ (E_{\mathbf{P}}[\rho_{1}(Y,W,\alpha)])^{2} + ||m_{2}(\cdot ,\alpha)||^{2}_{L^{2}(\mathbf{P})}\right\}.
\end{align*}	
	
	We now turn to the proof of the claim in the Lemma. We show this by contradiction. That is, suppose that there exists a $c>0$ such that $ ||\alpha_{L,0} - \Pi_{K} \alpha_{0} || \geq c$ i.o. Therefore, for any $L$ for which the expression holds, it follows that
	\begin{align*}
	\inf_{\alpha \in \bar{\mathcal{A}}_{L,n} \colon || \alpha - \Pi_{K} \alpha_{0} || \geq c} Q_{J}(\alpha,\mathbf{P}) \leq Q_{J}(\Pi_{K} \alpha_{0},\mathbf{P}).
	\end{align*}
	
	Under Assumption \ref{ass:reg}, $\{ \alpha \in \bar{\mathcal{A}}_{L,n} \colon || \alpha - \Pi_{K} \alpha_{0} || \geq c \}$ is compact under $||.||$ and by the proof of Lemma \ref{lem:alpha0L-exists}, $\alpha \mapsto Q_{J}(\alpha,\mathbf{P})$ is continuous under $||.||$. Thus, there exists a $\alpha_{L} \in \bar{\mathcal{A}}_{L,n}$ such that (i) $|| \alpha_{L} - \Pi_{K} \alpha_{0} || \geq c$ and (ii) $Q_{J}(\alpha_{L},\mathbf{P}) \leq Q_{J}(\Pi_{K} \alpha_{0},\mathbf{P})$.

	By the first part of the proof and our conditions, $Q_{J}(\Pi_{K} \alpha_{0},\mathbf{P}) = o(1)$ so this implies that $Q_{J}(\alpha_{L},\mathbf{P}) = o(1)$. By the first part of the proof, this result implies that
	\begin{align*}
	(E_{\mathbf{P}}[\rho_{1}(Y,W,\alpha_{L})])^{2} + ||m_{2}(\cdot,\alpha_{L})||^{2}_{L^{2}(\mathbf{P})} = o(1).
	\end{align*}
	
	Since, under our conditions, $\mho_{L,n} \leq K < \infty$, it follows that the sequence $(\alpha_{L})_{L}$ belongs to $\{ \alpha \colon Pen(\alpha) \leq K  \}$ which is compact under $||.||$ by Assumption \ref{ass:reg}. Thus, there exists a convergent subsequence with limit $\alpha^{\ast}$. On the one hand, $|| \alpha_{L} - \Pi_{K} \alpha_{0} || \geq c$, it follows that $\alpha^{\ast} \ne \alpha_{0}$. But, on the other hand, continuity of $\alpha \mapsto 	(E_{\mathbf{P}}[\rho_{1}(Y,W,\alpha)])^{2} + ||m_{2}(\cdot,\alpha)||^{2}_{L^{2}(\mathbf{P})} $ and the previous display imply that
	\begin{align*}
	(E_{\mathbf{P}}[\rho_{1}(Y,W,\alpha^{\ast})])^{2} + ||m_{2}(\cdot,\alpha^{\ast})||^{2}_{L^{2}(\mathbf{P})} =0
	\end{align*}
	but this contradicts the identification condition in Assumption \ref{ass:ident}.
\end{proof}

\begin{proof}[\textbf{Proof of Lemma \ref{lem:sieve-IU}}]
	The proof of Lemma \ref{lem:alpha0L-exists} implies that $\alpha \mapsto Q_{J}(\alpha,\mathbf{P})$ is continuous and that $t \mapsto \{ \alpha \in \bar{\mathcal{A}}_{L,n} \colon ||\alpha - \alpha_{L,0} || \geq t \}$ is a compact-valued correspondence which is also continuous. Thus by the Theorem of the Maximum, $\varpi_{L,n}$ is continuous. This also implies that the ``inf" is achieved; this and the definition of $\alpha_{L,0}$ and assumption \ref{ass:ID-pseudotrue} imply that  $\varpi_{L,n}(t) = 0$ iff $t = 0$. The fact that it is non-decreasing is trivial to show.
\end{proof}

\begin{proof}[\textbf{Proof of Lemma \ref{lem:QJ-approx-min}}]
	Observe that
	\begin{align*}
	Q_{J}(\hat{\alpha}_{L,n},P_{n}) \leq & \frac{1}{e_{min}(H_{J}(\alpha_{0},\mathbf{P}))} ||E_{P_{n}}[g_{J}(Z,\hat{\alpha}_{L,n})]||^{2}_{e} .
	\end{align*}
	By Lemma \ref{lem:H-bound}(2), it follows that
	\begin{align*}
Q_{J}(\hat{\alpha}_{L,n},P_{n}) \leq & C^{-1} ||E_{P_{n}}[g_{J}(Z,\hat{\alpha}_{L,n})]||^{2}_{e}.
\end{align*}	
By Lemma \ref{lem:norm-g-rate}
	\begin{align*}
Q_{J}(\hat{\alpha}_{L,n},P_{n}) \precsim  C_{L,n} \left\{ \delta_{n} + \delta^{-1}_{n} \left\{
||E_{P_{n}}[g_{J}(Z,\alpha_{L,0})]||^{2}_{e} + \gamma_{K} Pen(\alpha_{L,0})
\right\} \right\}.
\end{align*}
By Lemma \ref{lem:avg-gJ} and definition of $\delta_{2,L,n}$,
	\begin{align*}
		Q_{J}(\hat{\alpha}_{L,n},P_{n}) =  O_{\mathbf{P}}(\delta_{2,L,n}).
	\end{align*}
\end{proof}

\begin{proof}[\textbf{Proof of Lemma \ref{lem:QJ-univ-conv}}]
	Note that
	\begin{align*}
	|Q_{J}(\alpha,P_{n}) - Q_{J}(\alpha,P)| \leq & ||H_{J}(\alpha_{0},\mathbf{P})^{-1/2} (E_{P_{n}}[g_{J}(Z,\alpha)] - E_{\mathbf{P}}[g_{J}(Z,\alpha)])||_{e} \\
	& \times ||H_{J}(\alpha_{0},\mathbf{P})^{-1/2}(E_{P_{n}}[g_{J}(Z,\alpha)] + E_{\mathbf{P}}[g_{J}(Z,\alpha)])||_{e}
	\end{align*}
	This bound and Lemma \ref{lem:H-bound}(2) (applied to $\alpha_0$) imply that it suffices to show that
	\begin{align*}
			&\sup_{\alpha \in \bar{\mathcal{A}}_{L,n}} Term_{2,J}(\alpha,P_{n}) = O_{\mathbf{P}} \left( \overline{\theta} + l_{n} \gamma_{K}^{-1} \Gamma_{L,n} + b_{2,J}  \right),\\
			and~& 	\sup_{\alpha \in \bar{\mathcal{A}}_{L,n}} Term_{1,J}(\alpha,P_{n}) = O_{\mathbf{P}} \left( \sqrt{\frac{J}{n}} \Delta_{L,n} \right)
	\end{align*}
	where
	\begin{align*}
		Term_{1,J}(\alpha,P_{n}) \equiv ||E_{P_{n}}[g_{J}(Z,\alpha)] - E_{\mathbf{P}}[g_{J}(Z,\alpha)]||_{e}\\
		Term_{2,J}(\alpha,P_{n}) \equiv ||E_{P_{n}}[g_{J}(Z,\alpha)] + E_{\mathbf{P}}[g_{J}(Z,\alpha)]||_{e}.
	\end{align*}
	
	By Lemma \ref{lem:moment-UBound},
	\begin{align*}
		\sup_{\alpha \in \bar{\mathcal{A}}_{L,n}} ||E_{P_{n}}[g_{J}(Z,\alpha)]||_{e} \precsim \overline{\theta} + l_{n} \gamma_{K}^{-1} \Gamma_{L,n} + E_{P_{n}}[||q^{J}(X)||_{e}]
	\end{align*}
	wpa1; and similarly for $\sup_{\alpha \in \bar{\mathcal{A}}_{L,n}} ||E_{\mathbf{P}}[g_{J}(Z,\alpha)]||_{e}$. Hence,
	\begin{align*}
		\sup_{\alpha \in \bar{\mathcal{A}}_{L,n}} Term_{2,J}(\alpha,P_{n}) = O_{\mathbf{P}} \left( \overline{\theta} + l_{n} \gamma_{K}^{-1} \Gamma_{L,n} + b_{2,J}  \right).
	\end{align*}

    Regarding $Term_{1,J}(\alpha,P_{n})$, by definition of $g_{J}$ and simple algebra,
    \begin{align*}
    	\sup_{\alpha \in \bar{\mathcal{A}}_{L,n}} Term_{1,J}(\alpha,P_{n}) \leq & n^{-1/2} \sup_{(\theta,h) \in \bar{\mathcal{A}}_{L,n}} |\mathbb{G}_{n}[\mu \cdot h']|\\
    	& + \sup_{\alpha \in \bar{\mathcal{A}}_{L,n}} || n^{-1} \sum_{i=1}^{n} \rho_{2}(Y_{i},W_{i},\alpha)q^{J}(X_{i}) - E_{\mathbf{P}}[\rho_{2}(Y,W,\alpha)q^{J}(X)] ||_{e}.
    \end{align*}
 By Assumption \ref{ass:Donsker}, the first term in the RHS is $O_{\mathbf{P}}(n^{-1/2}\Delta_{L,n})$; the second term in the RHS can be bounded above by $\sqrt{\frac{J}{n}} \max_{1 \leq j \leq J} \sup_{g \in \mathcal{G}_{K}}|\mathbb{G}_{n}[g \cdot q_{j} ]|$, which by Assumption \ref{ass:Donsker} is $O_{\mathbf{P}} \left( \sqrt{\frac{J}{n}} \Delta_{L,n} \right)$.
	
\end{proof}

\section{Proof of Theorem \ref{thm:QLR}}
\label{app:Thm-QLR}

By Assumption \ref{ass:undersmooth}(i) and Lemma \ref{lem:anorm} in Appendix \ref{app:ADT},  $n^{-1/2} \sum_{i=1}^{n} (G(\alpha_{L,0})[u^{\ast}_{L,n}])^{T}H_{L}^{-1} g_{J}(Z_{i},\alpha_{L,0})  \Rightarrow N(0,1)$. This fact, Lemma \ref{lem:QLR-A-rep} and Assumption \ref{ass:undersmooth}(ii), imply the desired result. $\square$

\section{Appendix for Section \protect\ref{sec:ADT}}

\label{app:ADT}

For any $L=(J,K) \in \mathbb{N}^{2}$, let $M_{L} \in \mathbb{R}^{(1+J)\times (1+K)}$ be defined as
\begin{align*}
M_{L} = & \left[
\begin{array}{cc}
1 & E[\ell(W)\varphi^{K}(W)^{T}] \\
\mathbf{0} & E[\mathbf{p}_{Y|WX}(h_{L,0}(W) | W,X) q^{J}(X)\varphi^{K}(W)^{T}]%
\end{array}
\right]
\end{align*}
where $\mathbf{0}$ is a $J \times 1$ vector of zeros. The following eight lemmas are proved in the Section \ref{supp:ADT} in the Supplemental Material.

\begin{lemma}
\label{lem:weak-norm} Let Assumptions \ref{ass:ident}-\ref{ass:ID-pseudotrue} and \ref{ass:rates-LQA}(iv)(v) hold. Then, there exists a $c>0$ such that for any $L = (J,K) \in \mathbb{N}^{2}$ and any $\alpha =
(\theta,\pi^{T} \varphi^{K})$ some $\pi \in \mathbb{R}^{K}$ (i.e., $\alpha
\in lin\{ \mathcal{A}_{K} \}$), it follows that
\begin{align*}
|| \alpha||_{w} \geq c \times \sqrt{e_{min}(M_{L}^{T} M_{L}) }
||(\theta,\pi)||_{e}
\end{align*}
where $e_{min}(M_{L}^{T} M_{L})>0$.
\end{lemma}

\begin{lemma}
\label{lem:HJhat-LQA} Let Assumptions \ref{ass:ident}-\ref{ass:ID-pseudotrue}, \ref{ass:HJ-sec} and \ref{ass:rates-LQA}(v) hold. Then
\begin{align*}
\sup_{\alpha \in \mathcal{N}_{L,n}} ||H_{J}(\alpha,P_{n}) -
H_{J}(\alpha_{L,0},\mathbf{P})||_{e} = O_{\mathbf{P}}(\Xi_{1,L,n})
\end{align*}
with
\begin{align*}
\Xi_{1,L,n} = O_{\mathbf{P}} \left( ( \overline{\theta} + ||h^{\prime
}_{L,0}||_{L^{\infty}(\mathbb{W},\mu)} )^{2} \sqrt{b_{4,J}/n} + \mho_{L,n}
\eta_{L,n} + \Xi_{L,n} \right)
\end{align*}
\end{lemma}

\begin{lemma}
\label{lem:sec-gJ} For any $L=(J,K) \in \mathbb{N}^{2}$,
\begin{align*}
&\sup_{ \alpha \in \mathcal{N}_{L,n}} \left \Vert n^{-1}\sum_{i=1}^{n}
g_{J}(Z_{i},\alpha) - g_{J}(Z_{i},\alpha_{L,0}) - E[g_{J}(Z,\alpha) -
g_{J}(Z,\alpha_{L,0})] \right \Vert_{e} \\
& \precsim   \sqrt{\frac{J}{n}} \left( \sup_{(\theta,h) \in \mathcal{N}_{L,n}} \mathbb{G}_{n}[\mu \cdot h^{\prime }] + \max_{1 \leq j \leq J}
\sup_{h \in \bar{\mathcal{G}}_{L,n}} \mathbb{G}_{n}[g \cdot q_{j}] \right).
\end{align*}
\end{lemma}

\begin{lemma}
\label{lem:charac-N} Let Assumptions \ref{ass:ident}-\ref{ass:ID-pseudotrue}, \ref{ass:dev-N}(ii)(iii) and \ref{ass:rates-LQA}(iv)(v) hold. If $\alpha \in int(\mathcal{N}_{L,n})$, then $\alpha + t
u^{\ast}_{L,n} \in \mathcal{N}_{L,n}$ for all $|t| \leq l_{n}n^{-1/2}$.
\end{lemma}

\begin{lemma}
\label{lem:SJ-bound-PSGEL} Let Assumptions \ref{ass:ident}-\ref{ass:ID-pseudotrue},
 \ref{ass:dev-N}, \ref{ass:Donsker-LQA} and \ref{ass:rates-LQA} hold. Then
all the conditions of Lemma \ref{lem:SJ-bound} hold for $(\hat{\alpha}%
_{L,n},P_{n})$ and $(\hat{\alpha}^{\nu}_{L,n},P_{n})$ wpa1.
\end{lemma}

\begin{lemma}
\label{lem:lambda-maxLQA} Let Assumptions \ref{ass:ident}-\ref{ass:ID-pseudotrue} and \ref{ass:rates-LQA}(v) hold. For
any $L=(J,K) \in \mathbb{N}^{2}$, any $\gamma>0$ and any $\alpha \in \bar{%
\mathcal{A}}_{L,n}$ such that $||\alpha - \alpha_{L,0}||_{w} \leq \gamma$,
it follows that
\begin{align*}
||H^{-1}_{L} \varDelta(\alpha)||_{e} = O_{\mathbf{P}}\left( \sqrt{ \frac{\bar{g}%
_{L,0}^{2}}{n} + ||E[g_{J}(Z,\alpha_{L,0})]||^{2}_{e} } + \gamma \right).
\end{align*}
\end{lemma}

\begin{lemma}
\label{lem:ALR-thetahat} Let Assumptions \ref{ass:ident}-\ref{ass:ID-pseudotrue},
and \ref{ass:rates-LQA} hold. Then
\begin{align*}
\langle u^{\ast}_{L,n} , \hat{\alpha}_{L,n} - \alpha_{L,0} \rangle_{w} =
n^{-1} \sum_{i=1}^{n} (G(\alpha_{L,0})[u^{\ast}_{L,n}])^{T}H_{L}^{-1}
g_{J}(Z_{i},\alpha_{L,0}) + o_{\mathbf{P}}(n^{-1/2}).
\end{align*}
\end{lemma}

\begin{lemma}
\label{lem:anorm} Let Assumptions \ref{ass:ident}, \ref{ass:ID-pseudotrue} and \ref{ass:rates-LQA} hold. Then,
\begin{align*}
n^{-1/2} \sum_{i=1}^{n} (G(\alpha_{L,0})[u^{\ast}_{L,n}])^{T}H_{L}^{-1}
\{g_{J}(Z_{i},\alpha_{L,0}) - E[g_{J}(Z,\alpha_{L,0})] \} \Rightarrow N(0,1).
\end{align*}
\end{lemma}

\begin{proof}[\textbf{Proof of Lemma \ref{lem:LQA}}]
  		By the calculations in the proof of Lemma \ref{lem:SJ-bound}, for all $(\alpha,\lambda) \in \mathcal{N}_{L,n} \times B(\delta_{n})$,
  		\begin{align*}
  			\hat{S}_{J}(\alpha,\lambda) \geq - \lambda^{T} E_{P_{n}}[g_{J}(Z,\alpha)] - \frac{1}{2} \lambda^{T} H_{J}(\alpha,P_{n}) \lambda + O( \delta^{3} E_{P_{n}}[||g_{J}(Z,\alpha)||^{3}_{e}] );
  		\end{align*}
  		the reverse inequality can also be shown in similar fashion.
  		
  		By Lemma \ref{lem:HJhat-LQA}, $\sup_{\alpha \in \mathcal{N}_{L,n}} ||H_{J}(\alpha,P_{n}) - H_{J}(\alpha_{L,0},\mathbf{P})||_{e} = O_{\mathbf{P}}(\Xi_{1,L,n})$. Hence
   		\begin{align*}
   		\hat{S}_{J}(\alpha,\lambda) = - \lambda^{T} E_{P_{n}}[g_{J}(Z,\alpha)] - \frac{1}{2} \lambda^{T} H_{J}(\alpha_{L,0},\mathbf{P}) \lambda + O( \delta^{3} E_{P_{n}}[||g_{J}(Z,\alpha)||^{3}_{e}] + \delta^{2} \Xi_{1,L,n} ).
   		\end{align*}
   		
   	By Lemma \ref{lem:sec-gJ} and Assumption \ref{ass:Donsker-LQA}, it follows that\begin{align*}
   		\sup_{ \alpha \in \mathcal{N}_{L,n}} \left \Vert n^{-1}\sum_{i=1}^{n} g_{J}(Z_{i},\alpha) - g_{J}(Z_{i},\alpha_{L,0})  - E[g_{J}(Z,\alpha) - g_{J}(Z,\alpha_{L,0})] \right \Vert_{e} = O_{\mathbf{P}} \left( \sqrt{\frac{J}{n}} \Delta_{2,L,n} \right).
   	\end{align*}
   	Hence
   	\begin{align*}
  \hat{S}_{J}(\alpha,\lambda) = & - \lambda^{T} \left\{ E_{P_{n}}[ g_{J}(Z,\alpha_{L,0})]  + E[g_{J}(Z,\alpha) - g_{J}(Z,\alpha_{L,0})]   \right\} - \frac{1}{2} \lambda^{T} H_{J}(\alpha_{L,0},\mathbf{P}) \lambda \\
  & + O \left( \delta^{3} E_{P_{n}}[||g_{J}(Z,\alpha)||^{3}_{e}] + \delta^{2} \Xi_{1,L,n} + \delta \sqrt{\frac{J}{n}} \Delta_{2,L,n} \right).
  \end{align*}

  By the mean value theorem,
  \begin{align}\notag
  	E[g_{J}(Z,\alpha) - g_{J}(Z,\alpha_{L,0})] = & G(\alpha_{L,0})[\alpha - \alpha_{L,0}] \\ \label{eqn:LQA-1}
  	 & +  \int_{0}^{1} \{ G(\alpha_{L,0} + t(\alpha - \alpha_{L,0}))[\alpha - \alpha_{L,0}] - G(\alpha_{L,0})[\alpha - \alpha_{L,0}] \} dt.
  \end{align}
  Note that, for any $\alpha$ and $\alpha_{1}$,
  \begin{align*}
  	G(\alpha)[\alpha_{1}] - G(\alpha_{L,0})[\alpha_{1}] = \left[
  	\begin{array}{c}
  	   0 \\
  	   E[\{  \mathbf{p}_{Y|WX}(h(W)|W,X) -  \mathbf{p}_{Y|WX}(h_{L,0}(W) | W,X) \} h_{1}(W) q^{J}(X) ]
  	\end{array}
  	\right].
  \end{align*}
  By Assumption \ref{ass:pdf0}, $|\frac{d\mathbf{p}_{Y|WX}(y|w,x)}{dy} | \leq C$ some finite $C$ and thus, for each $j$,
  \begin{align*}
  	& |E[\{  \mathbf{p}_{Y|WX}(h(W)|W,X) -  \mathbf{p}_{Y|WX}(h_{L,0}(W) | W,X) \} h_{1}(W) q_{j}(X) ]|\\
  	 \leq & C E[|h(W) - h_{L,0}(W)| \times |h_{1}(W)| |q_{j}(X)| ] \\
  	\leq & C \sup_{x} E[|h(W) - h_{L,0}(W)| \times |h_{1}(W)| |X=x] \times E[|q_{j}(X)|].
  \end{align*}
  Therefore,
  \begin{align*}
  	||	G(\alpha)[\alpha_{1}] - G(\alpha_{L,0})[\alpha_{1}] ||_{e} \precsim \sup_{x} E[|h(W) - h_{L,0}(W)| \times |h_{1}(W)| |X=x] \times E[||q^{J}(X)||_{e}].
  \end{align*}

  Applying these observations to the last term in the RHS of expression (\ref{eqn:LQA-1}), it follows that
  \begin{align*}
  	|| E[g_{J}(Z,\alpha) - g_{J}(Z,\alpha_{L,0})] - G(\alpha_{L,0})[\alpha - \alpha_{L,0}] ||_{e} = & O (\sup_{x} E[|h(W) - h_{L,0}(W)|^{2} |X=x] \times E[||q^{J}(X)||_{e}]) \\
  	= O & (||h - h_{L,0}||^{2}_{L^{2}(Leb)} E[||q^{J}(X)||_{e}])
  \end{align*}
  where the last line follows by Assumption \ref{ass:pdf0}. Therefore,
   	\begin{align*}
   	\hat{S}_{J}(\alpha,\lambda) = & - \lambda^{T} \left\{ E_{P_{n}}[ g_{J}(Z,\alpha_{L,0})]  + G(\alpha_{L,0})[\alpha - \alpha_{L,0}]    \right\} - \frac{1}{2} \lambda^{T} H_{J}(\alpha_{L,0},\mathbf{P}) \lambda \\
   	& + O \left( \delta^{3} E_{P_{n}}[||g_{J}(Z,\alpha)||^{3}_{e}] + \delta^{2} \Xi_{1,L,n} + \delta \{   \sqrt{\frac{J}{n}}\Delta_{2,L,n} + ||h - h_{L,0}||^{2}_{L^{2}(Leb)} b_{2,J}  \} \right).
   	\end{align*}
	Since $\hat{\alpha}_{L,n} \in \bar{\mathcal{A}}_{L,n}$ wpa1 (Lemma \ref{lem:eff-sieve}), it follows by the proof of Lemma \ref{lem:Lambda-charac} that $||g_{J}(Z,\hat{\alpha}_{L,n})||^{3}_{e} \leq (\overline{\theta} + l_{n} \gamma_{K}^{-1} \Gamma_{L,n} + ||q^{J}(X)||_{e} )^{3}$. So by the Markov inequality
   	\begin{align*}
   		E_{P_{n}}[||g_{J}(Z,\hat{\alpha}_{L,n})||^{3}_{e}] = O_{\mathbf{P}} \left( (\overline{\theta} + l_{n} \gamma_{K}^{-1} \Gamma_{L,n} + b_{3,J} )^{3} \right).
   	\end{align*}
   	
	\end{proof}

\begin{proof}[\textbf{Proof of Lemma \ref{lem:QLR-A-rep}}]
		
		\textsc{Step 1.} We show that, for $\alpha \in \{ \hat{\alpha}_{L,n} , \hat{\alpha}^{\nu}_{L,n}  \}$,
		\begin{align*}
		\sup_{\lambda \in \hat{\Lambda}_{J}(\alpha)} \hat{S}_{J}(\alpha,\lambda) = \frac{1}{2} \varDelta(\alpha)^{T}H_{L}^{-1}\varDelta(\alpha) + o_{\mathbf{P}}(n^{-1}).
		\end{align*}
		
		We first note that, for any $\alpha \in \{ \hat{\alpha}_{L,n} , \hat{\alpha}^{\nu}_{L,n}  \}$, by Lemma \ref{lem:SJ-bound-PSGEL} and Lemmas \ref{lem:Lambda-charac} $\{ B(\delta_{n}) \subseteq \hat{\Lambda}_{J}(\alpha)  \}$ wpa1. Also, by Assumption \ref{ass:dev-N}, under the null, $\{ \hat{\alpha}_{L,n} , \hat{\alpha}^{\nu}_{L,n}  \} \subseteq \mathcal{N}_{n}$. Hence, under Assumption \ref{ass:rates-LQA}, Lemma \ref{lem:LQA} implies that
		\begin{align*}
		\sup_{\lambda \in \hat{\Lambda}_{J}(\alpha)} \hat{S}_{J}(\alpha,\lambda) \leq & \sup_{\lambda \in B(\delta_{n})} - \lambda^{T} \varDelta(\alpha) - 0.5 \lambda^{T} H_{L}^{-1} \lambda + o_{\mathbf{P}}(n^{-1}) \\
		\leq & \frac{1}{2} \varDelta(\alpha)^{T}H_{L}^{-1}\varDelta(\alpha) + o_{\mathbf{P}}(n^{-1})
		\end{align*}
		where the first inequality is valid, because $\sup_{\lambda \in \hat{\Lambda}_{J}(\alpha)} \hat{S}_{J}(\alpha,\lambda) = \sup_{\lambda \in B(\delta_{n})} \hat{S}_{J}(\alpha,\lambda)$ by Lemmas \ref{lem:SJ-bound-PSGEL} and Lemma \ref{lem:SJ-bound}; the last inequality follows because the RHS is obtained by maximizing over the whole $\mathbb{R}^{J+1}$ not only $\hat{\Lambda}_{J}(\alpha)$.
		
		By Lemma \ref{lem:SJ-bound-PSGEL} and Lemma \ref{lem:SJ-bound},
		\begin{align*}
		\sup_{\lambda \in \hat{\Lambda}_{J}(\alpha)} \hat{S}_{J}(\alpha,\lambda) \geq - \lambda^{T} \varDelta(\alpha) - 0.5 \lambda^{T} H_{L}^{-1} \lambda + o_{\mathbf{P}}(n^{-1}),
		\end{align*}
		for all $\lambda \in B(\delta_{n})$. By Lemma \ref{lem:lambda-maxLQA} and Assumption \ref{ass:dev-N}, the maximizer of the RHS, $\lambda^{\ast}$, is such that $||\lambda^{\ast}||_{e} = O_{\mathbf{P}}\left( \sqrt{ \frac{\bar{g}_{L,0}^{2}}{n} + ||E[g_{J}(Z,\alpha_{L,0})]||^{2}_{e}  } + \eta_{w,L,n}       \right)$. Hence, by Assumption \ref{ass:rates-LQA},  $\lambda^{\ast} \in B(\delta_{n})$. Therefore,
		\begin{align*}
		\sup_{\lambda \in \hat{\Lambda}_{J}(\alpha)} \hat{S}_{J}(\alpha,\lambda) \geq \frac{1}{2} \varDelta(\alpha)^{T}H_{L}^{-1}\varDelta(\alpha) + o_{\mathbf{P}}(n^{-1}).
		\end{align*}
		
		\medskip
		
		\textsc{Step 2.} We now show that
		\begin{align*}
		\hat{\mathcal{L}}_{L,n}(\nu)  \geq &  o_{\mathbf{P}}(n^{-1}) +  \left( n^{-1} \sum_{i=1}^{n} (G(\alpha_{L,0})[u^{\ast}_{L,n}])^{T}H_{L}^{-1} g_{J}(Z_{i},\alpha_{L,0}) \right)^{2}  \\
		& + 2\frac{(\nu - \theta_{L,0})}{||v^{\ast}_{L,n}||_{w}}\left( n^{-1} \sum_{i=1}^{n} (G(\alpha_{L,0})[u^{\ast}_{L,n}])^{T}H_{L}^{-1} g_{J}(Z_{i},\alpha_{L,0}) \right).
		\end{align*}

		Using the results in step 1,
		\begin{align*}
		\hat{\mathcal{L}}_{L,n}(\nu) \geq &  \left\{  \varDelta(\hat{\alpha}^{\nu}_{L,n})^{T}H_{L}^{-1}\varDelta(\hat{\alpha}^{\nu}_{L,n}) - \varDelta(\alpha)^{T}H_{L}^{-1}\varDelta(\alpha)  \right\} +\gamma_{K}\{Pen(\hat{\alpha}^{\nu}_{L,n}) - Pen(\alpha) \} \\
		&  +  o_{\mathbf{P}}(n^{-1}).
		\end{align*}
		For any $\alpha \in \mathcal{A}_{K}$. In particular $\alpha = \hat{\alpha}^{\nu}_{L,n} - tu^{\ast}_{L,n}$ with $t = (G(\alpha_{L,0})[u^{\ast}_{L,n}])^{T}H_{L}^{-1} E_{P_{n}}[g_{J}(Z,\alpha_{L,0})]$. By Lemma \ref{lem:anorm} and Assumption \ref{ass:undersmooth}, $|t| = O_{\mathbf{P}}(n^{-1/2})$, so by Lemma \ref{lem:charac-N}, under Assumption \ref{ass:dev-N}, this choice of $\alpha$ belongs to $\mathcal{N}_{L,n}$. Moreover, by Assumption \ref{ass:dev-N}, $\gamma_{K}\{Pen(\hat{\alpha}^{\nu}_{L,n}) - Pen(\alpha) \} = o_{\mathbf{P}}(n^{-1})$. Hence, after some simple calculations,
		\begin{align*}
		\hat{\mathcal{L}}_{L,n}(\nu)  \geq &  o_{\mathbf{P}}(n^{-1}) -  \left\{ t^{2} ||u^{\ast}_{L,n}||^{2}_{w} - 2 t (G(\alpha_{L,0})[u^{\ast}_{L,n}])^{T}H_{L}^{-1} \varDelta(\hat{\alpha}^{\nu}_{L,n})  \right\}.
		\end{align*}
		Note that $\varDelta(\hat{\alpha}^{\nu}_{L,n}) = E_{P_{n}}[g_{J}(Z,\alpha_{L,0})] + G(\alpha_{L,0})[\hat{\alpha}^{\nu}_{L,n} - \alpha_{L,0}]$,	so \begin{align*}
		2  (G(\alpha_{L,0})[u^{\ast}_{L,n}])^{T}H_{L}^{-1} \varDelta(\hat{\alpha}^{\nu}_{L,n})  = &  2  (G(\alpha_{L,0})[u^{\ast}_{L,n}])^{T}H_{L}^{-1} E_{P_{n}}[g_{J}(Z,\alpha_{L,0})] \\
		& +  2  (G(\alpha_{L,0})[u^{\ast}_{L,n}])^{T}H_{L}^{-1}(G(\alpha_{L,0})[\hat{\alpha}^{\nu}_{L,n} - \alpha_{L,0}]) \\
		= &  2  (G(\alpha_{L,0})[u^{\ast}_{L,n}])^{T}H_{L}^{-1} E_{P_{n}}[g_{J}(Z,\alpha_{L,0})] \\
		& +  2 (\nu - \theta_{L,0})/||v^{\ast}_{L,n}||_{w},
		\end{align*}
		where the last line follows because $(G(\alpha_{L,0})[u^{\ast}_{L,n}])^{T}H_{L}^{-1}(G(\alpha_{L,0})[\hat{\alpha}^{\nu}_{L,n} - \alpha_{L,0}]) = \langle u^{\ast}_{L,n} , \hat{\alpha}^{\nu}_{L,n} - \alpha_{L,0}] \rangle_{w} = (\hat{\theta}_{L,0} - \theta_{L,0})/||v^{\ast}_{L,n}||_{w} $.
		
		This observation and the fact that $||u^{\ast}_{L,n}||_{w} = 1$, imply
		\begin{align*}
		\hat{\mathcal{L}}_{L,n}(\nu)  \geq  o_{\mathbf{P}}(n^{-1})  - \left\{ t^{2} - 2t   (G(\alpha_{L,0})[u^{\ast}_{L,n}])^{T}H_{L}^{-1} E_{P_{n}}[g_{J}(Z,\alpha_{L,0})] + 2t \frac{(\theta_{L,0}-\nu)}{||v^{\ast}_{L,n}||_{w}}    \right\}.
		\end{align*}
		By our choice of $t$,  it follows that
		\begin{align*}
		\hat{\mathcal{L}}_{L,n}(\nu)  \geq &  o_{\mathbf{P}}(n^{-1}) +  \left( n^{-1} \sum_{i=1}^{n} (G(\alpha_{L,0})[u^{\ast}_{L,n}])^{T}H_{L}^{-1} g_{J}(Z_{i},\alpha_{L,0}) \right)^{2}  \\
		& + 2\frac{(\nu - \theta_{L,0})}{||v^{\ast}_{L,n}||_{w}} \left( n^{-1} \sum_{i=1}^{n} (G(\alpha_{L,0})[u^{\ast}_{L,n}])^{T}H_{L}^{-1} g_{J}(Z_{i},\alpha_{L,0}) \right).
		\end{align*}
		
		\medskip
		
		\textsc{Step 3.} We now show that
		\begin{align*}
		\hat{\mathcal{L}}_{L,n}(\nu)  \leq &  \left(\frac{\theta_{L,0} - \theta_{0}}{||v^{\ast}_{L,n}||_{w}} - \left( n^{-1} \sum_{i=1}^{n} (G(\alpha_{L,0})[u^{\ast}_{L,n}])^{T}H_{L}^{-1} g_{J}(Z_{i},\alpha_{L,0}) \right) \right)^{2}  +  o_{\mathbf{P}}(n^{-1}).
		\end{align*}
		
		To do this we proceed as in Step 2. 	Using the results in step 1,
		\begin{align*}
		\hat{\mathcal{L}}_{L,n}(\nu) \leq & \left\{   \varDelta(\alpha)^{T}H_{L}^{-1}\varDelta(\alpha) - \varDelta(\hat{\alpha}_{L,n})^{T}H_{L}^{-1}\varDelta(\hat{\alpha}_{L,n})  \right\} \\
		&  +  \gamma_{K}\{Pen(\alpha) - Pen(\hat{\alpha}_{L,n}) \} + o_{\mathbf{P}}(n^{-1}),
		\end{align*}	
		for any $\alpha \in \mathcal{A}_{K}$ such that $\theta = \nu$. Let $\alpha = \hat{\alpha}_{L,n} - t u^{\ast}_{L,n}$, where $t = \langle u^{\ast}_{L,n} , \hat{\alpha}_{L,n} - \alpha_{L,0}  \rangle_{w} + \frac{\theta_{L,0} - \theta_{0}}{||v^{\ast}_{L,n}||_{w}} = \frac{\hat{\theta}_{L,n} - \theta_{0}}{||v^{\ast}_{L,n}||_{w}} $. Note that
		\begin{align*}
		\langle v^{\ast}_{L,n},  \alpha  \rangle_{w} =  \hat{\theta}_{L,n} - t ||v^{\ast}_{L,n}||_{w} = \hat{\theta}_{L,n} - (\hat{\theta}_{L,n} - \theta_{0}) = \theta_{0},
		\end{align*}
		and under the null, $\theta_{0} = \nu$. Thus $\alpha$ satisfies with the restriction $\nu = \theta$. Moreover, by Lemma \ref{lem:ALR-thetahat} and Assumption \ref{ass:undersmooth}, it follows that $t = O_{\mathbf{P}}(n^{-1/2})$. Thus, by Lemma \ref{lem:charac-N} --- under Assumption \ref{ass:dev-N} and the null ---, $\alpha \in \mathcal{N}_{L,n}$ wpa1.
		
		This observation, Assumption \ref{ass:dev-N}, and analogous calculations to those in Step 1 imply
		\begin{align*}
		\hat{\mathcal{L}}_{L,n}(\nu) \leq \left\{  t^{2} - 2  t (G(\alpha_{L,0})[u^{\ast}_{L,n}])^{T}H_{L}^{-1} \varDelta(\hat{\alpha}_{L,n}) \right\}   +  o_{\mathbf{P}}(n^{-1}),
		\end{align*}
		and \begin{align*}
		(G(\alpha_{L,0})[u^{\ast}_{L,n}])^{T}H_{L}^{-1} \varDelta(\hat{\alpha}_{L,n}) = & \langle u^{\ast}_{L,n} , \hat{\alpha}_{L,n} - \alpha_{L,0} \rangle_{w} + n^{-1} \sum_{i=1}^{n} (G(\alpha_{L,0})[u^{\ast}_{L,n}])^{T}H_{L}^{-1} g_{J}(Z_{i},\alpha_{L,0})\\
		\equiv & \frac{\hat{\theta}_{L,n} - \theta_{0}}{||v^{\ast}_{L,n}||_{w}} - \frac{\theta_{L,0} - \theta_{0}}{||v^{\ast}_{L,n}||_{w}}+ F_{n} \\
		= & t - \frac{\theta_{L,0} - \theta_{0}}{||v^{\ast}_{L,n}||_{w}}+ F_{n}.
		\end{align*}
		
		Therefore,
		\begin{align*}
		\hat{\mathcal{L}}_{L,n}(\nu) \leq &  \left\{ t^{2} - 2  t (t - \frac{\theta_{L,0} - \theta_{0}}{||v^{\ast}_{L,n}||_{w}}+ F_{n}) \right\}   +  o_{\mathbf{P}}(n^{-1}) \\
		= &  \left\{ - t^{2} + 2  t (\frac{\theta_{L,0} - \theta_{0}}{||v^{\ast}_{L,n}||_{w}} - F_{n}) \right\}   +  o_{\mathbf{P}}(n^{-1})\\
		= &  \left\{ - t^{2} + 2  t \left(\frac{\theta_{L,0} - \theta_{0}}{||v^{\ast}_{L,n}||_{w}} - F_{n} \right) - \left(\frac{\theta_{L,0} - \theta_{0}}{||v^{\ast}_{L,n}||_{w}} - F_{n} \right)^{2} \right\} + \left(\frac{\theta_{L,0} - \theta_{0}}{||v^{\ast}_{L,n}||_{w}} - F_{n} \right)^{2}  +  o_{\mathbf{P}}(n^{-1})\\
		\leq & \left(\frac{\theta_{L,0} - \theta_{0}}{||v^{\ast}_{L,n}||_{w}} - F_{n} \right)^{2}  +  o_{\mathbf{P}}(n^{-1}).
		\end{align*}

		\medskip

		\textsc{Step 4.} The desired result follows from imposing the null hypothesis in Step 2.
	\end{proof}

\pagebreak

\setcounter{section}{0}
\setcounter{page}{1}
\renewcommand{\thesection}{SM.\Roman{section}}

\begin{center}
	\Huge{Supplemental Material}
\end{center}

\section{Supplemental Material for Appendix \ref{app:eff-bound}}
\label{supp:eff-supp}

The next lemma presents some properties of the
tangent space.

\begin{lemma}
	\label{lem:tangent} The tangent space of $\mathcal{M}$ at $P\in \mathcal{M}$
	is included in the class of all $g \in L^{2}_{0}(P)$ such that
	\begin{align*}
	\int \rho_{2}(y,w,h(P)) g(y,w,\cdot) P_{YW \mid X}(dy,dw\mid \cdot ) \in
	Range(\mathbf{T}_{P}).
	\end{align*}
\end{lemma}

\begin{proof} By our definition of $\mathcal{M}$ any curve, $t \mapsto P[t]$, must satisfy \begin{align*}
	\int \rho_{2}(y,w,h(P[t]))p[t]_{YW|X} (y,w\mid X) dydw = 0.
	\end{align*}
	Since the curve is in $\mathcal{M}$ it has a pdf, which we denote as $p[t]$; also we write $\rho_{2}$ only as a function of $h$ and not $\alpha$ to stress the fact that $\theta$ is not present. Since $\theta(P) \in int(\Theta)$ the equation $\theta(P[t]) = - E_{P[t]}[\ell_{P[t]}(W) h(P[t])(W)]$ does not impose --- locally --- any restrictions; so it can be ignored for the computation of the tangent space.
	
	Since $p_{X} > 0$ for any $P$ in the model, the previous display implies that
	\begin{align*}
	\int \rho_{2}(y,w,h(P[t]))p_{YW|X}[t](y,w\mid X) dydw = \int (F_{Y|WX}[t](h(P[t])(w) \mid w,x) - \tau) p_{WX}[t](w, x) dw,
	\end{align*}
	where $F_{Y|WX}[t]$ is the conditional cdf of $Y$, given $w,x$ associated to $P[t]$. By our assumption \ref{ass:pdf0}, taking derivative with respect to $t$ implies
	\begin{align*}
	\int \rho_{2}(y,w,h(P)) g(y,w,\cdot) p(y,w \mid x)dydw = \mathbf{T}_{P}\left[ \dot{h}(P)[g]  \right](x)
	\end{align*}
	for all $x \in \mathbb{X}$. Hence, for any $g$ in the tangent space, it must hold that the LHS function belongs to $Range (\mathbf{T}_{P})$.
\end{proof}

\begin{proof}[\textbf{Proof of Lemma \ref{lem:diff-alpha}}]
	
	\textsc{Part 1.} By laws of differentiation, if the derivative exists, it has to satisfy
	\begin{align*}
	\frac{d \int (\ell_{\mathbf{P}[t](g)}(w) h_{id}(\mathbf{P}[t](g))(w)) dw }{dt} \mid_{t=0} = \frac{d \int \ell_{\mathbf{P}}(w) h_{id}(\mathbf{P}[t](g))(w) dw }{dt} \mid_{t=0} + \frac{d \int \ell_{\mathbf{P}[t](g)}(w) h_{id}(\mathbf{P})(w) dw }{dt} \mid_{t=0}.
	\end{align*}
	
	Since for all $P \in \mathcal{M}$, \begin{align*}
		\ell_{P}(w) = \mu^{\prime}(w)p_{W}(w) + \mu(w) p^{\prime}_{W}(w) = \mu^{\prime}(w)\int p(y,w,x)dydx + \mu(w) \frac{d\int p(y,w,x) dydx}{dw}
	\end{align*},
	it follows that
	\begin{align*}
	\frac{d \int \ell_{\mathbf{P}[t](g)}(w) h_{id}(\mathbf{P})(w) dw }{dt} \mid_{t=0} = & \frac{d \int \mu'(w) h_{id}(\mathbf{P})(w) \mathbf{p}[t](g)(w)  dw }{dt} \mid_{t=0} \\
	& + \frac{d \int \mu(w) h_{id}(\mathbf{P})(w) \mathbf{p}^{\prime}[t](g)(w)  dw }{dt} \mid_{t=0}.	
	\end{align*}
	
	Since $\mathbf{p}[t](g)$ belongs to the model, $\mathbf{p}^{\prime}[t](g)$ is continuous.  We now show that\begin{align*}
	\lim_{w \rightarrow \mp \infty} \mu(w) h(\mathbf{P})(w) \mathbf{p}[t](g)(w) dw = 0.
	\end{align*}
	Since $\mathbf{p}[t](g)$ is a density, it is enough to show that $||\mu h(\mathbf{P})||_{L^{\infty}(\mathbf{P}_{W})} < \infty$, but this follows by the fact that $\mu$ and $\mathbf{p}_{W}$ are uniformly bounded and by definition of $\mathbb{H}$. 	By this result and integration by parts,  it follows that
	\begin{align*}
	\frac{d \int \ell_{\mathbf{P}[t](g)}(w) h_{id}(\mathbf{P})(w) dw }{dt} \mid_{t=0} = & \frac{d \int \mu^{\prime}(w) h_{id}(\mathbf{P})(w) \mathbf{p}[t](g)(w)  dw }{dt} \mid_{t=0} \\
	& - \frac{d \int \{ \mu^{\prime}(w) h_{id}(\mathbf{P})(w) + \mu(w) h^{\prime}_{id}(\mathbf{P})(w) \} \mathbf{p}[t](g)(w)  dw }{dt} \mid_{t=0} \\
	= & -\frac{d \int \mu(w) h^{\prime}_{id}(\mathbf{P})(w) \mathbf{p}[t](g)(w)  dw }{dt} \mid_{t=0}.
	\end{align*}
	
	By definition of derivative of the curve $t \mapsto \mathbf{p}[t](g)$ and the fact that $||\mu h^{\prime}_{id}(\mathbf{P}) ||_{L^{2}(Leb)} \precsim  || h^{\prime}_{id}(\mathbf{P}) ||_{L^{2}(Leb)} < \infty$, it follows by Dominated Convergence Theorem that 	
	\begin{align*}
	\frac{d \int \ell_{\mathbf{P}[t](g)}(w) h_{id}(\mathbf{P})(w) dw }{dt} \mid_{t=0} = &  \frac{d \int \mu(w) h^{\prime}_{id}(\mathbf{P})(w) \mathbf{p}[t](g)(w)  dw }{dt} \mid_{t=0} \\
	= & -\int \mu(w) h^{\prime}_{id}(\mathbf{P})(w) \int g(y,w,x) \mathbf{p}(y,w,x) dy dx  dw.
	\end{align*}

	Since $\ell_{\mathbf{P}} \in L^{2}(\mathbf{P}_{W})$, $g \mapsto \int \ell_{\mathbf{P}}(w) g(w) dw$ is a bounded linear functional; this, part 2 below and the Dominated Convergence Theorem, imply
	\begin{align*}
	\frac{d \int \ell_{\mathbf{P}}(w) h(\mathbf{P}[t](g))(w) dw }{dt} \mid_{t=0} = \int \ell_{\mathbf{P}}(w) \dot{h}_{id}(\mathbf{P})[g](w) dw.
	\end{align*}
	
	Hence,
	\begin{align*}
	\dot{\theta}(\mathbf{P})[g] =  - \int \ell_{\mathbf{P}}(w) \dot{h}_{id}(\mathbf{P})[g](w) dw  +  \int \mu(w) h^{\prime}_{id}(\mathbf{P})(w)  \int g(y,w,x) \mathbf{p}(y,w,x) dydx dw.	
	\end{align*}
	
	Note that the last term in the RHS is, by definition, $\theta(g \cdot \mathbf{P})$. Hence
	\begin{align*}
	\dot{\theta}(\mathbf{P})[g] = - \int \ell_{\mathbf{P}}(w) \dot{h}_{id}(\mathbf{P})[g](w) dw + \theta(g \cdot \mathbf{P}).
	\end{align*}
	
	\medskip
	
	\textsc{Part 2.} This part of the proof follows from applying the implicit function theorem to $(h,P) \mapsto G(h,P) \equiv \int (F_{Y|WX}(h(w) \mid w,x) - \tau) p_{W|X}(w \mid x) dw $, where $F_{Y|WX}$ is the cdf associated to the probability measure $P$. The derivatives of the mapping $G$ at $(h(\mathbf{P}),\mathbf{P})$, with direction $(\zeta,Q)$ where $ \zeta \in \mathbb{H}$ and $Q$ is a measure over $\mathbb{Z}$, are given by
	
	\begin{align*}
	\frac{dG(h(\mathbf{P}),\mathbf{P})}{dP}[Q](x) = \int \rho_{2}(y,w,h(\mathbf{P})) Q_{YW|X}(dy,dw \mid x) dydw~\forall x,
	\end{align*}
	and
	\begin{align*}
	\frac{dG(h(\mathbf{P}),\mathbf{P})}{dh}[\zeta](x) = \mathbf{T}[\zeta](x) = \mathbf{T}[\zeta_{id}](x) = \mathbf{T}_{id}[\zeta_{id}](x),~\forall x,
	\end{align*}
	where the second equality follows because any  $\zeta \in \mathbb{H}$ can be decomposed as the sum of $\zeta_{id} \in Kernel(\mathbf{T})^{\perp}$ and an element in $Kernel(\mathbf{T})$; in the third expression $\mathbf{T}_{id}$ is defined as the restriction of $\mathbf{T}$ to $Kernel(\mathbf{T})^{\perp}$. We observe that $\mathbf{T}_{id}$ is invertible (in the sense that the inverse is a linear functional); its extension to the whole of $L^{2}(Leb)$ is the generalized inverse and we denote it as $\mathbf{T}^{+}$ (see \cite{engl1996regularization} p. 33).
		
	Taking $Q$ such that $Q_{YW|X} = g \cdot \mathbf{P}_{YW|X}$ for $g \in \mathcal{T}$, it follows that $\frac{dG(h(\mathbf{P}),\mathbf{P})}{dP}[Q] = A_{\mathbf{P}}[g]$. Moreover, by Lemma \ref{lem:tangent}, any $g \in \mathcal{T}$ is such that  $A_{\mathbf{P}}[g] \in Range(\mathbf{T}_{id}) = Range(\mathbf{T}) $. Thus, it follows that by the Implicit Function Theorem that
	\begin{align*}
	\dot{h}_{id}(\mathbf{P})[g] = (\mathbf{T})^{+} [A_{\mathbf{P}}[g]] =  (\mathbf{T}^{\ast} \mathbf{T})^{+} \mathbf{T}^{\ast} [A_{\mathbf{P}}[g]]
	\end{align*}
	where the second equality follows from the results by \cite{engl1996regularization} p. 35.
\end{proof}	

\begin{proof}[\textbf{Proof of Lemma \ref{lem:proj-T}}]
	The calculations are analogous to those in \cite{SeveriniTripathi2012} so they are omitted.
\end{proof}

\section{Supplemental Material for Appendix \ref{app:conv-rate}}
\label{supp:conv-rate}

Let $\Pi_{K} \alpha_{0}$ be the projection of $\alpha_{0}$ onto $\mathcal{A}_{K}=\Theta \times \mathcal{H}_{K}$; since $\mathcal{H}_{K}$ is closed and convex (Assumption \ref{ass:reg}%
), it is well-defined.

\begin{lemma}
	\label{lem:bound-hLo} Suppose Assumption \ref{ass:reg} holds. Then
	\begin{align*}
	||E[g_{J}(Z,\alpha_{L,0})]||^{2}_{e} + \gamma_{K} Pen(\alpha_{L,0}) \precsim
	||E[g_{J}(Z,\Pi_{K}\alpha_{0})]||^{2}_{e} + \gamma_{K}
	Pen(\Pi_{K}\alpha_{0}),
	\end{align*}
	and $||\mu h^{\prime }_{L,0}||_{L^{\infty}(\mathbb{W})} \leq \gamma_{K}^{-1}
	||E[g_{J}(Z,\Pi_{K}\alpha_{0})]||^{2}_{e} + Pen(\Pi_{K}\alpha_{0})$.
\end{lemma}

\begin{proof}
	By definition of the $\alpha_{L,0}$ and the fact that $\Pi_{K} \alpha_{0} \in \mathcal{A}_{K}$, $\bar{Q}_{J}(\alpha_{L,0},\mathbf{P}) \leq \bar{Q}_{J}(\Pi_{K}\alpha_{0},\mathbf{P}) $. By Lemma \ref{lem:H-bound}, $C^{-1} I \leq H^{-1}_{J}(\alpha_{0},\mathbf{P}) \leq C I$ some $C \geq 1$, thus
	\begin{align*}
	C^{-1} ||E[g_{J}(Z,\alpha_{L,0})]||^{2}_{e} + \gamma_{K} Pen(\alpha_{L,0}) \leq C ||E[g_{J}(Z,\Pi_{K} \alpha_{0})]||^{2}_{e} + \gamma_{K} Pen(\Pi_{K}\alpha_{0})
	\end{align*}
	and the first result follows.
	
	The second result follows from the first result and Assumption \ref{ass:reg}(iii).
\end{proof}

\subsection{Supplementary Lemmas}

We present and prove a sequence of lemmas used in the proofs of the Lemmas of Section \ref{sec:conv-rate}.

\begin{lemma}
	\label{lem:gJ-bound} For any $\alpha \in \mathbb{A}$, any $A \equiv A_{y}
	\times A_{w} \times A_{x} \subseteq \mathbb{Z}$ Borel, and any $L=(J,K) \in
	\mathbb{N}^{2}$,
	\begin{align*}
	\sup_{z \in A} ||g_{J}(z,\alpha)||_{e} \leq \bar{g}_{L}(\alpha,A) \equiv
	\overline{\theta} + ||h^{\prime }||_{L^{\infty}(A_{w},\mu)} + \sup_{x \in A_{x}}
	||q^{J}(x)||_{e}.
	\end{align*}
\end{lemma}

\begin{proof}
	The result follows from the fact that,
	\begin{align*}
		\sup_{z \in A}|| g_{J}(z,\alpha)||_{e} \leq & \sup_{z \in A}|\theta - \mu(w)h'(w)| +  \sup_{z \in A} ||q^{J}(x)||_{e} \\
		\leq & \overline{\theta} +  \sup_{w} |\mu(w) h'(w)| + \sup_{x} ||q^{J}(x)||_{e}.
	\end{align*}
\end{proof}

Throughout, for any $\delta>0$, let $B(\delta) \equiv \{ \lambda \in \mathbb{%
	R}^{J+1} \colon ||\lambda||_{e} \leq \delta \}$.

\begin{lemma}
	\label{lem:Lambda-charac} Suppose Assumption \ref{ass:reg} holds. Then,
	there exists a $\eta>0$ such that for any $\epsilon>0$, any $(L=(J,K),n)$
	and any $\delta > 0$ for which
	\begin{align*}
	\delta \overline{\theta} < \eta/3,~and~ b^{\varrho}_{\varrho,J} n
	(3\delta)^{\varrho}/(\eta)^{\varrho} < \epsilon/2,~and~\delta \times
	\mho_{L,n} < \eta/3,~some~\varrho>0,
	\end{align*}
	it follows that $\mathbf{P} \left( B(\delta) \subseteq \Lambda(\alpha,P_{n})
	~\forall \alpha \in \bar{\mathcal{A}}_{L,n} \right) \geq 1-\epsilon$.
\end{lemma}

\begin{proof}
	Since $\mathcal{S} \ni 0$, there exists a $\eta>0$ for which: For any $(Z_{i})_{i=1}^{\infty}$,  if
	\begin{align*}
		\sup_{\alpha \in \bar{\mathcal{A}}_{L,n}} \sup_{z \in supp(P_{n})} \max_{\lambda \in B(\delta)} |\lambda^{T} g_{J}(z,\alpha)| \leq \eta,
	\end{align*}
	then $\lambda \in \Lambda(\alpha,P_{n})$, for all $\alpha \in \bar{\mathcal{A}}_{L,n}$.
	
	By the Cauchy-Schwarz inequality and Lemma \ref{lem:gJ-bound},
	\begin{align*}
		\sup_{\alpha \in \bar{\mathcal{A}}_{k}}  \sup_{z \in supp(P_{n})} ||\lambda||_{e}|| g_{J}(z,\alpha)||_{e} \leq \delta \times \sup_{\alpha \in \bar{\mathcal{A}}_{L,n}}  \bar{g}_{L}(\alpha,supp(P_{n})).
	\end{align*}Thus, it suffices to show that $P( \delta \times  \sup_{\alpha \in \bar{\mathcal{A}}_{L,n}} \bar{g}_{L}(\alpha,supp(P_{n})) \leq \eta ) \geq 1-\epsilon$ (for the $\epsilon$ in the statement of the lemma).
	
	To show this we bound each term of $\sup_{\alpha \in \bar{\mathcal{A}}_{L,n}} \bar{g}_{L}(\alpha,supp(P_{n}))$. Note that $\delta \overline{\theta} < \eta/3$; also note that by the Markov inequality
	\begin{align*}
		P\left( \delta \max_{1 \leq i \leq n} ||q^{J}(X_{i})||_{e} \geq \eta/3  \right) \leq n E_{P}[||q^{J}(X)||^{\varrho}_{e}]/(\eta)^{\varrho} \times (3\delta)^{\varrho} \leq b^{\varrho}_{\varrho,J} n (3\delta)^{\varrho}/(\eta)^{\varrho},
	\end{align*}
	it suffices that the RHS is less than $\epsilon/2$, which it does by assumption.
	
	Finally, by definition of $\bar{\mathcal{A}}_{L,n}$ and Assumption \ref{ass:reg}, $|| h^{\prime} ||_{L^{\infty}(\mathbb{W},\mu)} \leq \mho_{L,n} $, hence
	\begin{align*}
		\sup_{\alpha \in \bar{\mathcal{A}}_{L,n}}	|| h'||_{L^{\infty}(supp(P_{n}),\mu)} \leq l_{n} \gamma_{K}^{-1}  \Gamma_{L,n} = \mho_{L,n},~a.s.
	\end{align*}
	Since by assumption $\delta \times  l_{n} \gamma_{K}^{-1} ( \Gamma_{L,n} ) < \eta/3$,
	
\end{proof}

\begin{lemma}
	\label{lem:sJ-bound-alpha0L} Suppose Assumptions \ref{ass:pdf0} and \ref%
	{ass:reg} hold. For any $(L=(J,K),n)$ and any positive real-valued sequence, $(\delta_{n})_{n}$, such that
	
	\begin{enumerate}
		\item $b^{4}_{4,J}/n = o(1)$.
		
		\item $\delta_{n} = o(1)$ and $b_{\varrho,J}^{\varrho} n
		\delta_{n}^{\varrho} = o(1)$ and $\delta_{n} \mho_{L,n} = o(1)$ for some $%
		\varrho > 0$.
		
		\item $\sqrt{ \frac{\bar{g}^{2}_{L,0}}{n} +
			||E[g_{J}(Z,\alpha_{L,0})]||^{2}_{e} } = o(\delta_{n})$.
	\end{enumerate}
	
	it follows that the conditions of Lemma \ref{lem:SJ-bound} hold for $%
	(\alpha_{L,0},P_{n})$ wpa1.
\end{lemma}

\begin{proof}
	
	To show condition (1) in Lemma \ref{lem:SJ-bound} it suffices to show that
	\begin{align*}
		\sup_{\lambda \in B(\delta_{n})} \sup_{z \in \mathbb{Z}}\lambda^{T}g_{J}(z,\alpha_{L,0})
	\end{align*} is bounded. But this follows from the proof of Lemma \ref{lem:Lambda-charac}, the fact that $\alpha_{L,0} \in \bar{\mathcal{A}}_{L,n}$ and the fact that all the conditions in that Lemma are satisfied.

	Condition (2) in Lemma \ref{lem:SJ-bound} holds wpa1 for some $C$ by Lemma \ref{lem:H-bound} applied to $\alpha = \alpha_{L,0}$ and $\mathbf{P}$, provided that  $b^{4}_{4,J}/n = o(1)$ (note that $|\theta| + || h'_{L,0}||_{L^{\infty}(\mathbb{W},\mu)} $ is bounded).
	
	By Lemma \ref{lem:avg-gJ}, for condition (3) in Lemma \ref{lem:SJ-bound} it is enough that
	\begin{align*}
		\sqrt{\frac{\overline{\theta} + ||\mu h'_{L,0}||_{L^{2}(\mathbf{P})} + b_{2,J}^{2}}{n} + ||E[g_{J}(Z,\alpha_{L,0})]||^{2}_{e}}  = \sqrt{ \frac{\bar{g}^{2}_{L,0}}{n}  + ||E[g_{J}(Z,\alpha_{L,0})]||^{2}_{e} } = o(\delta_{n})
	\end{align*}  which holds by assumption.
	
	Finally, $\delta_{n} = o(1)$ and $b_{\rho,J}^{\rho} n \delta_{n}^{\rho} = o(1)$ and $\delta_{n} l_{n} \gamma_{K}^{-1} \Gamma_{L,n} = \delta_{n} \mho_{L,n} = o(1)$, so condition (4) in Lemma \ref{lem:SJ-bound} holds for any $\epsilon>0$.
\end{proof}

\begin{remark}
	We note that the requirement of $\delta_{n} \times l_{n} \gamma_{K}^{-1}
	\Gamma_{L,n} = \delta_{n} \times \mho_{L,n} = o(1)$ (condition 2), does not
	contradict condition 3. Even though $\Gamma_{L,n}$ contains $\sqrt{ \frac{%
			\bar{g}^{2}_{L,0}}{n} + ||E[g_{J}(Z,\alpha_{L,0})]||^{2}_{e} }$ in condition
	2, there is the $l_{n}\gamma_{K}^{-1}$ term which is large, at least for
	sufficiently large $K$. $\triangle$
\end{remark}

\begin{lemma}
	\label{lem:HJhat-bound} Suppose Assumption \ref{ass:reg} holds. Then, for
	any $(L=(J,K),n)$ such that
	\begin{align}  \label{eqn:cond-HJhat-bound}
	(1 + \overline{\theta} + \mho_{L,n})^{2} \sqrt{\frac{ 1 + b_{4,J}^{4} } {n}} =
	o(1),
	\end{align}
	it follows that
	\begin{align*}
	\mathbf{P} \left( e_{max}\left( H_{J}(\hat{\alpha}_{L,n},P_{n}) \right) \leq
	C_{L,n} \right) \rightarrow 1
	\end{align*}
	where $C_{L,n} \equiv (1 + \overline{\theta} + \mho_{L,n})^{2} $.
\end{lemma}

\begin{proof}
	\textsc{Step 1.} We show that 	
	\begin{align*}
	n^{-1} \sum_{i=1}^{n} g_{J}(Z_{i},\alpha)g_{J}(Z_{i},\alpha)^{T} \leq  n^{-1}\sum_{i=1}^{n} ||\rho(Y_{i},W_{i},\alpha)||^{2}_{e} M_{J}(X_{i})M_{J}(X_{i})^{T},
	\end{align*}
	where
	\begin{align*}
	M_{J}(x) = & \left[
	\begin{array}{cc}
	1 & 0 \\
	\mathbf{0} & q^{J}(x)
	\end{array}
	\right]
	\end{align*}
	$\mathbf{0}$ is a $J \times 1$ vector of zeros.	
	
	To do this, note that for each $z \in \mathbb{Z}$,
	\begin{align*}
	g_{J}(z,\alpha)g_{J}(z,\alpha)^{T} = & \left[
	\begin{array}{cc}
	(\rho_{1}(y,w,\alpha))^{2} & \rho_{1}(y,w,\alpha)\rho_{2}(y,w,\alpha)q^{J}(x)^{T} \\
	\rho_{1}(y,w,\alpha)\rho_{2}(y,w,\alpha)q^{J}(x) & (\rho_{2}(y,w,\alpha))^{2}q^{J}(x)q^{J}(x)^{T}
	\end{array}
	\right] \\
	\equiv & M_{J}(x) \rho(y,w,\alpha)\rho(y,w,\alpha)^{T} M_{J}(x)^{T}.
	\end{align*}
	
	It follows that $\rho(y,w,\alpha)\rho(y,w,\alpha)^{T} \leq ||\rho(y,w,\alpha)||^{2}_{e} I$.
	Therefore
	\begin{align*}
	n^{-1} \sum_{i=1}^{n} g_{J}(Z_{i},\alpha)g_{J}(Z_{i},\alpha)^{T} \leq & n^{-1}\sum_{i=1}^{n} ||\rho(Y_{i},W_{i},\alpha)||^{2}_{e} M_{J}(X_{i})M_{J}(X_{i})^{T}.
	\end{align*}
	
	\medskip
	
	\textsc{Step 2.} Let $(y,w) \mapsto R(y,w) \equiv \sup_{\alpha \in \bar{\mathcal{A}}_{L,n}} ||\rho(y,w,\alpha)||^{2}_{e}$. We now show that
	\begin{align*}
	 & \left \Vert n^{-1}\sum_{i=1}^{n} R(Y_{i},W_{i}) M_{J}(X_{i})M_{J}(X_{i})^{T} - E\left[  R(Y,W) M_{J}(X)M_{J}(X)^{T}  \right] \right\Vert_{e} \\
	= &  O_{\mathbf{P}} \left( \sqrt{\frac{ E\left[ E[(R(Y,W))^{2}|X] (1+||q^{J}(x)||^{2}_{e})^{2} \right] } {n}  }  \right).
	\end{align*}
	
	By the Markov inequality it is enough to bound
	\begin{align*}
	n^{-1/2} \sqrt{E \left[   \left \Vert  R(Y,W) M_{J}(X)M_{J}(X)^{T} \right\Vert^{2}_{e} \right] } \leq n^{-1/2} \sqrt{ E\left[ (R(Y,W))^{2} (trace\{ M_{J}(X)M_{J}(X)^{T}  \})^{2}  \right]   }.
	\end{align*}
	Since
	\begin{align*}
	M_{J}(x)M_{J}(x)^{T} = & \left[
	\begin{array}{cc}
	1 & \mathbf{0}^{T} \\
	\mathbf{0} & q^{J}(x)q^{J}(x)^{T}
	\end{array}
	\right],
	\end{align*}
	the previous display implies that
	\begin{align*}
	n^{-1/2} \sqrt{E \left[   \left \Vert  R(Y,W) M_{J}(X)M_{J}(X)^{T} \right\Vert^{2}_{e} \right] } \leq n^{-1/2} \sqrt{ E\left[ E[(R(Y,W))^{2}|X] (1+||q^{J}(x)||^{2}_{e})^{2} \right]   }.
	\end{align*}	
	
	\medskip
	
	\textsc{Step 3.} By Steps 1-2 and Assumption \ref{ass:reg}
	\begin{align*}
	n^{-1} \sum_{i=1}^{n} g_{J}(Z_{i},\alpha)g_{J}(Z_{i},\alpha)^{T} \leq & \sup_{x} E[R(Y,W) \mid X = x] \times I \\
	& +  O_{\mathbf{P}} \left( \sqrt{\sup_{x}  E[(R(Y,W))^{2}|X=x]  \frac{ 1 + E\left[||q^{J}(X)||^{4}_{e} \right] } {n}  }  \right).
	\end{align*}
	
	Also,
	\begin{align*}
	R(y,w)  \leq  \sup_{\alpha \in \bar{\mathcal{A}}_{L,n}} |\theta - \mu(w)h'(w)|^{2} + 1  	\leq  1 + \overline{\theta}^{2} + (l_{n} \gamma_{K}^{-1} \Gamma_{L,n})^{2}
	\end{align*}
	where the last inequality follows from Assumption \ref{ass:reg} and definition of $\bar{\mathcal{A}}_{L,n}$. So, to the extent that
	\begin{align*}
	(1 + \overline{\theta} + l_{n} \gamma_{K}^{-1} \Gamma_{L,n})^{4} \frac{ 1 + b_{4,J}^{4} } {n} = o(1),
	\end{align*}
	it follows that\begin{align*}
	\sup_{\alpha \in \bar{\mathcal{A}}_{L,n}} e_{max} \left( n^{-1} \sum_{i=1}^{n} g_{J}(Z_{i},\alpha)g_{J}(Z_{i},\alpha)^{T}   \right) \leq (1 + \overline{\theta} + l_{n} \gamma_{K}^{-1} \Gamma_{L,n})^{2}
	\end{align*}
	wpa1. Since $\hat{\alpha}_{L,n} \in \bar{\mathcal{A}}_{L,n}$ wpa1, this implies the result.
\end{proof}

\subsection{Proofs of the Lemmas stated in Appendix \ref{app:conv-rate}}

We now present the proofs of the Lemmas stated in Appendix \ref{app:conv-rate}.

\begin{proof}[\textbf{Proof of Lemma \ref{lem:SJ-bound}}]
	By the mean value theorem,
	\begin{align*}
	S_{J}(\alpha,\lambda,P)  = & s'(0) \lambda^{T} E_{P}[g_{J}(Z,\alpha)] + \frac{1}{2} \lambda^{T} \left\{ \int_{0}^{1} E_{P} \left[ s'' \left( t  \lambda^{T} g_{J}(Z,\alpha)  \right) g_{J}(Z,\alpha) g_{J}(Z,\alpha)^{T}  \right] dt  \right\} \lambda \\
	= & - \lambda^{T} E_{P}[g_{J}(Z,\alpha)] + \frac{1}{2} \lambda^{T} \left\{ \int_{0}^{1} E_{P} \left[ s'' \left( t  \lambda^{T} g_{J}(Z,\alpha)  \right) g_{J}(Z,\alpha) g_{J}(Z,\alpha)^{T}  \right] dt  \right\} \lambda.
	\end{align*}
	Note that
	\begin{align*}
	\int_{0}^{1} E_{P} \left[ s'' \left( t  \lambda^{T} g_{J}(Z,\alpha)  \right) g_{J}(Z,\alpha) g_{J}(Z,\alpha)^{T}  \right] dt \leq \int_{0}^{1} \sup_{z} s'' \left( t  \lambda^{T} g_{J}(z,\alpha)  \right) dt E_{P} \left[ g_{J}(Z,\alpha) g_{J}(Z,\alpha)^{T}  \right].
	\end{align*}
	
	Under our assumptions, $\sup_{z} s'' \left( t  \lambda^{T} g_{J}(z,\alpha)  \right) \leq - \sqrt{C}$ for all $t \in [0,1]$. This and the fact that\begin{align*}
	e_{\min} \left(  E_{P} \left[ g_{J}(Z,\alpha) g_{J}(Z,\alpha)^{T}  \right] \right) \geq \sqrt{C}
	\end{align*}imply that
	\begin{align*}
	S_{J}(\alpha,\lambda,P) \leq - \lambda^{T} E_{P}[g_{J}(Z,\alpha)] - \frac{C}{2} ||\lambda||^{2}_{e} \leq ||\lambda||_{e} ||E_{P}[g_{J}(Z,\alpha)]||_{e} - \frac{C}{2} ||\lambda||^{2}_{e}.
	\end{align*}
	Hence, since $S_{J}(\alpha,0,P) = 0$, by evaluating the previous expression in $\lambda_{1} \in \arg\max_{\lambda \in B(\delta)} S_{J}(\alpha,\lambda,P) $ (it exists by continuity of $S_{J}(\alpha,.,P)$ and compactness of the set), we obtain
	\begin{align*}
	||\lambda_{1}||_{e} \leq 2 C^{-1} ||E_{P}[g_{J}(Z,\alpha)]||_{e}.
	\end{align*}
	
	Since $2 C^{-1} ||E_{P}[g_{J}(Z,\alpha)]||_{e}  < \delta$ by assumption, $\lambda_{1}$ is in the interior of $B(\delta)$. Since, by Lemma \ref{lem:Lambda-charac} --- all the hypothesis of the lemma are satisfied --- $\lambda_{1} \in \Lambda_{J}(\alpha,P)$ with probability higher than $1-\epsilon$, it holds that $\lambda_{1} \in \arg\max_{\lambda \in \Lambda_{J}(\alpha,P)} S_{J}(\alpha,\lambda,P)$ with probability higher than $1-\epsilon$.
	
	Finally, this implies that
	\begin{align*}
	\sup_{\lambda \in \Lambda_{J}(\alpha,P)} S_{J}(\alpha,\lambda,P) = S_{J}(\alpha,\lambda_{1},P) \leq 2 C^{-1} ||E_{P}[g_{J}(Z,\alpha)]||^{2}_{e},
	\end{align*}
	with probability higher than $1-\epsilon$.
\end{proof}

\begin{proof}[\textbf{Proof of Lemma \ref{lem:avg-gJ}}]
	It suffices to show that
	\begin{align*}
		||n^{-1}\sum_{i=1}^{n} g_{J}(Z_{i},\alpha) - E[g_{J}(Z,\alpha)]||^{2}_{e} = O_{\mathbf{P}}\left( \frac{\overline{\theta} + ||\mu h'||^{2}_{L^{2}(\mathbf{P})} + E[||q^{J}(X)||^{2}_{e}] }{n}  \right).
	\end{align*}
	By the Markov inequality, we can study
	\begin{align*}
	E[||n^{-1}\sum_{i=1}^{n} g_{J}(Z_{i},\alpha) - E[g_{J}(Z,\alpha)]||^{2}_{e}] \leq  n^{-1} E[||g_{J}(Z,\alpha)||^{2}_{e}],
	\end{align*}
	and it follows by Lemma \ref{lem:moment-UBound}, $E[||g_{J}(Z,\alpha)||^{2}_{e}] \leq \overline{\theta} + E[|\mu(W) h'(W)|^{2}] + E[||q^{J}(X)||^{2}_{e}]$.
	\end{proof}

\begin{proof}[\textbf{Proof of Lemma \ref{lem:H-bound}}]
	\textsc{Part 1.} By analogous calculations to those  in the proof of Lemma \ref{lem:HJhat-bound} and Lemma A.6 in DIN, $||H_{J}(\alpha,P_{n}) - H_{J}(\alpha,\mathbf{P})||_{e} = O_{\mathbf{P}}(\sqrt{\sup_{x} E[||\rho(Y,W,\alpha)||^{4}_{e}|X=x] b^{4}_{4,J}/n})$. Note that,
	\begin{align*}
		E[||\rho(Y,W,\alpha)||^{4}_{e}|X=x] \leq  E\left[ (1 + |\theta - \mu(W)h'(W)|^{2})^{2}   | X=x \right] \precsim &  \theta^{4} +  || \mu h'||^{4}_{L^{4}(\mathbf{P}_{W})}\\
		\precsim &  \theta^{4} +  || h'||^{4}_{L^{\infty}(\mathbb{W},\mu)}
	\end{align*}
	where the last inequality follows from Assumption \ref{ass:pdf0} and some trivial algebra.
	
	\medskip

	\textsc{Part 2.} From Part 1 and the fact that $A \mapsto e_{min}(A)$ is Lipschitz, is suffices to show the result for $P=\mathbf{P}$. Note that
	\begin{align*}
		E_{\mathbf{P}}[g_{J}(Z,\alpha)g_{J}(Z,\alpha)^{T}] = E[M_{J}(X) E[\rho(Y,W,\alpha)\rho(Y,W,\alpha)^{T}|X]M_{J}(X)^{T}]
	\end{align*}
	where $M_{J}$ is defined in the proof of Lemma \ref{lem:HJhat-bound}.

	We now argue that $C^{-1} I \leq E[\rho(Y,W,\alpha)\rho(Y,W,\alpha)^{T}|X] \leq C I$ for any $\alpha$ in some neighborhood of $\alpha_{0}$. First note that, under Assumption \ref{ass:pdf0}, $E[(\rho_{1}(Y,W,\alpha))^{2} \mid X] = \int (\theta_{0} - \mu(w)h'(w))^{2} \mathbf{p}_{W|X}(w|x)dw \geq C Var_{\mathbf{P}}(\mu(W) h'(W)) > 0 $ some $C>0$. Hence the trace of $E[\rho(Y,W,\alpha)\rho(Y,W,\alpha)^{T}|X]$ is positive uniformly in $x$. Since $\rho_{1}(.,\alpha)$ and $\rho_{2}(.,\alpha)$ are not linearly dependent, by the Cauchy-Schwarz (strict) inequality the determinant is also positive; thus $C^{-1} I \leq E[\rho(Y,W,\alpha)\rho(Y,W,\alpha)^{T}|X]$ uniformly on $X$. The reverse inequality is obtained in a similar fashion.

	Hence under Assumption \ref{ass:reg} the desired result follows.   	
\end{proof}

\begin{proof}[\textbf{Proof of Lemma \ref{lem:norm-g-rate}}]
	By the mean value theorem, for any $\alpha \in \mathcal{A}$, $P$ and $\lambda$ in a neighborhood of 0,
	\begin{align*}
	S_{J}(\alpha,\lambda,P)  = & s'(0) \lambda^{T} E_{P}[g_{J}(Z,\alpha)] + \frac{1}{2} \lambda^{T} \left\{ \int_{0}^{1} E_{P} \left[ s'' \left( t  \lambda^{T} g_{J}(Z,\alpha)  \right) g_{J}(Z,\alpha) g_{J}(Z,\alpha)^{T}  \right] dt  \right\} \lambda \\
	= & - \lambda^{T} E_{P}[g_{J}(Z,\alpha)] + \frac{1}{2} \lambda^{T} \left\{ \int_{0}^{1} E_{P} \left[ s'' \left( t  \lambda^{T} g_{J}(Z,\alpha)  \right) g_{J}(Z,\alpha) g_{J}(Z,\alpha)^{T}  \right] dt  \right\} \lambda.
	\end{align*}
	
	By our assumptions over $s$ (in particular, that it has H\"{o}lder continuous second derivative), it follows 	
	\begin{align*}
	|s''(\lambda^{T}g_{J}(z,\alpha)) - s''(0)| \leq C \times |\lambda^{T}g_{J}(z,\alpha)|
	\end{align*}
	for some $C>0$.	Hence
	\begin{align*}
	& \left| \lambda^{T} \left\{ \int_{0}^{1} E_{P} \left[ s'' \left( t  \lambda^{T} g_{J}(Z,\alpha)  \right) g_{J}(Z,\alpha) g_{J}(Z,\alpha)^{T}  \right] dt  \right\} \lambda - \lambda^{T} \left\{ E_{P} \left[ s'' \left(0\right) g_{J}(Z,\alpha) g_{J}(Z,\alpha)^{T}  \right]  \right\} \lambda \right| \\
	\leq &  \left|  \int_{0}^{1} E_{P} \left[ (s'' \left( t  \lambda^{T} g_{J}(Z,\alpha)  - s''(0)  \right) \lambda^{T}g_{J}(Z,\alpha) g_{J}(Z,\alpha)^{T} \lambda  \right] dt \right|\\
	\leq & C \times E_{P} \left[ |\lambda^{T}g_{J}(Z,\alpha)| \lambda^{T}g_{J}(Z,\alpha) g_{J}(Z,\alpha)^{T} \lambda  \right] \\
	\leq & C ||\lambda||_{e}^{3}  E_{P} \left[ ||g_{J}(Z,\alpha)||^{3}_{e} \right].
	\end{align*}
	
	Therefore, for $\alpha = \hat{\alpha}_{L,n}$ and $P=P_{n}$,
	\begin{align*}
	S_{J}(\alpha,\lambda,P)  \geq &  - \lambda^{T} E_{P}[g_{J}(Z,\alpha)] - \frac{1}{2} \lambda^{T} H_{J}(\alpha,P) \lambda - C ||\lambda||_{e}^{3} \times E_{P} \left[ ||g_{J}(Z,\alpha)||^{3}_{e} \right] \\
	\geq & - \lambda^{T} E_{P}[g_{J}(Z,\alpha)] - \frac{C \times C_{L,n}}{2} ||\lambda||^{2}_{e} - C ||\lambda||_{e}^{3} \times E_{P} \left[ ||g_{J}(Z,\alpha)||^{3}_{e} \right]
	\end{align*}	
	where the second inequality follows from Lemma \ref{lem:HJhat-bound}.
	
	Let $\bar{\lambda} \equiv - \delta_{n} E_{P_{n}}[g_{J}(Z,\hat{\alpha}_{L,n})]/||E_{P_{n}}[g_{J}(Z,\hat{\alpha}_{L,n})]||_{e}$ with $\delta_{n}$ satisfying the assumptions in the statement of the Lemma. Since $||\bar{\lambda}||_{e} = \delta_{n}$ and $\hat{\alpha}_{L,n} \in \bar{\mathcal{A}}_{L,n}$ wpa1 (Lemma \ref{lem:eff-sieve}), by Lemma \ref{lem:Lambda-charac}, this choice of $\delta_{n}$ ensures that  $\bar{\lambda} \in \Lambda_{J}(\hat{\alpha}_{L,n},P_{n})$ wpa1.
	
	The previous expression for $S_{J}(\alpha,\lambda,P)$ yields
	\begin{align*}
	S_{J}(\hat{\alpha}_{L,n},\bar{\lambda},P_{n})  \geq \delta_{n} ||E_{P_{n}}[g_{J}(Z,\hat{\alpha}_{L,n})]||_{e} - \frac{C \times C_{L,n}}{2} \delta^{2}_{n} - C \delta^{3}_{n} \times E_{P_{n}} \left[ ||g_{J}(Z,\hat{\alpha}_{L,n})||^{3}_{e} \right].
	\end{align*}
By definition of $\hat{\alpha}_{L,n}$, $$S_{J}(\hat{\alpha}_{L,n},\bar{\lambda},P_{n}) + \gamma_{K} Pen(\hat{\alpha}_{L,n}) \leq \sup_{\lambda \in \Lambda_{J}(\alpha_{L,0},P_{n})} S_{J}(\alpha_{L,0},\lambda,P_{n}) + \gamma_{K} Pen(\alpha_{L,0})~~ wpa1.$$
	
	By Lemma \ref{lem:sJ-bound-alpha0L}, all the conditions in Lemma \ref{lem:SJ-bound} hold for $(\alpha_{L,0},P_{n})$ wpa1. In this case, the previous inequality implies that $S_{J}(\hat{\alpha}_{L,n},\bar{\lambda},P_{n}) + \gamma_{K} Pen(\hat{\alpha}_{L,n}) \precsim ||E_{P_{n}}[g_{J}(Z,\alpha_{L,0})]||^{2}_{e} + \gamma_{K} Pen(\alpha_{L,0})$ wpa1. Therefore
	\begin{align*}
& \delta_{n} ||E_{P_{n}}[g_{J}(Z,\hat{\alpha}_{L,n})]||_{e} - \frac{C \times C_{L,n}}{2} \delta^{2}_{n} - C \delta^{3}_{n} \times  E_{P_{n}} \left[ ||g_{J}(Z,\hat{\alpha}_{L,n})||^{3}_{e} \right]\\
& \precsim ||E_{P_{n}}[g_{J}(Z,\alpha_{L,0})]||^{2}_{e} + \gamma_{K} Pen(\alpha_{L,0}),~~wpa1.
	\end{align*}
Since $\hat{\alpha}_{L,n} \in \bar{\mathcal{A}}_{L,n}$ wpa1 (Lemma \ref{lem:eff-sieve}), it follows by the proof of Lemma \ref{lem:Lambda-charac} that $||g_{J}(Z,\hat{\alpha}_{L,n})||^{3}_{e} \leq (\overline{\theta} + l_{n} \gamma_{K}^{-1} \Gamma_{L,n} + ||q^{J}(X)||_{e} )^{3}$. We now show that $\delta^{3}_{n} \left\{ (\overline{\theta} + l_{n} \gamma_{K}^{-1} \Gamma_{L,n})^{3} + E_{P_{n}}[||q^{J}(X)||_{e}^{3}]) \right\} = o_{\mathbf{P}}(\delta^{2}_{n}C_{L,n})$, so that the term $ C \delta^{3}_{n} \times  E_{P_{n}} \left[ ||g_{J}(Z,\hat{\alpha}_{L,n})||^{3}_{e} \right] $ can be ignored.
	
	To show this, note by definition of $C_{L,n}$, it suffices to show $\delta^{3}_{n} \left\{ (C_{L,n})^{2} + E_{P_{n}}[||q^{J}(X)||_{e}^{3}]) \right\} = o_{\mathbf{P}}(\delta^{2}_{n}C_{L,n})$. 	By the conditions in the lemma $\delta_{n}C_{L,n} = o(1)$ and taking $C_{L,n} \geq 1$ (if this is not the case, the solution is trivial), and thus $\delta^{3}_{n}  (C_{L,n})^{2} = o(\delta^{2}_{n}C_{L,n})$. By the Markov inequality, it remains to show that $\delta_{n} b_{3,J}^{3} = o_{\mathbf{P}}(C_{L,n})$. As we take $C_{L,n} \geq 1$ and $\delta_{n} b_{3,J}^{3} = o(1)$, the desired equality holds. Hence
	\begin{align*}
	||E_{P_{n}}[g_{J}(Z,\hat{\alpha}_{L,n})]||_{e}  \precsim C_{L,n} \delta_{n} +  \delta^{-1}_{n} \left\{ ||E_{P_{n}}[g_{J}(Z,\alpha_{L,0})]||^{2}_{e} + \gamma_{K} Pen(\alpha_{L,0}) \right\},~~wpa1.
	\end{align*}
\end{proof}	

\begin{proof}[\textbf{Proof of Lemma \ref{lem:moment-UBound}}]
	By definition of $g_{J}$ and simple algebra, it follows that
	\begin{align*}
	||E_{P}[g_{J}(Z,\alpha)]||_{e}  \leq \overline{\theta} + \sup_{w}  |h^{\prime}(w) \mu (w)|  +  E_{P}[||q^{J}(X)||_{e}].
	\end{align*}
	By the fact that $\alpha \in \bar{\mathcal{A}}_{L,n}$, $Pen(\alpha) \leq \mho_{L,n} = l_{n} \gamma^{-1}_{K} \Gamma_{L,n} $. So, by Assumption \ref{ass:reg}(iii) it follows that $\sup_{w}  |h^{\prime}(w) \mu (w)|  \leq \mho_{L,n} = l_{n} \gamma^{-1}_{K} \Gamma_{L,n} $, uniformly on $h$.
\end{proof}

\section{Supplemental Material for Appendix \ref{app:ADT}}
\label{supp:ADT}

\begin{proof}[\textbf{Proof of Lemma \ref{lem:weak-norm}}]
	Recall that $|| \cdot ||^{2}_{w} 	\equiv (G(\alpha_{L,0})[\cdot])^{T} H_{L}^{-1} (G(\alpha_{L,0})[\cdot])$. Since $\alpha \in lin\{ \mathcal{A}_{K} \}$, we can cast $\alpha = (\theta,\varphi^{K}(\cdot)^{T}\pi) $ some $\pi \in \mathbb{R}^{K}$. Hence
	
	\begin{align*}
	G(\alpha_{L,0})[\alpha] = & \left[
	\begin{array}{c}
	\theta - E[\ell(W)\varphi^{K}(W)^{T}\pi] \\
	E[\mathbf{p}_{Y|WX}(h_{L,0}(W) | W,X) q^{J}(X)\varphi^{K}(W)^{T}\pi  ]
	\end{array}
	\right]\\
	= & M_{L} \times (\theta,\pi)^{T}.
	\end{align*}
	Hence
	\begin{align*}
	|| \alpha ||^{2}_{w} = (\theta,\pi) M_{L}^{T} H_{L}^{-1} M_{L}  (\theta,\pi)^{T}.
	\end{align*}
	Thus, it is sufficient to show that $M_{L}^{T} H_{L}^{-1} M_{L}$ is positive definite. By Lemma \ref{lem:H-bound} --- since $\{ \theta^{2}_{L,0} + ||\mu h'_{L,0}||^{2}_{L^{\infty}(\mathbb{W})} \} \sqrt{b^{4}_{4,J}/n} = o(1)$ --- it follows that $e_{min}(H_{L}^{-1}) \geq c^{2}> 0$, Hence,
	\begin{align*}
	|| \alpha ||^{2}_{w} \geq c^{2} \times e_{min}(M_{L}^{T} M_{L}) ||\theta,\pi||^{2}_{e}.
	\end{align*}
	Note that the eigenvalues of $M_{L}^{T} M_{L}$ are those of
	\begin{align*}
	(E[\mathbf{p}_{Y|WX}(h_{L,0}(W) | W,X) q^{J}(X)\varphi^{K}(W)^{T}])^{T}(E[\mathbf{p}_{Y|WX}(h_{L,0}(W) | W,X) q^{J}(X)\varphi^{K}(W)^{T}])
	\end{align*}
	and $1$. By Assumption \ref{ass:rates-LQA}(iv) $E[\mathbf{p}_{Y|WX}(h_{L,0}(W) | W,X) q^{J}(X)\varphi^{K}(W)^{T}]$ has full rank and thus the eigenvalues are positive.
\end{proof}

\begin{proof}[\textbf{Proof of Lemma \ref{lem:HJhat-LQA}}]
	Note that
	\begin{align*}
	\left \Vert H_{J}(\alpha,P_{n})  - H_{J} \right \Vert_{e} \leq  & 	\left \Vert H_{J}(\alpha,P_{n}) - H_{J}(\alpha_{L,0},P_{n}) - \{ H_{J}(\alpha,\mathbf{P}) - H_{J}(\alpha_{L,0},\mathbf{P}) \} \right \Vert_{e} \\
	& + \left \Vert H_{J}(\alpha_{L,0},P_{n}) - H_{J}(\alpha_{L,0},\mathbf{P})  + H_{J}(\alpha,\mathbf{P}) - H_{J} \right \Vert_{e} \\
	\leq & \left \Vert H_{J}(\alpha,P_{n}) - H_{J}(\alpha_{L,0},P_{n}) - \{ H_{J}(\alpha,\mathbf{P}) - H_{J}(\alpha_{L,0},\mathbf{P}) \} \right \Vert_{e} \\
	& + \left \Vert H_{J}(\alpha_{L,0},P_{n}) - H_{J}(\alpha_{L,0},\mathbf{P}) \right \Vert_{e} \\
	& + \left \Vert H_{J}(\alpha,\mathbf{P}) - H_{J} \right \Vert_{e} \\
	\equiv & Term_{1,n} + Term_{2,n} + Term_{3,n}.
	\end{align*}
	
	The term $Term_{1,n}$ is controlled by Assumption \ref{ass:HJ-sec}. By Lemma \ref{lem:H-bound} applied to $\alpha_{L,0}$, $Term_{2,n} = O_{\mathbf{P}} \left( \{ \theta_{L,0} + || h'_{L,0}||_{L^{\infty}(\mathbb{W},\mu)}   \}^{2} \sqrt{b_{4,J}/n} \right)$.
	
	Note that\begin{align*}
	H_{J}(\alpha,\mathbf{P}) - H_{J} = & E_{\mathbf{P}} \left[ M_{J}(X) \{ \rho(Y,W,\alpha)\rho(Y,W,\alpha)^{T} - \rho(Y,W,\alpha_{L,0})\rho(Y,W,\alpha_{L,0})^{T}  \} M_{J}(X)^{T}   \right] \\
	= & E_{\mathbf{P}} \left[ M_{J}(X) E[\{ \rho(Y,W,\alpha)\rho(Y,W,\alpha)^{T} - \rho(Y,W,\alpha_{L,0})\rho(Y,W,\alpha_{L,0})^{T}  \}|X] M_{J}(X)^{T}   \right] \\
	\leq & \sup_{x} |trace\{ E_{\mathbf{P}}[\rho(Y,W,\alpha)\rho(Y,W,\alpha)^{T}  -  \rho(Y,W,\alpha_{L,0})\rho(Y,W,\alpha_{L,0})^{T}   |X=x]  \}| \\
	& \times E_{\mathbf{P}} \left[ M_{J}(X) M_{J}(X)^{T}   \right].
	\end{align*}
	Moreover,
	\begin{align*}
	& |trace\{ E_{\mathbf{P}}[\rho(Y,W,\alpha)\rho(Y,W,\alpha)^{T}  -  \rho(Y,W,\alpha_{L,0})\rho(Y,W,\alpha_{L,0})^{T}   |X=x]  \}| \\
	= & | E_{\mathbf{P}}[\rho_{1}(Y,W,\alpha)^{2}  -  \rho_{1}(Y,W,\alpha_{L,0})^{2}   |X=x]  |  \\
	& + | E_{\mathbf{P}}[\rho_{2}(Y,W,\alpha)^{2}  -  \rho_{2}(Y,W,\alpha_{L,0})^{2}   |X=x]  |\\
	= &  | E_{\mathbf{P}}[(\theta - \mu(W)h'(W))^{2}  -  (\theta_{L,0} - \mu(W)h_{L,0}'(W))^{2}   |X=x]  | \\
	& + | (1-2\tau)| | E_{\mathbf{P}}[ ( 1\{Y \leq h(W)\} - 1\{Y \leq h_{L,0}(W)\} ) |X=x]  | \\
	= & | E_{\mathbf{P}}[\theta^{2} - \theta_{L,0}^{2} - 2\theta \mu(W)h'(W) + (\mu(W)h'(W))^{2} - (\mu(W)h'_{L,0}(W))^{2}  + 2 \theta_{L,0}\mu(W)h_{L,0}'(W)   |X=x]  | \\
	& + | (1-2\tau)| | E_{\mathbf{P}}[ ( F_{Y|W,X}(h(W)|W,X) -  F_{Y|W,X}(h_{L,0}(W)|W,X) ) |X=x]  |.
	\end{align*}
	
	By assumption \ref{ass:pdf0},
	\begin{align*}
	| E_{\mathbf{P}}[ ( F_{Y|W,X}(h(W)|W,X) -  F_{Y|W,X}(h_{L,0}(W)|W,X) ) |X=x]  | \precsim ||h-h_{L,0}||_{L^{2}(Leb)}.
	\end{align*}
	
	Since $E[M_{J}(X)M_{J}(X)^{T}] = I$ by assumption \ref{ass:reg}, it follows that \begin{align*}
	||H_{J}(\alpha,\mathbf{P}) - H_{J}||_{e} = O\left( \mho_{L,n} \times ||\alpha - \alpha_{L,0}||   \right).
	\end{align*}	
\end{proof}

\begin{proof}[\textbf{Proof of Lemma \ref{lem:sec-gJ}}]
	Observe that
	\begin{align*}
	&\left \Vert n^{-1}\sum_{i=1}^{n} g_{J}(Z_{i},\alpha) - g_{J}(Z_{i},\alpha_{L,0})  - E[g_{J}(Z,\alpha) - g_{J}(Z,\alpha_{L,0})] \right \Vert_{e} \\
	\leq & \left | n^{-1}\sum_{i=1}^{n} \mu(W_{i})\{h'(W_{i}) - h_{L,0}'(W_{i})\}  - E[\mu(W)\{h'(W) - h_{L,0}'(W)\}] \right | \\
	& + \left \Vert n^{-1}\sum_{i=1}^{n} (1\{ Y_{i} \leq h(W_{i})  \} - 1\{ Y_{i} \leq h_{L,0}(W_{i})  \})q^{J}(X_{i})  - E[(1\{ Y \leq h(W)  \} - 1\{ Y \leq h_{L,0}(W)  \})q^{J}(X) ] \right \Vert_{e} \\
	\leq & \sqrt{\frac{J}{n} } \left( \sup_{(\theta,h) \in \mathcal{N}_{L,n}} \mathbb{G}_{n}[\mu \cdot (h' - h_{L,0}')] +  \max_{1 \leq j \leq J} \sup_{g \in \bar{\mathcal{G}}_{L,n}} \mathbb{G}_{n}[g \cdot q_{j}]   \right).
	\end{align*}
\end{proof}

\begin{proof}[\textbf{Proof of Lemma \ref{lem:charac-N}}]
	By the triangle inequality and definition of $u^{\ast}_{L,n}$, it suffices to check that $||\alpha - \alpha_{L,0}||_{w} + t \leq \eta_{w,L,n}$, which holds since $\eta_{w,L,n} \succsim l_{n}n^{-1/2}$. Regarding the latter term, by the triangle inequality it suffices to check that $||\alpha - \alpha_{L,0}|| + t ||u^{\ast}_{L,n}|| \leq \eta_{L,n}$. By Lemma \ref{lem:weak-norm} and Assumption \ref{ass:reg}, $||u^{\ast}_{L,n}|| \precsim (e_{min}(M_{L}^{T}M_{L}))^{-1}$. By Assumption \ref{ass:dev-N}(iii), $l_{n}n^{-1/2} (e_{min}(M_{L}^{T}M_{L}))^{-1} = o(\eta_{L,n})$.
	
	It follows that by Assumption \ref{ass:dev-N}, $Pen(\alpha +t u^{\ast}_{L,n}) \leq Pen(\alpha) + n^{-1}/\gamma_{K}$. Since $\alpha \in \bar{\mathcal{A}}_{L,n}$ and $\Gamma_{L,n} \succsim n^{-1}$, it follows that $\alpha + t u^{\ast}_{L,n} \in \bar{\mathcal{A}}_{L,n}$.
\end{proof}

\begin{proof}[\textbf{Proof of Lemma \ref{lem:SJ-bound-PSGEL}}]
	Throughout the proof, let $\alpha \in \{ \hat{\alpha}_{L,n},  \hat{\alpha}^{\nu}_{L,n}   \}$.
	
	By the calculations in the proof of Lemma \ref{lem:Lambda-charac} and the fact that $\alpha \in \bar{\mathcal{A}}_{L,n}$ wpa1 (for $\alpha = \hat{\alpha}^{\nu}_{L,n}$, this follows from Assumption \ref{ass:dev-N}), $||\lambda||_{e} \times ||E_{P_{n}}[g_{J}(Z,\alpha)]||_{e} \precsim  ( ||\lambda||_{e} \overline{\theta} + ||\lambda||_{e} \mho_{L,n} + b^{\varrho}_{\varrho,J} n ||\lambda||_{e}^{\varrho} )$. Since $\lambda \in B(\delta_{n})$, under Assumption \ref{ass:rates-LQA} (with $\varrho = 3$), $|\lambda^{T}g_{J}(z,\alpha)| = o_{\mathbf{P}}(1)$, since $s''$ is continuous, this implies condition 1.
	
	Regarding Condition 3, note that
	\begin{align*}
	||E_{P_{n}}[g_{J}(Z,\alpha)]||_{e} \leq & ||E_{P_{n}}[g_{J}(Z,\alpha)-g_{J}(Z,\alpha_{L,0})] - E_{\mathbf{P}}[g_{J}(Z,\alpha)-g_{J}(Z,\alpha_{L,0})]||_{e} \\
	& + ||E_{P_{n}}[g_{J}(Z,\alpha_{L,0})] ||_{e} + ||E_{\mathbf{P}}[g_{J}(Z,\alpha)-g_{J}(Z,\alpha_{L,0})]||_{e}.
	\end{align*}
	By Lemma \ref{lem:sec-gJ} and Assumption \ref{ass:Donsker-LQA}, the first term in the RHS is of order $\sqrt{J/n} \Delta_{2,J,n}$. By the proof of Lemma \ref{lem:LQA}, the third term is of order $\eta^{2}_{L,n} b_{2,J} + \eta_{w,L,n}$. Finally, by Lemma \ref{lem:avg-gJ}, the second term if of order $\sqrt{ \bar{g}_{L,0}^{2}/n + ||E_{\mathbf{P}}[g_{J}(Z,\alpha_{L,0})]||^{2}_{e}  }$. These results and assumption \ref{ass:rates-LQA} imply the condition.
	
	By the triangle inequality,
	\begin{align*}
	||E_{P_{n}}[g_{J}(Z,\alpha)]||_{e} \leq & ||E_{P_{n}}[g_{J}(Z,\alpha)] - E_{P_{n}}[g_{J}(Z,\alpha_{L,0})] - \{ E_{\mathbf{P}}[g_{J}(Z,\alpha)] - E_{\mathbf{P}}[g_{J}(Z,\alpha_{L,0})] \}||_{e} \\
	& + ||E_{P_{n}}[g_{J}(Z,\alpha_{L,0})]||_{e} + ||E_{\mathbf{P}}[g_{J}(Z,\alpha) - g_{J}(Z,\alpha_{L,0})] ||_{e}.
	\end{align*}
	The first term in the RHS is of order $O_{\mathbf{P}}(\Delta_{2,L,n})$ by Assumption \ref{ass:Donsker-LQA}. By Lemma \ref{lem:avg-gJ}, the second term is of order $O_{\mathbf{P}}( \sqrt{ \frac{\overline{\theta} + ||\mu h'_{L,0} ||^{2}_{L^{2}(\mathbf{P})} + b^{2}_{2,J}}{n}      +   ||E_{\mathbf{P}}[g_{J}(Z,\alpha_{L,0})]||^{2}_{e}      }   ) $. Finally, by the proof of Lemma \ref{lem:LQA} the third term is of order $O_{\mathbf{P}} \left( ||\alpha - \alpha_{L,0}||_{w} + ||\alpha - \alpha_{L,0}||^{2} b_{2,J}  \right)$.
	
	Thus, since $\alpha \in \mathcal{N}_{n}$ wpa1 (by Lemma \ref{lem:charac-N}),
	\begin{align*}
	||E_{P_{n}}[g_{J}(Z,\alpha)]||_{e} = &  O_{\mathbf{P}} \left( \Delta_{2,L,n} +\sqrt{ \frac{\overline{\theta} + ||\mu h'_{L,0} ||^{2}_{L^{2}(\mathbf{P})} + b^{2}_{2,J}}{n}      +   ||E_{\mathbf{P}}[g_{J}(Z,\alpha_{L,0})]||^{2}_{e}      }   + \eta_{w,L,n} + \eta_{L,n}^{2} b_{2,J} \right) \\
	= & O_{\mathbf{P}} \left( \Delta_{2,L,n} +\sqrt{ \frac{\bar{g}_{L,0}^{2} }{n}      +   ||E_{\mathbf{P}}[g_{J}(Z,\alpha_{L,0})]||^{2}_{e}      }   + \eta_{w,L,n} + \eta_{L,n}^{2} b_{2,J} \right)
	\end{align*}
	By Assumption \ref{ass:rates-LQA}(ii), this result implies Condition 4. Condition 5 holds by assumption.	
\end{proof}

\begin{proof}[\textbf{Proof of Lemma \ref{lem:lambda-maxLQA}}]
	By Lemma \ref{lem:H-bound}(i) (applied to $\alpha=\alpha_{L,0}$) it suffices to bound $||\varDelta(\alpha)||_{e}$. By Lemma \ref{lem:avg-gJ}, $||E_{P_{n}}[g_{J}(Z,\alpha_{L,0})]||_{e} = O_{\mathbf{P}}\left( \sqrt{ \frac{\bar{g}_{L,0}^{2}}{n} + ||E[g_{J}(Z,\alpha_{L,0})]||^{2}_{e}  }    \right)$.
	
	Also, by Lemma \ref{lem:H-bound} (applied to $\alpha=\alpha_{L,0}$),
	\begin{align*}
	||G(\alpha_{L,0})[\alpha - \alpha_{L,0}]||_{e} \precsim ||\alpha - \alpha_{L,0}||_{w}.
	\end{align*}
\end{proof}

\begin{proof}[\textbf{Proof of Lemma \ref{lem:ALR-thetahat}}]
	By Lemma \ref{lem:charac-N}, $\hat{\alpha}_{L,n} - tu^{\ast}_{L,n} \in \mathcal{N}_{L,n}$ for all $t = O(l_{n} n^{-1/2})$. By the same calculations of Step 1 in the proof of lemma \ref{lem:QLR-A-rep} and definition of $\hat{\alpha}_{L,n}$, for any $t = O(l_{n} n^{-1/2})$.
	\begin{align*}
	0 \leq &  \left\{ (\varDelta(\hat{\alpha}_{L,n} + tu^{\ast}_{L,n}))^{T}H_{L}^{-1}(\varDelta(\hat{\alpha}_{L,n} + tu^{\ast}_{L,n}))   \right\} - \left\{ (\varDelta(\hat{\alpha}_{L,n}))^{T}H_{L}^{-1}(\varDelta(\hat{\alpha}_{L,n}))   \right\} + rem_{n} \\
	= & \left\{ t^{2} - 2 t (\varDelta(\hat{\alpha}_{L,n})^{T}H_{L}^{-1}(G(\alpha_{L,0})[u^{\ast}_{L,n}])  \right\} + rem_{n},
	\end{align*}
	with $rem_{n} = o_{\mathbf{P}}(n^{-1})$. Taking $t = \pm \sqrt{rem_{n}}$, it follows that
	\begin{align*}
	|(\varDelta(\hat{\alpha}_{L,n})^{T}H_{L}^{-1}(G(\alpha_{L,0})[u^{\ast}_{L,n}])| \precsim \sqrt{rem_{n}}.
	\end{align*}
	Since $\varDelta(\hat{\alpha}_{L,n})^{T}H_{L}^{-1}(G(\alpha_{L,0})[u^{\ast}_{L,n}]) = \langle u^{\ast}_{L,n} , \hat{\alpha}_{L,n} - \alpha_{L,0} \rangle_{w} + n^{-1} \sum_{i=1}^{n} (G(\alpha_{L,0})[u^{\ast}_{L,n}])^{T} H_{L}^{-1} g_{J}(Z_{i},\alpha_{L,0}) $, the desired result follows.
\end{proof}

\begin{proof}[\textbf{Proof of Lemma \ref{lem:anorm}}]
	Let $\zeta_{L,n} \equiv (G(\alpha_{L,0})[u^{\ast}_{n}])^{T}H_{L}^{-1} \{g_{J}(Z_{i},\alpha_{L,0})  - E[g_{J}(Z,\alpha_{L,0})] \} $. It is clear that $E[\zeta_{L,n}] = 0$ and $Var(\zeta_{L,n}) = 1$, so to show asymptotic normality it suffices to show that the Lyapounov condition holds.
	
	By Cauchy-Swarchz inequality and definition of $u^{\ast}_{n}$, for any $\varrho>0$,
	\begin{align*}
	E_{\mathbf{P}}\left[ \left| \zeta_{L,n} \right|^{2+\varrho}   \right] \leq & E_{\mathbf{P}}\left[ ||u^{\ast}_{n}||_{w}^{2+\varrho}  || g_{J}(Z,\alpha_{L,0})  - E[g_{J}(Z,\alpha_{L,0})]||_{e}^{2+\varrho}   \right]\\
	= & E_{\mathbf{P}}\left[ || g_{J}(Z,\alpha_{L,0})  - E[g_{J}(Z,\alpha_{L,0})]||_{e}^{2+\varrho}   \right] \\
	\precsim & E_{\mathbf{P}}\left[ || g_{J}(Z,\alpha_{L,0}) ||_{e}^{2+\varrho}   \right].
	\end{align*}
	By the calculations in the proof of Lemmas \ref{lem:gJ-bound} and \ref{lem:Lambda-charac}, $|| g_{J}(Z,\alpha_{L,0}) ||_{e} \leq \overline{\theta} + || h'_{L,0}||_{L^{\infty}(\mathbb{W},\mu)} + ||q^{J}(X)||_{e}$. Thus
	\begin{align*}
	E_{\mathbf{P}}\left[ \left| \zeta_{L,n} \right|^{2+\varrho}   \right] \precsim   (\overline{\theta} + || h'_{L,0}||_{L^{\infty}(\mathbb{W},\mu)})^{2+\varrho}  + b^{2+\varrho}_{2+\varrho,J}.
	\end{align*}	
	By Assumption \ref{ass:rates-LQA}, $ b^{2+\varrho}_{2+\varrho,J}/n = o(1)$ and $(\overline{\theta} + || h'_{L,0}||_{L^{\infty}(\mathbb{W},\mu)})^{2+\varrho}/n^{2+\varrho} = o(1)$ for some $\varrho> 0$. Thus, the Lyapounov condition holds.
\end{proof}

\end{document}